\documentclass[12pt,reqno]{amsart}

\usepackage[utf8]{inputenc}
\usepackage[english]{babel}
\usepackage[letterpaper,margin=1in]{geometry}
\usepackage[T1]{fontenc}
\usepackage{charter}
\usepackage{microtype}
\usepackage{graphicx}
\usepackage[inline,shortlabels]{enumitem}
\usepackage{comment}
\usepackage{csquotes}

\usepackage{amsmath,amssymb,amsfonts,amsthm,amscd}
\usepackage{mathtools}
\usepackage{colonequals}
\usepackage{mathrsfs}
\usepackage{bm}

\usepackage{tikz}
\usepackage{tikz-cd}
\usetikzlibrary{shapes,backgrounds,matrix,positioning,arrows.meta}

\usepackage[citation-order]{amsrefs}
\usepackage{xcolor}
\definecolor{citationblue}{HTML}{004B87}
\definecolor{referenceblue}{HTML}{63BAFF}
\definecolor{reflavender}{HTML}{b5a2cb}
\definecolor{referencepink}{HTML}{B05276}
\usepackage[colorlinks=true,linkcolor=referencepink,citecolor=citationblue,urlcolor=referenceblue,linkbordercolor=reflavender,citebordercolor=citationblue,urlbordercolor=referenceblue]{hyperref}
\AddToHook{begindocument/end}{\hypersetup{pdfborder={0 0 1}}}

\let\MolinoOriginalTableOfContents\tableofcontents
\renewcommand{\tableofcontents}{%
  \begingroup
  \setlength{\parskip}{0pt}%
  \small
  \MolinoOriginalTableOfContents
  \endgroup
}

\numberwithin{equation}{section}

\theoremstyle{plain}
\newtheorem{theorem}{Theorem}[section]
\newtheorem{cor}[theorem]{Corollary}
\newtheorem{corollary}[theorem]{Corollary}
\newtheorem{lemma}[theorem]{Lemma}
\newtheorem{proposition}[theorem]{Proposition}
\newtheorem{hypothesis}[theorem]{Hypothesis}

\theoremstyle{definition}
\newtheorem{definition}[theorem]{Definition}

\theoremstyle{remark}
\newtheorem{remark}[theorem]{Remark}

\newcommand{\cH}{\mathcal{H}}
\newcommand{\cO}{\mathcal{O}}
\newcommand{\cW}{\mathcal{W}}
\newcommand{\cG}{\mathcal{G}}
\newcommand{\cE}{\mathcal{E}}

\newcommand{\F}{\mathcal{F}}

\newcommand{\R}{\mathbb{R}}
\newcommand{\RR}{\mathbb{R}}
\newcommand{\ZZ}{\mathbb{Z}}
\newcommand{\QQ}{\mathbb{Q}}

\DeclareMathOperator{\codim}{codim}
\DeclareMathOperator{\rank}{rank}
\DeclareMathOperator{\pr}{pr}
\DeclareMathOperator{\id}{id}
\DeclareMathOperator{\Lie}{Lie}
\DeclareMathOperator{\Aut}{Aut}
\DeclareMathOperator{\Hom}{Hom}
\DeclareMathOperator{\bas}{bas}
\DeclareMathOperator{\ev}{ev}

\renewcommand{\subset}{\subseteq}

\title{A Generalization of Molino's Theory to Riemannian Groupoids}
\author{Lily Zhang}
\address{Department of Mathematics, Texas A\&M University, College Station, TX 77843, USA}
\email{lilyzhang@tamu.edu}
\date{}

\begin{document}

\begin{abstract}

Riemannian groupoids describe Riemannian foliations together with their symmetries. In this article, we extend classical Molino's theory, which concerns the structure of Riemannian foliations on compact manifolds, to the setting of regular Riemannian groupoids with compact connected object manifolds.

We observe that the orbit foliation associated with a regular Riemannian groupoid defines a Riemannian foliation on the object manifold. We show that the normal representation of a Riemannian groupoid extends naturally to an action on Molino's structures of the orbit foliation, thereby yielding the fundamental structural description of a regular Riemannian groupoid.

In addition to the main result, we clarify the essential role played by basic Lie algebroids in Molino's theory and identify the normal representation with the degree-one cohomology representation of the adjoint representation up to homotopy. We also establish comparison results under Riemannian Morita equivalence. For regular Riemannian groupoids, we give sufficient conditions for the associated basic Lie algebroids to become isomorphic after pullback to a common refinement.
\end{abstract}

\maketitle
\tableofcontents

\section{Introduction}\label{sec:introduction}

Reinhart introduced bundle-like metrics \cite{Reinhart:BundleLikeMetrics} in 1959 and singled out Riemannian foliations as those whose leaves remain locally at constant distance from one another. For a bundle-like metric, a geodesic initially perpendicular to a leaf remains perpendicular to every leaf it meets \cite{Molino:RF}. The Riemannian condition does not, however, force the leaves to be closed.

Molino gave a global description of leaf closures in terms of transverse geometry. On the transverse orthonormal frame bundle, there is a lifted foliation. The transverse canonical form and transverse Levi-Civita connection form determine a transverse parallelism for this foliation. When the manifold is compact and connected, the closures of the lifted leaves are the fibres of a smooth fibre bundle. On each fibre, the lifted foliation restricts to a Lie foliation. Its tangent distribution is the kernel of a Lie algebra-valued Maurer--Cartan form. On the original manifold, the orbit of each leaf under Molino's commuting sheaf is its closure. The commuting sheaf is a locally constant sheaf of germs of transverse Killing fields \cite{Molino:RF}*{Introduction, Thm. 5.2}.

Liu studied pairs of Riemannian foliations on the same manifold, one contained in the other \cite{Liu:MolinoTwoFoliations}. \'Alvarez L\'opez and Moreira Galicia developed a topological version of Molino's theory \cite{AlvarezLopezMoreiraGalicia:TopologicalMolino}. Dyer, Hurder, and Lukina studied Molino theory for equicontinuous matchbox manifolds \cite{DyerHurderLukina:MolinoMatchbox}. Lin and Miyamoto \cite{LinMiyamoto:RiemannianFoliationsQuasifolds} consider complete Killing foliations on connected manifolds with a complete transverse action of the structure algebra. They prove that the leaf spaces are diffeological quasifolds and that the holonomy groupoids are locally Morita equivalent to action groupoids of countable groups acting affinely on Cartesian spaces. For a Riemannian foliation on a compact manifold, Gorokhovsky and Lott restrict its holonomy groupoid to a complete transversal. They take its closure in the groupoid of isometric $1$-jets and identify the resulting orbit space with the space of leaf closures \cite{GorokhovskyLott:TransverseIndex}*{Sec. 2.2}.

del Hoyo and Fernandes \cite{MdHF-LieGpdMetrics} developed the Riemannian geometry of Lie groupoids. A $2$-metric induces a metric on the arrow manifold for which multiplication is a Riemannian submersion. It also induces a $0$-metric on the object manifold.

In this article, we consider a regular Lie groupoid with a compact connected object manifold. The connected components of its orbits form the orbit foliation \cite{Moerdijk:RegularGroupoids}, but this foliation alone does not determine the groupoid's isotropy or normal representation. We assume that the object manifold carries a $0$-metric, a Riemannian metric for which the normal representation preserves the induced inner products on the normal spaces. The orbit foliation is then Riemannian, so we apply Molino's theory without assuming a $2$-metric. We show that the normal representation lifts to the transverse orthonormal frame bundle. The maps induced by local bisections preserve the lifted foliation and its transverse canonical and connection forms. On the domains where these maps descend to the leaf closure base, they induce isomorphisms of the structural Lie algebras. These isomorphisms are compatible with the fibrewise Maurer--Cartan forms.

\subsection{Main results}

We formulate the following structure theorem.

\begin{theorem}[{Theorem \ref{thm:molino-groupoid}}]\label{thm:intro-main}
Let $\cG\rightrightarrows M$ be a regular Lie groupoid equipped with a $0$-metric, and assume that $M$ is compact and connected. Let $\mathcal O$ be its orbit foliation, with normal bundle $N\colonequals TM/T\mathcal O$ and normal representation $\lambda^N$, and put $q\colonequals \rank N$. For the induced metric on $N$, let $\pi:OF(M,\mathcal O)\rightarrow M$ be the transverse orthonormal frame bundle, with lifted foliation $\widetilde{\mathcal O}$.
\begin{enumerate}[label=\textup{(\arabic*)}]
\item The orbit foliation $(M,\mathcal O)$ is a Riemannian foliation, and $(OF(M,\mathcal O),\widetilde{\mathcal O})$ is transversely parallelizable, with transverse canonical form $\theta$ and transverse Levi-Civita connection form $\omega$.

\item The closures of the leaves of $\widetilde{\mathcal O}$ are the fibres of an $O(q)$-equivariant fibre bundle $\kappa:OF(M,\mathcal O)\rightarrow B$. On each fibre, $\widetilde{\mathcal O}$ restricts to a Lie foliation with dense holonomy group. The basic Lie algebroid $A\rightarrow B$, with anchor $\rho_A$, has isotropy Lie algebras identified by restriction with the structural Lie algebras of these Lie foliations. Under these identifications, the fibrewise Maurer--Cartan forms assemble into a unique smooth section
\[
\omega^\kappa_{\mathrm{MC}}\in\Gamma\bigl((\ker d\kappa)^*\otimes\kappa^*(\ker\rho_A)\bigr).
\]

\item The normal representation $\lambda^N$ lifts to a smooth action of $\cG$ on $OF(M,\mathcal O)$, given by $g\cdot e\colonequals \lambda^N_g\circ e$. This action commutes with the right $O(q)$-action and leaves the transverse canonical form invariant. For every local bisection $\sigma:U\rightarrow\cG$, the induced map on $\pi^{-1}(U)$ preserves $\widetilde{\mathcal O}$, $\theta$, and $\omega$. Its restriction to the union of the $\kappa$-fibres contained in $\pi^{-1}(U)$ maps fibres onto fibres. This restriction descends to a local $O(q)$-equivariant diffeomorphism of $B$. The corresponding fibre maps induce isomorphisms of the structural Lie algebras, with respect to which $\omega^\kappa_{\mathrm{MC}}$ is equivariant.
\end{enumerate}
\end{theorem}





Let $A_{\cG}\colonequals \Lie(\cG)$, with anchor $\rho_{\cG}$. The lifted action defines the action groupoid $\cH=\cG\ltimes OF(M,\mathcal O)$, whose Lie algebroid is canonically the action Lie algebroid on $\pi^*A_{\cG}$. When the arrow manifold of $\cG$ admits a Riemannian metric, we identify $\lambda^N$ with the ordinary representation on $H^1(\operatorname{Ad}(\cG))=N$ (Proposition \ref{prop:normal-adjoint-cohomology}). This representation is induced by the adjoint representation up to homotopy of Arias Abad and Crainic \cite{AriasAbadCrainic:RepUpToHomotopyGroupoids}.

Two Lie algebroids thus appear with different roles. The anchor of $A_{\cG}\rightarrow M$ has image $T\mathcal O$, and the isotropy Lie algebras of $A_{\cG}$ are those of the original groupoid (Proposition \ref{prop:LieG-recovers-orbit}). The basic Lie algebroid $A\rightarrow B$ instead comes from the lifted transversely parallelizable foliation \cite{MoerdijkMrcun:IFLG}*{Sec. 6.4}. By Lemma \ref{lem:isotropy-structural-clean}, restriction identifies the isotropy Lie algebra at each $b\in B$ with the structural Lie algebra on the fibre $L_b$. Under evaluation, its elements correspond to directions tangent to $L_b$ and transverse to $\widetilde{\mathcal O}$. Under the evaluation isomorphism $\kappa^*A\cong N(\widetilde{\mathcal O})$, $\omega^\kappa_{\mathrm{MC}}$ is the restriction to $\ker d\kappa$ of the associated Lie algebroid-valued Maurer--Cartan form \cite{MoerdijkMrcun:DevelopabilitySubalgebroids}.

We illustrate this distinction with the Kronecker flow groupoid $\RR\ltimes\mathbb T^2$, associated with an irrational linear flow. It is regular and nonproper, but admits a $2$-metric. In her thesis \cite{Zhang:MolinoThesis}, the author constructs such a metric using the gauge construction in the proof of \cite{MdHF-LieGpdMetrics}*{Prop. 4.7}. The lifted leaf closures are the two components of the transverse orthonormal frame bundle, and
\[
\ker\rho_{\cG}=0,\qquad \ker\rho_A=A\cong B\times\RR.
\]

Finally, we establish comparison results under Riemannian Morita equivalence. For a proper effective \'etale Riemannian groupoid $\cE\rightrightarrows T$, with $q\colonequals \dim T$, the lifted orbit foliation has point leaves, so its leaf closure quotient is $OF(T)$. Here we consider the quotient manifold $B_{\cE}\colonequals OF(T)/\widehat{\cE}$ by the full lifted frame groupoid. A Riemannian Morita equivalence bibundle $P$ between two such groupoids $\cE$ and $\cE'$ induces, through $OF(P)$, an $O(q)$-equivariant diffeomorphism $\Phi_P:B_{\cE}\rightarrow B_{\cE'}$ (Proposition \ref{prop:OF-P-Morita-bibundle}). We also construct a $2$-metric on $B_{\cE}\rtimes O(q)$ for which $OF(T)$ is a Riemannian Morita equivalence bibundle (Corollary \ref{cor:etale-compatible-2metric}).

In the regular case, under the metric and leaf closure hypotheses of Proposition \ref{prop:regular-morita-invariance}, the leaf closure bases admit a common refinement by surjective local diffeomorphisms. The basic Lie algebroids become isomorphic after pullback to this refinement (Lemma \ref{lem:morita-basic-algebroid-comparison}). The induced frame maps identify the transverse canonical and connection forms by pullback. In particular, consider a Hausdorff Riemannian Morita equivalence bibundle between regular groupoids with connected source fibres. If one of the object manifolds is compact and connected, the bibundle induces an $O(q)$-equivariant diffeomorphism of the leaf closure bases. The corresponding basic Lie algebroid isomorphisms are compatible with the structural Lie algebras and fibrewise Maurer--Cartan forms (Corollary \ref{cor:morita-bibundle-source-connected}).

\subsection{Organization of the article}

We divide the article into two parts. After the preliminaries in Section \ref{ch:preliminaries}, we begin the \textit{first part} by developing the normal representation in Section \ref{ch:regular}. In Section \ref{sec:molino-frame-lifted-action}, we construct the lifted action on the transverse orthonormal frame bundle. In Section \ref{sec:basic-algebroid-molino-theorem}, we develop the basic Lie algebroid, the fibrewise structural Lie algebras, and the global Maurer--Cartan section $\omega^\kappa_{\mathrm{MC}}$. We then return to the original groupoid in Section \ref{sec:groupoid-lie-algebroid}, studying $A_{\cG}=\Lie(\cG)$ and its relation to the orbit foliation. These constructions lead to Molino's structure theorem for regular Riemannian groupoids, stated and proved in Section \ref{subsec:regular-groupoid-molino-theorem}.

In the \textit{second part} (Section \ref{sec:riemannian-morita-invariance}), we study Molino's structures under Riemannian Morita equivalence. We first treat proper effective \'etale groupoids by lifting Riemannian Morita equivalence bibundles to orthonormal frame bundles. For regular groupoids, we then compare the leaf closure bases and the basic Lie algebroids.


\subsection*{Acknowledgements}
The author would like to thank her advisor, Xiang Tang, for proposing this topic and for his helpful discussions and encouragement. She is grateful for the time and effort he devoted to their weekly meetings and to commenting on and reviewing this article.

The author also thanks the organizers of Poisson 2026 for the opportunity to present this work. She is grateful for the constructive feedback from Rui Loja Fernandes, Sven Holtrop, David Miyamoto, and many other friends she met at the conference.

\subsection*{AI Disclaimer}
This article originated from the author's thesis, which was developed before the breakthrough in AI tools for advanced mathematical work. The mathematical arguments were written entirely by the author and discussed with her advisor. More recently, ChatGPT was used to check the article for errors. During this review, a few errors were identified in local computations, which were subsequently recalculated and revised accordingly. After the review, ChatGPT identified a few citation mismatches because the publication status and formatting of some referenced articles had changed. It was then used for citation verification, notation checks, and proofreading of the article. The author takes full responsibility for all the mathematical content and writing in this article.

\section{Preliminaries}\label{ch:preliminaries}

Throughout, we use the terminology of Lee \cite{Lee:ISM} for smooth manifolds and of Lee \cite{Lee:IRM} for Riemannian manifolds. For foliations and Lie groupoids, we follow Moerdijk--Mr\v cun \cite{MoerdijkMrcun:IFLG}; for regular Lie groupoids and their orbit foliations, we follow Moerdijk \cite{Moerdijk:RegularGroupoids}. The Riemannian groupoid background comes from del Hoyo--Fernandes \cite{MdHF-LieGpdMetrics}.

\subsection{Foliations and transverse geometry}

\subsubsection{Regular foliations and basic forms}

\begin{definition}[{\cite{MoerdijkMrcun:IFLG}*{Secs. 1.1--1.2}}]
A \emph{(regular) foliation} $\F$ of dimension $p$ on a manifold $M$ is an involutive rank-$p$ subbundle $T(\F)\subset TM$. Its maximal connected immersed integral manifolds are the \emph{leaves} of $\F$. The codimension of $\F$ is $q\colonequals \dim(M)-p$, and its normal bundle is $N(\F)\colonequals TM/T(\F)$.
\end{definition}

We write $L_x$ for the leaf through $x\in M$. A subset $U\subset M$ is \emph{saturated} if it is a union of leaves. In a foliated chart $U\cong \RR^p\times\RR^q$, the plaques $\RR^p\times\{y\}$ are the leaves of the restricted foliation $\F|_U$ \cite{MoerdijkMrcun:IFLG}*{Sec. 1.1}.

\begin{definition}[{\cite{Molino:RF}*{Sec. 2.3}}]
A differential form $\omega\in\Omega^k(M)$ is \emph{basic} with respect to $\F$ if $\iota_X\omega=0$ and $\iota_X(d\omega)=0$ for every vector field $X\in\Gamma(T(\F))$.
\end{definition}

\subsubsection{Riemannian foliations and Molino's structure theorem}

\begin{definition}[{\cite{MoerdijkMrcun:IFLG}*{Sec. 2.2}}]
A Riemannian metric $g$ on $M$ is \emph{bundle-like} for a foliation $\F$ if the induced metric on the normal bundle $N(\F)$ is invariant under holonomy. A foliation is called \emph{Riemannian} if it admits a bundle-like metric.
\end{definition}

The term \emph{bundle-like metric} goes back to Reinhart's 1959 paper \cite{Reinhart:BundleLikeMetrics}.

For a Riemannian foliation $(\F,g)$ of codimension $q$ on $M$, the transverse orthonormal frame bundle
\begin{equation}
\pi:OF(M,\F)\longrightarrow M
\end{equation}
carries a lifted foliation $\widetilde{\F}$. The transverse canonical form and the transverse Levi-Civita connection form on $OF(M,\F)$ determine a transverse parallelism for $(OF(M,\F),\widetilde{\F})$ \cite{MoerdijkMrcun:IFLG}*{Ex. 4.19 and Thm. 4.20}.

\begin{theorem}[Molino, {\cite{MoerdijkMrcun:IFLG}*{Thm. 4.26}}]\label{thm:mol}
Let $(\F,g)$ be a Riemannian foliation of codimension $q$ on a compact connected manifold $M$, and let $\widetilde{\F}$ be the associated lifted foliation of the transverse orthonormal frame bundle $OF(M,\F)$.
\begin{enumerate}
\item The foliated manifold $(OF(M,\F),\widetilde{\F})$ is transversely parallelizable.

\item There exists a manifold $B$ with an $O(q)$-action and an $O(q)$-equivariant fibre bundle $\kappa:OF(M,\F)\rightarrow B$ such that the fibres of $\kappa$ are exactly the closures of the leaves of $\widetilde{\F}$.

\item The Lie algebra $\mathfrak g$ of transverse vector fields of $\widetilde{\F}|_{\kappa^{-1}(b)}$ is independent, up to isomorphism, of $b\in B$. The foliation $\widetilde{\F}|_{\kappa^{-1}(b)}$ is a Lie foliation given by a canonical $\mathfrak g$-valued Maurer--Cartan form with a dense holonomy group.
\end{enumerate}
\end{theorem}

\subsection{Lie groupoids, Lie algebroids, and actions}

\subsubsection{Lie groupoids, actions, orbits, and isotropy}

\begin{definition}[{\cite{MoerdijkMrcun:IFLG}*{Sec. 5.1}}]
A \emph{Lie groupoid} $\cG\rightrightarrows M$ is a groupoid with object manifold $M$ and arrow manifold $\cG$, where $M$ is a smooth Hausdorff manifold and $\cG$ is a smooth manifold, possibly non-Hausdorff, such that the source map $s:\cG\rightarrow M$ is a smooth submersion with Hausdorff fibres, and all the other structure maps are smooth.
\end{definition}

The other structure maps are the target map $t:\cG\rightarrow M$, the unit map $u:M\rightarrow \cG$, the inverse map $i:\cG\rightarrow \cG$, and the multiplication map $m:\cG\times_M \cG\rightarrow \cG$, $(h,g)\mapsto hg$, defined on the fibre product $\cG\times_M \cG\colonequals \{(h,g)\in \cG\times \cG\mid s(h)=t(g)\}$.

\begin{definition}[{\cite{MoerdijkMrcun:IFLG}*{Sec. 5.3}}]
Let $\cG\rightrightarrows M$ be a Lie groupoid and let $\epsilon:N\rightarrow M$ be a smooth map. A \emph{left action} of $\cG$ on $N$ along $\epsilon$ is a smooth map
\begin{equation}
\begin{aligned}
\cG\times_M N\colonequals \{(g,y)\in \cG\times N\mid s(g)=\epsilon(y)\} &\longrightarrow N,\\
(g,y) &\longmapsto g\cdot y,
\end{aligned}
\end{equation}
such that
\begin{equation}\label{eq:groupoid-action-axioms}
\epsilon(g\cdot y)=t(g),\qquad 1_{\epsilon(y)}\cdot y=y,\qquad g'\cdot(g\cdot y)=(g'g)\cdot y.
\end{equation}

The associated \emph{translation groupoid} (or \emph{action groupoid}) $\cG\ltimes N \rightrightarrows N$ has arrow manifold $\cG\times_M N$, source and target maps $s(g,y)=y$ and $t(g,y)=g\cdot y$, and multiplication $(g',g\cdot y)(g,y)=(g'g,y)$.
\end{definition}

With $N=M$ and $\epsilon=\id_M$, any Lie groupoid acts canonically on its object manifold by $g\cdot x\colonequals t(g)$ whenever $s(g)=x$.

For $x\in M$, the \emph{isotropy group} at $x$ is $\cG_x\colonequals \{g\in \cG\mid s(g)=t(g)=x\}$, and the \emph{orbit} through $x$ is $O_x\colonequals t\bigl(s^{-1}(x)\bigr)\subset M$ {\cite{Moerdijk:RegularGroupoids}*{1.5}}.

\begin{definition}[{\cite{Moerdijk:RegularGroupoids}*{1.3(b)}}]
A Lie groupoid $\cG\rightrightarrows M$ is \emph{proper} if the map $(s,t):\cG\rightarrow M\times M$ is proper, that is, a closed map with compact Hausdorff fibres.
\end{definition}

Further background on Lie groupoids can be found in \cite{Li:ConstructionsLieGroupoids} and \cite{Holtrop:RiemannianGroupoids}.

\subsubsection{Lie algebroids, regularity, and normal representations}

Associated with a Lie groupoid $\cG\rightrightarrows M$ is its Lie algebroid $A_{\cG}\colonequals \ker(ds)|_M\rightarrow M$, with anchor $\rho_{\cG}\colonequals dt|_{A_{\cG}}:A_{\cG}\rightarrow TM$ {\cite{MoerdijkMrcun:IFLG}*{Sec. 6.1}}. In Section \ref{sec:groupoid-lie-algebroid}, we construct the Lie bracket on $\Gamma(A_{\cG})$ and identify the isotropy Lie algebras of $A_{\cG}$.

\begin{definition}[{\cite{Moerdijk:RegularGroupoids}*{1.3(g), 1.5}}]\label{def:orbit-regular}
A Lie groupoid $\cG\rightrightarrows M$ is \emph{regular} if its anchor $\rho_{\cG}:A_{\cG}\rightarrow TM$ has locally constant rank; equivalently, the orbit dimension is locally constant.
\end{definition}

If $\cG$ is regular, then $\rho_{\cG}(A_{\cG})\subset TM$ is a smooth involutive subbundle, and hence defines a foliation $\cO$ on $M$ whose leaves are the connected components of the orbits {\cite{Moerdijk:RegularGroupoids}*{1.5}}.

If $O\subset M$ is an orbit, the restricted groupoid $\cG|_O\rightrightarrows O$ acts on the normal bundle $\nu(O)\colonequals TM|_O/TO$ by the \emph{normal representation} {\cite{MdHF-LieGpdMetrics}*{Sec. 2.2}}.

\subsection{Riemannian groupoids}

\begin{definition}[{\cite{MdHF-LieGpdMetrics}*{Definition 2.7}}]\label{def:2.2.2}
Let $\theta:\cG\curvearrowright E$ be an action of a Lie groupoid $\cG\rightrightarrows M$ on a manifold $E$. A Riemannian metric on $E$ is \emph{transversely invariant} for $\theta$ if the normal representation of the action groupoid $\cG\ltimes E\rightrightarrows E$ acts by isometries for the induced metrics on the normal bundles of its orbits. The normal representation along an orbit is recalled in Definition \ref{def:normal-orbit}; here it is applied to the action groupoid $\cG\ltimes E$.
\end{definition}

Let $\cG\rightrightarrows M$ be a Lie groupoid. We write $\cG^{(0)}\colonequals M$ and, for $n\geq1$, denote by $\cG^{(n)}\subset\cG^n$ the manifold of chains of $n$ composable arrows:
\begin{equation}
\cG^{(n)}\colonequals \underbrace{\cG\times_M\cdots\times_M\cG}_{n\ \text{factors}}.
\end{equation}
In particular,
\begin{equation}
\cG^{(1)}=\cG,\qquad
\cG^{(2)}=\cG\times_M\cG=\{(h,g)\in\cG\times\cG\mid s(h)=t(g)\}.
\end{equation}
On $\cG^{(2)}$, the projections and multiplication are given by
\begin{equation}
\pi_1(h,g)=h,\qquad \pi_2(h,g)=g,\qquad m(h,g)=hg.
\end{equation}

\begin{definition}[{\cite{MdHF-LieGpdMetrics}*{Definitions 3.1, 3.7, 3.14}}]\label{def:n-metrics}
Let $\cG\rightrightarrows M$ be a Lie groupoid.
\begin{enumerate}
\item A \emph{$0$-metric} on $\cG\rightrightarrows M$ is a Riemannian metric $\eta^{(0)}$ on $\cG^{(0)}=M$ that is transversely invariant for the canonical action $\cG\curvearrowright M$, given by $g\cdot s(g)=t(g)$.

\item A \emph{$1$-metric} on $\cG\rightrightarrows M$ is a Riemannian metric $\eta^{(1)}$ on $\cG^{(1)}=\cG$ that is transversely left-invariant and for which the inversion $i$ is an isometry.

Here \emph{transversely left-invariant} means transversely invariant for the left-translation action of $\cG$ on itself. Equivalently, $\eta^{(1)}$ induces a metric on $M$ for which $s:\cG\rightarrow M$ is a Riemannian submersion. Since inversion is an isometry, $t$ is then a Riemannian submersion for the same metric {\cite{MdHF-LieGpdMetrics}*{Example 2.9 and Proposition 3.8(i)}}.

\item A \emph{$2$-metric} on $\cG\rightrightarrows M$ is a Riemannian metric $\eta^{(2)}$ on $\cG^{(2)}\colonequals \cG\times_M\cG$ that is transversely invariant for the action $\theta_1:\cG\curvearrowright \cG^{(2)}$, $k\cdot(h,g)\colonequals (kh,g)$, and for which the action of $S_3$ by permutation of the vertices of the corresponding commutative triangle is isometric.
\end{enumerate}
\end{definition}

Following \cite{MdHF-LieGpdMetrics}*{Definition 3.14}, we call a Lie groupoid equipped with a $2$-metric a \emph{Riemannian groupoid}. In Theorem \ref{thm:molino-groupoid} and Corollary \ref{cor:morita-bibundle-source-connected}, we only require $0$-metrics on the groupoids. When a $2$-metric is given, we use the induced $0$-metric of Proposition \ref{prop:2metric-induces}.

\begin{proposition}[{\cite{MdHF-LieGpdMetrics}*{Proposition 3.16}}]\label{prop:2metric-induces}
Let $\cG\rightrightarrows M$ be a Lie groupoid. A $2$-metric $\eta^{(2)}$ on $\cG^{(2)}$ induces a $1$-metric $\eta^{(1)}$ on $\cG^{(1)}$, and hence also a $0$-metric $\eta^{(0)}$ on $\cG^{(0)}$.
\end{proposition}

\begin{theorem}[Del Hoyo--Fernandes, {\cite{MdHF-LieGpdMetrics}*{Thm. 1}}]\label{thm:existence-2metric}
Every Hausdorff proper Lie groupoid admits a $2$-metric.
\end{theorem}

\subsubsection{Morita equivalence}

We use Morita equivalence of Lie groupoids in the sense of \cite{MoerdijkMrcun:IFLG}*{Sec. 5.4}: two Lie groupoids are \emph{Morita equivalent} if they are connected by weak equivalences through a third Lie groupoid.

\begin{definition}[{\cite{MdHFStackMetrics}*{Sec. 6.1}; \cite{PTW-Resolutions}*{Defs. 2.1--2.2 and 6.11}}]\label{def:morita-data-prelim}
Let $\cG\rightrightarrows M$ and $\cG'\rightrightarrows M'$ be Lie groupoids.

\begin{enumerate}[label=(\roman*), leftmargin=*]
\item A Lie groupoid map $\phi:\widetilde{\cG}\rightarrow \cG$ is a \emph{Morita fibration} if it is fully faithful and its object map $\phi_0:\widetilde M\rightarrow M$ is a surjective submersion.


\item A \emph{Morita equivalence bibundle} from $\cG'$ to $\cG$ is a manifold $P$ equipped with commuting left $\cG'$- and right $\cG$-actions, with moment maps $\alpha':P\rightarrow M'$ and $\alpha:P\rightarrow M$ (the right action being defined for $\alpha(p)=t(g)$, with $\alpha(p\cdot g)=s(g)$), such that the left $\cG'$-action is principal over $M$ with projection $\alpha$, and the right $\cG$-action is principal over $M'$ with projection $\alpha'$.

\item Suppose that $\cG$ and $\cG'$ are equipped with $0$-metrics, and let $(P,\alpha',\alpha)$ be a Morita equivalence bibundle as in \textup{(ii)}. If $P$ is equipped with a metric $\eta_P$ such that both moment maps $\alpha:P\rightarrow M$ and $\alpha':P\rightarrow M'$ are Riemannian submersions with respect to the given $0$-metrics on $M$ and $M'$, then $(P,\alpha',\alpha,\eta_P)$ is called a \emph{Riemannian Morita equivalence bibundle}.
\end{enumerate}
\end{definition}

In \textup{(iii)}, we use the metric condition of \cite{PTW-Resolutions}*{Definition 6.11} for groupoids equipped with $0$-metrics. For groupoids with $2$-metrics, we use their induced $0$-metrics.


\section{Regular groupoids, metrics, and normal representations}\label{ch:regular}\label{sec:regular-groupoids-normal}

We show that a $0$-metric on a regular Lie groupoid is bundle-like for its orbit foliation. We then describe the normal representation along each orbit. These representations assemble into a smooth representation on the normal bundle of the orbit foliation.

\subsection{0-metrics and the orbit foliation}
\label{subsec:regular-orbit-foliation}

Recall from Definition \ref{def:orbit-regular} that a Lie groupoid $\cG\rightrightarrows M$ is \emph{regular} if the anchor of its Lie algebroid, $\rho_{\cG}:A_{\cG}\rightarrow TM$, has locally constant rank. Equivalently, the rank of $t:s^{-1}(x)\rightarrow M$, which equals $\dim O_x$, is locally constant as a function of $x\in M$ {\cite{Moerdijk:RegularGroupoids}*{1.3(g)}}.

If this is the case, $T(\mathcal O)\colonequals \operatorname{im}(\rho_{\cG})\subset T M$ is a smooth involutive subbundle, defining a regular foliation $\mathcal O$ on $M$ whose leaves are the connected components of the orbits $O_x\colonequals t(s^{-1}(x))$.

We fix the following standing hypotheses.

\begin{hypothesis}\label{hyp:standing}
Let $\cG \rightrightarrows M$ be a regular Lie groupoid. Assume that $M$ is compact and connected, and fix a $0$-metric $\eta^{(0)}$ on $M$. Let $\mathcal O$ be the orbit foliation, and set $N\colonequals T M/T(\mathcal O)$ with $\rank N=q=\codim \mathcal O$.
\end{hypothesis}



We first discuss the metric preserved by the normal representation along each orbit.

\begin{proposition}\label{prop:0-metric}
Under Hypothesis \ref{hyp:standing}, let $g^M\colonequals \eta^{(0)}$. For each orbit $O\subset M$, write $\cG_O\colonequals \cG|_O$. Then the normal representation $\lambda:\cG_O\curvearrowright\nu(O)$ acts by fibrewise isometries for the metric on $\nu(O)$ induced by $g^M$.\qed
\end{proposition}

For regular groupoids, this transverse invariance makes the orbit foliation Riemannian.

\begin{proposition}\label{prop:RF}

Let $\cG\rightrightarrows M$ be a regular Lie groupoid equipped with a $0$-metric $\eta^{(0)}$; in particular, this holds under Hypothesis \ref{hyp:standing}. Then the metric $g^M\colonequals \eta^{(0)}$ makes the orbit foliation $\mathcal O$ into a Riemannian foliation.
\end{proposition}

\begin{proof}
Let $g^N$ be the quotient metric on $N=TM/T(\mathcal O)$ induced by $g^M$, and let $X$ be a local vector field tangent to $\mathcal O$. Since $\cG$ is regular, the anchor $\rho_{\cG}:A_{\cG}\rightarrow T(\mathcal O)$ is a surjective vector bundle map. After restricting the domain of $X$, choose a local section $a$ of $A_{\cG}$ such that $\rho_{\cG}(a)=X$. Its right invariant extension
\begin{equation}\label{eq:right-invariant-extension}
(a^r)_g=(dR_g)_{1_{t(g)}}a_{t(g)}
\end{equation}
satisfies $ds(a^r)=0$ and $dt(a^r)=X\circ t$ \cite{MoerdijkMrcun:IFLG}*{Sec. 6.1, Prop. 6.1(ii)}. Locally near each unit, the flow of $a^r$ gives smooth maps $\sigma_\tau$, defined for sufficiently small $\tau$, such that
\begin{equation}\label{eq:flow-bisection-projections}
\sigma_0(x)=1_x,\qquad
s\circ\sigma_\tau=\id,\qquad
t\circ\sigma_\tau=\varphi_\tau,
\end{equation}
where $\varphi_\tau$ is the local flow of $X$.

For $x$ in the domain of $\sigma_\tau$ and $v\in T_xM$, we have $ds_{\sigma_\tau(x)}\circ d(\sigma_\tau)_x(v)=v$. By the description of the normal representation in \cite{MdHF-LieGpdMetrics}*{Sec. 2.2}, we therefore have
\begin{equation}\label{eq:flow-normal-representation}
\lambda_{\sigma_\tau(x)}([v])
=
\bigl[dt_{\sigma_\tau(x)}
       \bigl(d(\sigma_\tau)_x(v)\bigr)\bigr]
=
[d(\varphi_\tau)_x(v)].
\end{equation}
Since $g^M$ is a $0$-metric, the normal representation acts by isometries for $g^N$. Thus the map on $N$ induced by $\varphi_\tau$ preserves $g^N$.

Since we chose $X$ arbitrarily, the symmetric form on $TM$ obtained by pulling back $g^N$ along $TM\rightarrow N$ is invariant under the local flows of all vector fields tangent to $\mathcal O$. Its kernel is $T(\mathcal O)$, so it is a transverse metric \cite{MoerdijkMrcun:IFLG}*{Sec. 2.2 and Remark 2.7(7)}. Hence $g^M$ is bundle-like for $\mathcal O$, and $(M,\mathcal O)$ is a Riemannian foliation.
%
\end{proof}

\subsection{Orbitwise and global normal representations}\label{subsec:normal-representations}

Every Lie groupoid has a normal representation along each orbit.

\begin{lemma}\label{lem:orbit-submersion}
Let $\cG \rightrightarrows M$ be a Lie groupoid and let $x \in M$. Then the restriction $t_x \colonequals t|_{s^{-1}(x)} \colon s^{-1}(x) \rightarrow O_x \colonequals t\bigl(s^{-1}(x)\bigr)$ is a surjective submersion and a principal $\cG_x$-bundle.

Consequently, for any arrow $g \colon x \rightarrow y$, we have $dt_g(\ker ds_g) = T_y(O_x)$ and $ds_g(\ker dt_g) = T_x(O_x)$.
\end{lemma}

\begin{proof}
The first part follows from \cite{MoerdijkMrcun:IFLG}*{Theorem 5.4(iv)}. For fixed $x \in M$,
\begin{equation}
t_x \colonequals t|_{s^{-1}(x)} \colon s^{-1}(x) \longrightarrow O_x
\end{equation}
is a principal $\cG_x$-bundle, hence a surjective submersion onto the orbit $O_x$.

For the second part, let $g \colon x \rightarrow y$. Since $s^{-1}(x) \subset \cG$ is a submanifold with $T_g\bigl(s^{-1}(x)\bigr) = \ker(ds_g)$, we have
\begin{equation}
d(t_x)_g = dt_g|_{\ker(ds_g)} \colon \ker(ds_g) \longrightarrow T_y(O_x),
\end{equation}
and it is surjective because $t_x$ is a submersion. Therefore
\begin{equation}
dt_g(\ker ds_g) = T_y(O_x).
\end{equation}

Using that the inversion map interchanges $s$ and $t$, we get
\begin{equation}
ds_g(\ker dt_g) = T_x(O_x).
\end{equation}
\end{proof}

\begin{definition}[{\cite{MdH:Orbispaces}*{Prop. 3.4.2}; \cite{MdHF-LieGpdMetrics}*{Sec. 2.2}}]\label{def:normal-orbit}
Let $O\subset M$ be an orbit of the Lie groupoid $\cG\rightrightarrows M$, and let $g:x\rightarrow y$ be an arrow in $\cG_O$. For $[v]\in \nu_x(O)\colonequals T_x(M)/T_x(O)$, choose a representative $v\in T_x(M)$. Since $s:\cG\rightarrow M$ is a submersion, there exists $X\in T_g(\cG)$ such that $ds_g(X)=v$. This allows us to define
\begin{equation}\label{eq:normal-representation-definition}
\lambda_g([v])\colonequals [dt_g(X)]\in \nu_y(O)\colonequals T_y(M)/T_y(O),
\end{equation}
where $[dt_g(X)]$ denotes the class of $dt_g(X)$ in the quotient. We call this the normal representation along the orbit.
\end{definition}


\begin{lemma}\label{lem:normal-well-defined}
For an arrow $g:x\rightarrow y$ in $\cG_O$,
\begin{equation}
[dt_g(X)]\in \nu_y(O)=T_y(M)/T_y(O)
\end{equation}
is independent of the choice of representative $v\in T_x(M)$ of $[v]\in \nu_x(O)=T_x(M)/T_x(O)$ and of the choice of lift $X\in T_g(\cG)$ with $ds_g(X)=v$. Hence
\begin{equation}
\lambda_g:\nu_x(O)\longrightarrow \nu_y(O),\qquad \lambda_g([v])\colonequals [dt_g(X)],
\end{equation}
is a well-defined linear map.
\end{lemma}

\begin{proof}
From Lemma \ref{lem:orbit-submersion}, we have
\begin{equation}\label{eq:dt-kerds}
dt_g\bigl(\ker ds_g\bigr)=T_y(O),
\end{equation}
\begin{equation}\label{eq:ds-kerdt}
ds_g\bigl(\ker dt_g\bigr)=T_x(O).
\end{equation}

If $X,X'\in T_g(\cG)$ satisfy $ds_g(X)=ds_g(X')=v$, then $X'-X\in\ker(ds_g)$. By \eqref{eq:dt-kerds}, $dt_g(X'-X)\in T_y(O)$, hence
\begin{equation}
[dt_g(X')]=[dt_g(X)]\in T_y(M)/T_y(O).
\end{equation}

Suppose $v,v'\in T_x(M)$ represent the same class in $\nu_x(O)=T_x(M)/T_x(O)$, so $v'-v\in T_x(O)$. By \eqref{eq:ds-kerdt}, we can choose $Z\in\ker(dt_g)$ such that $ds_g(Z)=v'-v$. If $X$ is a lift of $v$ such that $ds_g(X)=v$, let $X'\colonequals X+Z$. Then
\begin{equation}
ds_g(X')=v' \qquad\text{and}\qquad dt_g(X')=dt_g(X).
\end{equation}

Thus $\lambda_g$ is well-defined, and linearity is immediate from the construction.
\end{proof}


\begin{lemma}[{\cite{MdHF-LieGpdMetrics}*{Sec. 2.2}}]\label{lem:normal-rep-axioms}
Let $O \subset M$ be an orbit and let $\lambda \colon \cG_O \curvearrowright \nu(O)$ be the normal representation. For the unit arrow $1_x$, we have 
\begin{equation}\label{eq:normal-representation-unit}
\lambda_{1_x} = \id_{\nu_x(O)}.
\end{equation}
 For composable arrows $g \colon x \rightarrow y$ and $h \colon y \rightarrow z$ in $\cG_O$, we have 
\begin{equation}\label{eq:normal-representation-composition}
\lambda_{hg} = \lambda_h \circ \lambda_g.
\end{equation}

\end{lemma}


\begin{proposition}[{\cite{MdHF-LieGpdMetrics}*{Sec. 2.2}}]\label{prop:normal-orbit}
For each orbit $O \subset M$, $\{\lambda_g\}_{g \in \cG_O}$ defines a smooth representation of the restriction Lie groupoid $\cG_O \rightrightarrows O$ on the vector bundle $\nu(O) \rightarrow O$. In particular, for each arrow $g \colon x \rightarrow y$ in $\cG_O$, the map $\lambda_g \colon \nu_x(O) \rightarrow \nu_y(O)$ is a linear isomorphism with inverse $\lambda_{g^{-1}}$.

Moreover, under the hypotheses of Proposition \ref{prop:0-metric}, each $\lambda_g$ is an isometry for the metric on $\nu(O)$ induced by $g^M$ \cite{MdHF-LieGpdMetrics}*{Def. 3.1}.
\end{proposition}


Assume now that $\cG$ is regular, so the orbit foliation $\mathcal O$ has locally constant rank and $N\colonequals T(M)/T(\mathcal O)\rightarrow M$ is a smooth vector bundle. Its rank is $q$ under Hypothesis \ref{hyp:standing}. We now consider the \emph{global} normal representation on $N$.

\begin{proposition}\label{prop:normal-global}
The orbitwise normal representations of Proposition \ref{prop:normal-orbit} assemble into a smooth representation $\lambda^N \colon \cG \times_M N \rightarrow N$, $(g,[v_x]) \mapsto \lambda^N_g([v_x])$, such that the bundle projection $\pi_N \colon N \rightarrow M$ satisfies $\pi_N(\lambda^N(g,[v_x])) = t(g)$. Equivalently, the following diagram commutes:
\begin{equation}
\begin{tikzcd}
\cG \times_M N \arrow[r,"\lambda^N"] \arrow[d,"\pr_1"'] & N \arrow[d,"\pi_N"] \\
\cG \arrow[r,"t"'] & M.
\end{tikzcd}
\end{equation}

Under the hypotheses of Proposition \ref{prop:0-metric}, each $\lambda^N_g$ is an isometry for the induced metric $g^N$ on $N$.
\end{proposition}

\begin{proof}
For each arrow $g \colon x \rightarrow y$, let $O_x$ denote the orbit through $x$ (as well as through $y$). Because $\cG$ is regular, we have $N_x = T_xM/T_x(\mathcal O) = \nu_x(O_x)$ and $N_y = T_yM/T_y(\mathcal O) = \nu_y(O_x)$. Let $\lambda_g^N \colon N_x \rightarrow N_y$ be the orbitwise normal map of Proposition \ref{prop:normal-orbit} under these identifications.

We want to show smoothness. Consider $Q \colonequals T\cG/\ker(ds) \rightarrow \cG$. Since $s \colon \cG \rightarrow M$ is a submersion, $ds$ induces a vector bundle isomorphism
\begin{equation}
\overline{ds} \colon Q \xlongrightarrow{\ \cong\ } s^*(TM).
\end{equation}

We have used earlier that $dt_g(\ker ds_g) = T_y(\mathcal O)$ for each $g \colon x \rightarrow y$, by Lemma \ref{lem:orbit-submersion}. Hence $dt$ descends to a smooth vector bundle map
\begin{equation}
\overline{dt} \colon Q \longrightarrow t^*(N).
\end{equation}

So far we can summarize the construction as follows:
\begin{equation}
\begin{tikzcd}[column sep=large,row sep=large]
s^*TM \arrow[r,"\overline{ds}^{-1}"] \arrow[d,two heads] &
T\cG/\ker(ds) \arrow[d,"\overline{dt}"] \\
s^*N \arrow[r,dashed,"\lambda^N"] & t^*N .
\end{tikzcd}
\end{equation}

The composite $s^*(TM) \xrightarrow{\ \overline{ds}^{-1}\ } Q \xrightarrow{\ \overline{dt}\ } t^*(N)$ vanishes on $s^*(T(\mathcal O))$. Indeed, if $g:x\rightarrow y$ and $v \in T_x(\mathcal O)$, then by Lemma \ref{lem:orbit-submersion} there exists $\xi \in \ker(dt_g)$ with $ds_g(\xi)=v$. Hence the image of $(g,v)\in s^*(T(\mathcal O))$ in $t^*(N)$ is zero. Therefore the composite factors uniquely through $s^*(N) = s^*(TM)/s^*(T(\mathcal O))$, which yields a smooth vector bundle map
\begin{equation}
\lambda^N \colon s^*(N) \longrightarrow t^*(N).
\end{equation}
Its restriction to the fibre over $g$ is $\lambda_g^N$. Equivalently, this is a smooth action map $\lambda^N \colon \cG \times_M N \rightarrow N$.

The last part follows from Proposition \ref{prop:0-metric}.
\end{proof}


\section{The transverse orthonormal frame bundle and the lifted groupoid action}\label{sec:molino-frame-lifted-action}

\subsection{The transverse orthonormal frame bundle and lifted foliation}
\label{subsec:frame-bundle-lifted-foliation}

In the following definitions and lemmas, we discuss Molino's construction of the transverse orthonormal frame bundle for the Riemannian foliation $(M,\mathcal O)$, together with the ingredients used below in Theorem \ref{thm:molino-groupoid}.

\begin{definition}[{\cite{MoerdijkMrcun:IFLG}*{Ex. 4.19}}]\label{def:OF}
Let $q=\rank N$. The transverse orthonormal frame bundle of $(\mathcal O,g^N)$ is the principal $O(q)$-bundle $\pi:OF(M,\mathcal O)\rightarrow M$ whose fibre over $x\in M$ is
\begin{equation}\label{eq:transverse-frame-fibre}
OF_x\colonequals \pi^{-1}(x)
=\left\{e:\R^q\longrightarrow N_x\;\middle|\;
\begin{aligned}
&e\text{ is an orthogonal linear isomorphism},\\
&\text{where }N_x\text{ is equipped with }g^N_x
\end{aligned}
\right\}.
\end{equation}
The right action of $O(q)$ is given by 
\begin{equation}\label{eq:frame-right-action}
e\cdot A\colonequals e\circ A
\end{equation}
 for $A\in O(q)$.
\end{definition}


To apply Molino's theory to the regular Riemannian foliation $(M,\mathcal O)$, we recall the transverse orthonormal frame bundle $\pi \colon OF(M,\mathcal O) \rightarrow M$ and the associated structures: the lifted foliation $\widetilde{\mathcal O}$, the transverse canonical form $\theta$, and the transverse Levi-Civita connection $\nabla^{\mathrm{tr}}$ on $N$.

\begin{remark}[{\cite{MoerdijkMrcun:IFLG}*{Remark 2.7(2)}}]\label{rem:riemannian-cocycle}
Since $(M,\mathcal O)$ is a Riemannian foliation with respect to $g^M$ by Proposition \ref{prop:RF}, we may choose a Haefliger cocycle $(s_i \colon U_i \rightarrow T_i)$, where $T_i \subset \RR^q$ is open, together with Riemannian metrics $g_i$ on the $T_i$, such that each $s_i$ is a submersion with connected fibres, $\ker(ds_i)_x=T_x(\mathcal O)$, and the induced map $(ds_i)^N_x:(N_x,g_x^N)\rightarrow\bigl(T_{s_i(x)}T_i,(g_i)_{s_i(x)}\bigr)$ is an isometry for every $x\in U_i$.

On overlaps $U_i\cap U_j$, the transition maps $s_{ij}:s_j(U_i\cap U_j)\rightarrow s_i(U_i\cap U_j)$ satisfy $s_i=s_{ij}\circ s_j$ and are local isometries.
\end{remark}

\begin{definition}[{\cite{MoerdijkMrcun:IFLG}*{Ex. 4.19}}]\label{def:lifted-foliation}
Choose a Haefliger cocycle $(s_i \colon U_i \rightarrow T_i)$ and metrics $g_i$ as in Remark \ref{rem:riemannian-cocycle}. Let $\pi_{T_i}:OF(T_i)\rightarrow T_i$ be the orthonormal frame bundle projection of the Riemannian manifold $(T_i,g_i)$.

For $x \in U_i$, $(ds_i)_x \colon T_xM \rightarrow T_{s_i(x)}T_i$ has kernel $T_x(\mathcal O)$, hence induces a linear isomorphism
\begin{equation}
(ds_i)^N_x:N_x=T_xM/T_x(\mathcal O)\longrightarrow T_{s_i(x)}T_i.
\end{equation}

Let $\pi \colon OF(M,\mathcal O)\rightarrow M$ be the transverse orthonormal frame bundle and define
\begin{equation}\label{eq:lifted-submersion}
\widetilde s_i:\pi^{-1}(U_i)\longrightarrow OF(T_i), \qquad \widetilde s_i(e)\colonequals (ds_i)^N_{\pi(e)}\circ e.
\end{equation}
For $f \in OF(T_i)$, the restriction of $\pi$ identifies the fibre $\widetilde s_i^{-1}(f)$ diffeomorphically with the fibre $s_i^{-1}(\pi_{T_i}(f))$. Since $\widetilde s_i$ covers the submersion $s_i$ and maps each fibre of $\pi$ diffeomorphically onto the corresponding fibre of $\pi_{T_i}$, it is a submersion with connected fibres.

On overlaps $U_i\cap U_j$, define
\begin{equation}\label{eq:lifted-transition-map}
\widetilde s_{ij}:OF(T_j)|_{s_j(U_i\cap U_j)} \longrightarrow OF(T_i)|_{s_i(U_i\cap U_j)}, \qquad \widetilde s_{ij}(f)\colonequals (ds_{ij})_y\circ f,
\end{equation}
for $f\in OF(T_j)_y$. On $\pi^{-1}(U_i\cap U_j)$, we have
\begin{equation}\label{eq:lifted-cocycle-compatibility}
\widetilde s_i=\widetilde s_{ij}\circ \widetilde s_j.
\end{equation}
Therefore the fibres of the $\widetilde s_i$ glue to a global foliation $\widetilde{\mathcal O}$ on $OF(M,\mathcal O)$, which we call the lifted orbit foliation.
\end{definition}

Thus $\pi:(OF(M,\mathcal O),\widetilde{\mathcal O})\rightarrow (M,\mathcal O)$ is a transverse principal $O(q)$-bundle.


\begin{remark}
For each $i$, letting $\pi_{T_i}:OF(T_i)\rightarrow T_i$ be the orthonormal frame bundle projection, the local defining submersions of the lifted foliation fit into the commutative diagram
\begin{equation}
\begin{tikzcd}
\pi^{-1}(U_i) \arrow[r, "\widetilde s_i"] \arrow[d, "\pi"'] &
OF(T_i) \arrow[d, "\pi_{T_i}"] \\
U_i \arrow[r, "s_i"'] & T_i .
\end{tikzcd}
\end{equation}

On overlaps $U_i \cap U_j$,
\begin{equation}
\begin{tikzcd}
\pi^{-1}(U_i \cap U_j)
\arrow[r, "\widetilde s_j"]
\arrow[dr, "\widetilde s_i"'] &
OF(T_j)|_{s_j(U_i \cap U_j)}
\arrow[d, "\widetilde s_{ij}"] \\
&
OF(T_i)|_{s_i(U_i \cap U_j)} .
\end{tikzcd}
\end{equation}
\end{remark}


\begin{definition}[{\cite{MoerdijkMrcun:IFLG}*{Ex. 4.19}; \cite{Molino:RF}*{Sec. 2.4; Prop. 2.6; Sec. 3.3}}]\label{def:theta-omega}
Let $\pi:OF(M,\mathcal O)\rightarrow M$ be the transverse orthonormal frame bundle, $\widetilde{\mathcal O}$ be the lifted orbit foliation, and $\pr^N:TM\rightarrow N\colonequals TM/T(\mathcal O)$ be the quotient projection. We define the following.

\begin{enumerate}[label=\textup{(\roman*)}]
\item The \emph{transverse canonical form} is the $1$-form $\theta\in\Omega^1(OF(M,\mathcal O),\RR^q)$ defined by
\begin{equation}\label{eq:transverse-canonical-form}
\theta_e(\xi) \colonequals e^{-1} \left(\pr^N_{\pi(e)}\bigl((d\pi)_e(\xi)\bigr)\right), \qquad e\in OF(M,\mathcal O),\ \xi\in T_eOF(M,\mathcal O).
\end{equation}
It is $O(q)$-equivariant, $R_A^*\theta=A^{-1}\theta$ for $A\in O(q)$, and satisfies 
\begin{equation}\label{eq:transverse-canonical-kernel}
\ker(\theta_e)=\ker\bigl((d\pi)_e\bigr)\oplus T_e(\widetilde{\mathcal O}).
\end{equation}

\item For each $i$, let $(ds_i)^N:N|_{U_i}\rightarrow s_i^*(TT_i)$ be the induced isometric bundle isomorphism, and let $\nabla^i$ be the ordinary Levi-Civita connection of the Riemannian manifold $(T_i,g_i)$. Define a connection on $N|_{U_i}$ by
\begin{equation}\label{eq:transverse-local-connection}
\nabla^{\mathrm{tr},i} \colonequals \bigl((ds_i)^N\bigr)^{-1}\circ s_i^*\nabla^i\circ (ds_i)^N.
\end{equation}
The $\nabla^{\mathrm{tr},i}$ can be glued to a global metric connection $\nabla^{\mathrm{tr}}$ on $N$, which we call the \emph{transverse Levi-Civita connection}.

\item Let $\varepsilon_1,\dots,\varepsilon_q$ be the standard basis of $\RR^q$. The tautological orthonormal frame $E_1,\dots,E_q$ of $\pi^*N$ is defined by $E_a(e)\colonequals e(\varepsilon_a)$, $e\in OF(M,\mathcal O)$. The \emph{transverse Levi-Civita connection form} is the unique $1$-form $\omega\in\Omega^1(OF(M,\mathcal O),\mathfrak o(q))$ characterized by
\begin{equation}\label{eq:transverse-connection-form}
(\pi^*\nabla^{\mathrm{tr}})_\xi E_a = \sum_{b=1}^q \omega_{ba}(\xi)\,E_b, \qquad \xi\in T_eOF(M,\mathcal O).
\end{equation}

Since $\nabla^{\mathrm{tr}}$ is metric and the frame $E_1,\dots,E_q$ is orthonormal, the matrix $(\omega_{ba})$ is skew-symmetric.
\end{enumerate}
\end{definition}

\begin{proposition}\label{prop:transverse-LC-local}
The local connections $\nabla^{\mathrm{tr},i}$ are compatible on overlaps and determine a global metric connection $\nabla^{\mathrm{tr}}$ on $N$, independent of the chosen Haefliger cocycle. If $\pi_{T_i} \colon OF(T_i)\rightarrow T_i$ denotes the orthonormal frame bundle projection of $(T_i,g_i)$ and $\omega_i$ denotes its ordinary Levi-Civita connection form, then 
\begin{equation}\label{eq:transverse-connection-local}
\omega=\widetilde s_i^*\omega_i
\end{equation}
 on $\pi^{-1}(U_i)$.

In particular, for every vector field $X$ tangent to $\widetilde{\mathcal O}$, we have $i_X\omega=0$ and $L_X\omega=0$.
\end{proposition}

\begin{proof}
By Remark \ref{rem:riemannian-cocycle}, on $U_i\cap U_j$, the transition map $s_{ij}:s_j(U_i\cap U_j)\rightarrow s_i(U_i\cap U_j)$ is a local isometry. A local isometry restricts near each point to an isometry onto an open subset. By naturality of the Levi-Civita connection \cite{Lee:IRM}*{Proposition 5.13}, $s_{ij}^{*}\nabla^i=\nabla^j$ on $s_j(U_i\cap U_j)$. Because $(ds_i)^N=(ds_{ij})\circ (ds_j)^N$ on $U_i\cap U_j$, the pullback connections agree:
\begin{equation}
\nabla^{\mathrm{tr},i}=\nabla^{\mathrm{tr},j}
\qquad\text{on }U_i\cap U_j.
\end{equation}
Therefore the $\nabla^{\mathrm{tr},i}$ glue to a global metric connection $\nabla^{\mathrm{tr}}$ on $N$.

To prove independence of the chosen cocycle, let $(s'_a:U'_a\rightarrow T'_a)$ be another Haefliger cocycle with metrics as in Remark \ref{rem:riemannian-cocycle}, and let $\nabla'_a$ be the corresponding Levi-Civita connection on $TT'_a$. For $x\in U_i\cap U'_a$, after shrinking if necessary, there is an open neighbourhood $W\subset U_i\cap U'_a$ of $x$ and a local diffeomorphism $h_{ia}:s'_a(W)\rightarrow s_i(W)$ such that $s_i|_W=h_{ia}\circ s'_a|_W$. Since both $(ds_i)^N$ and $(ds'_a)^N$ are fibrewise isometries, $h_{ia}$ is a local isometry. Applying \cite{Lee:IRM}*{Proposition 5.13} again, we obtain $h_{ia}^{*}\nabla^i=\nabla'_a$. Hence the corresponding local pullback connections agree on $W$, so the global connection $\nabla^{\mathrm{tr}}$ is independent of the cocycle.

Pulling back $(ds_i)^N$ along $\pi$, we identify $\pi^*N|_{\pi^{-1}(U_i)}$ with $\widetilde s_i^*(\pi_{T_i}^*TT_i)$, and the tautological frame $E_1,\dots,E_q$ on $\pi^*N$ corresponds to the pullback of the tautological frame $E^i_1,\dots,E^i_q$ on $\pi_{T_i}^*TT_i$. Hence, the pullback connection $\pi^*\nabla^{\mathrm{tr}}$ corresponds to $\widetilde s_i^*(\pi_{T_i}^*\nabla^i)$, and by \eqref{eq:transverse-connection-form}, $\omega|_{\pi^{-1}(U_i)}=(\widetilde s_i)^*\omega_i$.

Finally, on $\pi^{-1}(U_i)$, the foliation $\widetilde{\mathcal O}$ is the foliation by fibres of $\widetilde s_i$. Hence, for every $X\in\mathfrak X(\widetilde{\mathcal O})$, we have $i_X\omega=i_X(\widetilde s_i)^*\omega_i=0$ and $i_X(d\omega)=i_X(\widetilde s_i)^*(d\omega_i)=0$, so, by Cartan's formula,
\begin{equation}
L_X\omega=d(i_X\omega)+i_X(d\omega)=0.
\end{equation}
Thus, for every vector field $X$ tangent to $\widetilde{\mathcal O}$, we have $i_X\omega=0$ and $L_X\omega=0$.
\end{proof}


\begin{lemma}[{\cite{MoerdijkMrcun:IFLG}*{Theorem 4.20}}]\label{lem:theta-omega-parallelism}
Let $\pi:OF(M,\mathcal O)\rightarrow M$ be the transverse orthonormal frame bundle of a Riemannian foliation $(M,\mathcal O)$ of codimension $q$, with lifted foliation $\widetilde{\mathcal O}$, transverse canonical form $\theta$, and transverse Levi-Civita connection form $\omega$.

Then $(\theta,\omega)$ induces a vector bundle isomorphism
\begin{equation}\label{eq:theta-omega-trivialization}
\begin{aligned}
\overline{\Psi}: N(\widetilde{\mathcal O}) &\longrightarrow OF(M,\mathcal O)\times\bigl(\RR^q\oplus\mathfrak o(q)\bigr),\\
[\xi] &\longmapsto \bigl(e,(\theta_e(\xi),\omega_e(\xi))\bigr),
\end{aligned}
\end{equation}
where $N(\widetilde{\mathcal O})=TOF(M,\mathcal O)/T(\widetilde{\mathcal O})$ and $\xi\in T_eOF(M,\mathcal O)$. In particular, $(OF(M,\mathcal O),\widetilde{\mathcal O})$ is transversely parallelizable.
\end{lemma}

\begin{proof}
For each $e\in OF(M,\mathcal O)$, define
\begin{equation}
\begin{aligned}
\Psi_e:T_eOF(M,\mathcal O)/T_e(\widetilde{\mathcal O}) &\longrightarrow \RR^q\oplus\mathfrak o(q),\\
[\xi] &\longmapsto \bigl(\theta_e(\xi),\omega_e(\xi)\bigr).
\end{aligned}
\end{equation}

First, by Definition \ref{def:theta-omega}\textup{(i)} we have $\theta|_{T(\widetilde{\mathcal O})}=0$, and by Proposition \ref{prop:transverse-LC-local} we also have $\omega|_{T(\widetilde{\mathcal O})}=0$. Hence $\Psi_e$ is well-defined.

Recall that $\ker(\theta_e)=\ker\bigl((d\pi)_e\bigr)\oplus T_e(\widetilde{\mathcal O})$. So let $[\xi]\in T_eOF(M,\mathcal O)/T_e(\widetilde{\mathcal O})$ satisfy $\Psi_e([\xi])=0$. Choose $\xi\in T_eOF(M,\mathcal O)$ with $\theta_e(\xi)=0$ and $\omega_e(\xi)=0$. Since $\theta_e(\xi)=0$ and $\ker(\theta_e)=\ker\bigl((d\pi)_e\bigr)\oplus T_e(\widetilde{\mathcal O})$, we can write $\xi=v+\xi_{\widetilde{\mathcal O}}$, with $v\in\ker\bigl((d\pi)_e\bigr)$ and $\xi_{\widetilde{\mathcal O}}\in T_e(\widetilde{\mathcal O})$.

Because $\omega$ vanishes on $T(\widetilde{\mathcal O})$, we have $0=\omega_e(\xi)=\omega_e(v)+\omega_e(\xi_{\widetilde{\mathcal O}}) =\omega_e(v)$. Since $\omega$ is a connection form on the principal $O(q)$-bundle $\pi:OF(M,\mathcal O)\rightarrow M$, its restriction $\omega_e:\ker\bigl((d\pi)_e\bigr)\rightarrow \mathfrak o(q)$ to the vertical tangent space is a linear isomorphism. Hence $v=0$, so $\xi\in T_e(\widetilde{\mathcal O})$, and therefore $[\xi]=0$. Thus $\Psi_e$ is injective.

Finally, $\dim N_e(\widetilde{\mathcal O}) = q+\dim\mathfrak o(q) = \dim\bigl(\RR^q\oplus\mathfrak o(q)\bigr)$, so $\Psi_e$ is an isomorphism. Since the construction is smooth in $e$, the maps $\Psi_e$ assemble to a smooth vector bundle isomorphism $\overline{\Psi}$.

By Definition \ref{def:theta-omega} and Proposition \ref{prop:transverse-LC-local}, $\theta$ and $\omega$ are locally the pullbacks along $\widetilde s_i$ of the canonical form and the Levi-Civita connection form on $OF(T_i)$. Consequently, the inverse images under $\overline{\Psi}$ of constant sections of this trivial bundle project along $\widetilde s_i$ to vector fields on $OF(T_i)$. They are therefore transverse vector fields, and the sections corresponding to a basis of $\RR^q\oplus\mathfrak o(q)$ give a transverse parallelism.
\end{proof}

\begin{definition}\label{def:local-transverse-isometry}
A diffeomorphism $\varphi:U\xrightarrow{\cong}V$ between open subsets $U,V\subset M$ is called a foliated transverse isometry if:
\begin{enumerate}[label=\textup{(\roman*)}]
\item $\varphi$ is foliated, i.e.\ it sends leaves of $\mathcal O|_U$ to leaves of $\mathcal O|_V$. Equivalently, $d\varphi_x\bigl(T_x(\mathcal O)\bigr)=T_{\varphi(x)}(\mathcal O)$ for all $x\in U$. In this case $\varphi$ induces a well-defined linear map
\begin{equation}\label{eq:normal-differential}
\begin{aligned}
(d\varphi)^N_x:\ N_x=T_x(M)/T_x(\mathcal O) &\longrightarrow N_{\varphi(x)}=T_{\varphi(x)}(M)/T_{\varphi(x)}(\mathcal O),\\
(d\varphi)^N_x([v])&\colonequals [d\varphi_x(v)].
\end{aligned}
\end{equation}
\item For every $x\in U$, the induced map $(d\varphi)^N_x$ is an isometry for $g^N$. We write
\begin{equation}\label{eq:transverse-isometry}
g^N_{\varphi(x)}\bigl((d\varphi)^N_x(u),(d\varphi)^N_x(v)\bigr)=g^N_x(u,v), \qquad u,v\in N_x.
\end{equation}
\end{enumerate}
\end{definition}

We summarize the naturality of the transverse orthonormal frame bundle, its lifted foliation, and the forms $\theta$ and $\omega$ under foliated transverse isometries.

\begin{lemma}\label{lem:functoriality-molino}
Let $\pi:OF(M,\mathcal O)\rightarrow M$ be the transverse orthonormal frame bundle, with lifted foliation $\widetilde{\mathcal O}$, transverse canonical form $\theta$, transverse Levi-Civita connection $\nabla^{\mathrm{tr}}$ on $N$, and associated principal connection form $\omega$ on $OF(M,\mathcal O)$. If $\varphi:U\rightarrow V$ is a foliated transverse isometry, then
\begin{equation}
OF(\varphi):OF(M,\mathcal O)|_U=\pi^{-1}(U)\longrightarrow OF(M,\mathcal O)|_V=\pi^{-1}(V),
\end{equation}
defined by 
\begin{equation}\label{eq:foliated-frame-map}
OF(\varphi)(e)\colonequals (d\varphi)^N_{\pi(e)}\circ e,
\end{equation}
 is a principal $O(q)$-bundle isomorphism such that 
\begin{equation}\label{eq:foliated-frame-projection}
\pi\circ OF(\varphi)=\varphi\circ\pi,
\end{equation}
 and
\begin{equation}\label{eq:frame-lift-naturality}
\begin{aligned}
OF(\varphi)^*\bigl(\theta|_{\pi^{-1}(V)}\bigr)&=\theta|_{\pi^{-1}(U)},\\
OF(\varphi)^*\bigl(\omega|_{\pi^{-1}(V)}\bigr)&=\omega|_{\pi^{-1}(U)},\\
dOF(\varphi)\bigl(T(\widetilde{\mathcal O})|_{\pi^{-1}(U)}\bigr)&=T(\widetilde{\mathcal O})|_{\pi^{-1}(V)}.
\end{aligned}
\end{equation}
\end{lemma}

\begin{proof}
For each $x \in U$, the induced linear map $(d\varphi)^N_x \colon N_x \rightarrow N_{\varphi(x)}$ is an isometry by Definition \ref{def:local-transverse-isometry}(ii). Thus, if $e \colon \mathbb R^q \rightarrow N_x$ is an orthonormal frame, then
\begin{equation}
(d\varphi)^N_x \circ e \colon \mathbb R^q \longrightarrow N_{\varphi(x)}
\end{equation}
is again an orthonormal frame. Therefore $OF(\varphi)$ is well-defined.

It is $O(q)$-equivariant by construction because
\begin{equation}
\begin{aligned}
OF(\varphi)(e\cdot A)
&=(d\varphi)^N_{\pi(e)}\circ (e\circ A)\\
&=\bigl((d\varphi)^N_{\pi(e)}\circ e\bigr)\circ A\\
&=OF(\varphi)(e)\cdot A.
\end{aligned}
\end{equation}
Since $\varphi$ is a diffeomorphism, the same construction applies to $\varphi^{-1}$; hence $OF(\varphi)$ is a principal $O(q)$-bundle isomorphism.

Let $e\in OF(M,\mathcal O)|_U$ and $\xi\in T_eOF(M,\mathcal O)$, and write $x\colonequals \pi(e)$. Using $\pi\circ OF(\varphi)=\varphi\circ\pi$, the definition of $\theta$, and the identity $\pr^N_{\varphi(x)}\circ d\varphi_x = (d\varphi)^N_x\circ \pr^N_x$ (which clearly holds because $\varphi$ is foliated), we compute
\begin{equation}
\begin{aligned}
\bigl(OF(\varphi)^*\theta\bigr)_e(\xi)
&=\theta_{OF(\varphi)(e)}\bigl(dOF(\varphi)_e(\xi)\bigr)\\
&=OF(\varphi)(e)^{-1}\left(\pr^N_{\varphi(x)}\Bigl(d\pi_{OF(\varphi)(e)}\bigl(dOF(\varphi)_e(\xi)\bigr)\Bigr)\right)\\
&=OF(\varphi)(e)^{-1}\left(\pr^N_{\varphi(x)}\Bigl(d(\pi\circ OF(\varphi))_e(\xi)\Bigr)\right)\\
&=OF(\varphi)(e)^{-1}\left((d\varphi)^N_x\Bigl(\pr^N_x\bigl(d\pi_e(\xi)\bigr)\Bigr)\right)\\
&=e^{-1}\left(\pr^N_x\bigl(d\pi_e(\xi)\bigr)\right)\\
&=\theta_e(\xi),
\end{aligned}
\end{equation}
since $OF(\varphi)(e)=(d\varphi)^N_x\circ e$. Thus,
\begin{equation}
OF(\varphi)^*\bigl(\theta|_{\pi^{-1}(V)}\bigr)=\theta|_{\pi^{-1}(U)}.
\end{equation}

Then we choose a Haefliger cocycle $(s_i \colon U_i \rightarrow T_i)$ and metrics $g_i$ as in Remark \ref{rem:riemannian-cocycle}. For each $i$, the lifted foliation $\widetilde{\mathcal O}$ on $\pi^{-1}(U_i)$ is the foliation by fibres of the submersion $\widetilde s_i \colon \pi^{-1}(U_i) \rightarrow OF(T_i)$. After refining the source charts over $U$, we may assume that each restriction $s_i|_{U\cap U_i}$ has connected fibres and that $\varphi(U\cap U_i)\subset V\cap U_j$ for some $j$. Since $\varphi$ is foliated, there is then a local diffeomorphism $h_{ji} \colon s_i(U\cap U_i) \rightarrow s_j(V\cap U_j)$ such that 
\begin{equation}\label{eq:transverse-chart-intertwining}
s_j\circ\varphi = h_{ji}\circ s_i
\end{equation}
 on $U\cap U_i$.

Because $\varphi$ is a transverse isometry, each $h_{ji}$ is a local isometry between open subsets of the Riemannian manifolds $(T_i,g_i)$ and $(T_j,g_j)$. Then
\begin{equation}\label{eq:lifted-chart-intertwining}
\widetilde s_j\circ OF(\varphi)=OF(h_{ji})\circ \widetilde s_i \quad\text{on }\pi^{-1}(U\cap U_i).
\end{equation}
Hence $OF(\varphi)$ sends fibres of $\widetilde s_i$ into fibres of $\widetilde s_j$, so it preserves the lifted foliation $\widetilde{\mathcal O}$.

We can write \eqref{eq:lifted-chart-intertwining} as the following commutative diagram.
\begin{equation}
\begin{tikzcd}
\pi^{-1}(U\cap U_i) \arrow[r, "OF(\varphi)"] \arrow[d, "\widetilde s_i"'] &
\pi^{-1}(V\cap U_j) \arrow[d, "\widetilde s_j"] \\
OF(T_i)|_{s_i(U\cap U_i)}\arrow[r, "OF(h_{ji})"'] & OF(T_j)|_{s_j(V\cap U_j)}
\end{tikzcd}
\end{equation}

Let $\omega_i$ and $\omega_j$ be the ordinary Levi-Civita connection forms on $OF(T_i)$ and $OF(T_j)$. By Proposition \ref{prop:transverse-LC-local},
\begin{equation}
\omega|_{\pi^{-1}(U_i)}=(\widetilde s_i)^*\omega_i, \qquad \omega|_{\pi^{-1}(U_j)}=(\widetilde s_j)^*\omega_j.
\end{equation}
Since $h_{ji}$ is a local isometry, it preserves the Levi-Civita connection \cite{Lee:IRM}*{Proposition 5.13}. Equivalently, on orthonormal frame bundles the induced map preserves the Levi-Civita connection forms: $OF(h_{ji})^*\omega_j=\omega_i$.

Therefore, on $\pi^{-1}(U\cap U_i)$, it is easy to compute that
\begin{equation}
\begin{aligned}
OF(\varphi)^*\omega
&=OF(\varphi)^*(\widetilde s_j)^*\omega_j\\
&=(\widetilde s_i)^*OF(h_{ji})^*\omega_j\\
&=(\widetilde s_i)^*\omega_i\\
&=\omega.
\end{aligned}
\end{equation}
Hence
\begin{equation}
OF(\varphi)^*\bigl(\omega|_{\pi^{-1}(V)}\bigr)=\omega|_{\pi^{-1}(U)}.
\end{equation}

Since $OF(\varphi)$ preserves $\widetilde{\mathcal O}$, we also have
\begin{equation}
dOF(\varphi)\bigl(T(\widetilde{\mathcal O})|_{\pi^{-1}(U)}\bigr) = T(\widetilde{\mathcal O})|_{\pi^{-1}(V)}.
\end{equation}
\end{proof}

\subsection{Lifted groupoid action and invariance}

We now use the global normal representation (Proposition \ref{prop:normal-global}) to lift the action of $\cG$ on $M$ to an action on the transverse orthonormal frame bundle $\pi:OF(M,\mathcal O)\rightarrow M$.

\begin{proposition}\label{prop:action-axioms}
Let $\lambda^N:\cG\curvearrowright N$ be the global normal representation on $N=T(M)/T(\mathcal O)$, and assume that it acts by fibrewise isometries for $g^N$. For an arrow $g\in \cG$ with $s(g)=x$ and $t(g)=y$, and for a frame $e\in OF_x\colonequals \pi^{-1}(x)$, set
\begin{equation}\label{eq:lifted-frame-action}
g\cdot e\colonequals \lambda^N_g\circ e\in OF_y\colonequals \pi^{-1}(y).
\end{equation}
Then $(g,e)\mapsto g\cdot e$ is a smooth left action of $\cG$ on the principal $O(q)$-bundle $\pi:OF(M,\mathcal O)\rightarrow M$, defined on $\cG\times_M OF(M,\mathcal O)\colonequals \{(g,e)\in \cG\times OF(M,\mathcal O)\mid s(g)=\pi(e)\}$. In particular, for all composable arrows $g,h$ and all frames $e$ with $\pi(e)=s(g)=x$, we have $h\cdot(g\cdot e)=(hg)\cdot e$ and $1_x\cdot e=e$, and the action commutes with the right $O(q)$-action: 
\begin{equation}\label{eq:lifted-action-Oq-equivariance}
g\cdot(eA)=(g\cdot e)A
\end{equation}
 for $A\in O(q)$.

Moreover, this action is the unique lift of $\lambda^N$ in the following sense: if $\cG$ acts on $OF(M,\mathcal O)\rightarrow M$ by principal $O(q)$-bundle isomorphisms over the action on $M$, that is, by a smooth $O(q)$-equivariant action map $\cG\times_M OF(M,\mathcal O)\rightarrow OF(M,\mathcal O)$, $(g,e)\mapsto g\cdot e$, and the induced action on the associated bundle $OF(M,\mathcal O)\times_{O(q)}\RR^q\cong N$ equals $\lambda^N$, then $g\cdot e=\lambda^N_g\circ e$ for every arrow $g:x\rightarrow y$ and every frame $e\in OF_x$.
\end{proposition}

\begin{proof}
Well-definedness is clear. Since each $\lambda^N_g:N_x\rightarrow N_y$ is a linear isomorphism, $\lambda^N_g\circ e$ is a linear isomorphism $\R^q\rightarrow N_y$. Moreover, $\lambda^N_g$ is an isometry with respect to the metric $g^N$, hence $\lambda^N_g\circ e$ is again an orthonormal frame, so $g\cdot e\in OF_y$.

We now show that $\mu:\cG\times_M OF(M,\mathcal O)\rightarrow OF(M,\mathcal O)$, where $(g,e)\mapsto \lambda^N_g\circ e$, is smooth. Let $U,V\subset M$ be open sets and choose smooth local sections $b_U:U\rightarrow OF(M,\mathcal O)$ and $b_V:V\rightarrow OF(M,\mathcal O)$. They give local trivializations of the principal $O(q)$-bundle $\pi:OF(M,\mathcal O)\rightarrow M$ such that $\psi_U:U\times O(q)\xrightarrow{\ \cong\ }\pi^{-1}(U)$, $\psi_U(x,A)=b_U(x)\cdot A=b_U(x)\circ A$, and similarly $\psi_V:V\times O(q)\xrightarrow{\ \cong\ }\pi^{-1}(V)$.

Consider the open set $\cG_U^V\colonequals \{ g\in \cG \mid s(g)\in U,\ t(g)\in V \}\subset \cG$. Define a map $L:\cG_U^V\rightarrow O(q)$ by $L(g)\colonequals b_V \bigl(t(g)\bigr)^{-1}\circ \lambda^N_g\circ b_U \bigl(s(g)\bigr)\in O(q)$. This is well defined because $b_U(s(g))$ and $b_V(t(g))$ are orthonormal frames and $\lambda^N_g$ is an isometry, hence the above composite is an orthogonal linear map $\R^q\rightarrow\R^q$. Moreover, by Proposition \ref{prop:normal-global}, $L$ is smooth since $\lambda^N:\cG\times_M N\rightarrow N$ is smooth and $b_U,b_V$ are smooth.

Now take $g\in \cG_U^V$ and write $x\colonequals s(g)$. For $A\in O(q)$ we have
\begin{equation}\label{eq:lifted-action-local-expression}
\begin{aligned}
\mu\bigl(g,\psi_U(x,A)\bigr)
&=\lambda^N_g\circ (b_U(x)\circ A)\\
&=\bigl(\lambda^N_g\circ b_U(x)\bigr)\circ A\\
&=\bigl(b_V(t(g))\circ L(g)\bigr)\circ A\\
&=\psi_V\bigl(t(g), L(g)A\bigr).
\end{aligned}
\end{equation}
Therefore, in the local trivializations given by $\psi_U$ and $\psi_V$, the action map is $(g,(x,A))\mapsto \bigl(t(g),L(g)A\bigr)$, which is smooth. Since such trivializations cover $M$, it follows that $\mu$ is smooth globally.

The remaining action axioms follow from Proposition \ref{prop:normal-global}, together with the definition of the right action $e \cdot A = e \circ A$ on $OF(M,\mathcal O)$.

For the uniqueness claim, let $\Xi:OF(M,\mathcal O)\times_{O(q)}\RR^q\rightarrow N$ be the canonical identification. That is, $\Xi([e,u])=e(u)\in N_{\pi(e)}$, where the equivalence relation is $(eA,u)\sim(e,Au)$. For $u\in\RR^q$, by definition of the induced action on the associated bundle we have $g\cdot [e,u]=[g\cdot e,u]$. Under $\Xi$, we have $(g\cdot e)(u)=\lambda^N_g(e(u))$. Since this holds for all $u\in\RR^q$, it follows that $g\cdot e=\lambda^N_g\circ e$.
\end{proof}

Following Proposition \ref{prop:action-axioms}, we obtain the left action groupoid
\begin{equation}
\cH\colonequals \cG\ltimes OF(M,\mathcal O)\rightrightarrows OF(M,\mathcal O),
\end{equation}
with manifold of arrows
\begin{equation}\label{eq:lifted-action-groupoid-arrows}
\cH_1=\cG\times_M OF(M,\mathcal O)=\{(g,e)\in \cG\times OF(M,\mathcal O)\mid s(g)=\pi(e)\}.
\end{equation}
The structure maps are 
\begin{equation}\label{eq:lifted-action-groupoid-source}
s_{\cH}(g,e)=e, \quad t_{\cH}(g,e)=g\cdot e.
\end{equation}
Composition is given by 
\begin{equation}\label{eq:lifted-action-groupoid-multiplication}
(h,g\cdot e)\circ(g,e)=(hg,e),
\end{equation}
 units by 
\begin{equation}\label{eq:lifted-action-groupoid-unit}
u_{\cH}(e)=(1_{\pi(e)},e),
\end{equation}
 and inversion by 
\begin{equation}\label{eq:lifted-action-groupoid-inverse}
(g,e)^{-1}=(g^{-1},g\cdot e).
\end{equation}

\begin{proposition}\label{prop:lifted-action-groupoid}
The lifted action defines a Lie groupoid
\begin{equation}
\cH\colonequals \cG\ltimes OF(M,\mathcal O)\rightrightarrows OF(M,\mathcal O),
\end{equation}
and the transverse canonical form is $\cH$-invariant: 
\begin{equation}\label{eq:lifted-canonical-invariance}
t_{\cH}^*\theta=s_{\cH}^*\theta.
\end{equation}

\end{proposition}

\begin{proof}
Let $X\colonequals OF(M,\mathcal O)$, which is a Hausdorff manifold. Since $s:\cG\rightarrow M$ is a submersion, the manifold of arrows $\cH_1=\cG\times_M X$ above is a smooth manifold, and the source map $s_{\cH}$ is a submersion.

For each $e\in X$,
\begin{equation}
s_{\cH}^{-1}(e)=s^{-1}(\pi(e))\times\{e\}\cong s^{-1}(\pi(e)),
\end{equation}
which is Hausdorff because the source fibres of $\cG$ are.

The target, unit, inverse, and multiplication maps are smooth because they are induced from the smooth structure maps of $\cG$ and the smooth lifted action of Proposition \ref{prop:action-axioms}. Hence $\cH$ is a Lie groupoid.

It remains to prove the invariance of the transverse canonical form. Let $(g,e)\in \cH_1$, where $g:x\rightarrow y$ and $e\in OF_x$. A tangent vector to $\cH_1=\cG\times_M OF(M,\mathcal O)$ at $(g,e)$ is a pair $(Y,\xi)\in T_g\cG\oplus T_eOF(M,\mathcal O)$ satisfying $ds_g(Y)=d\pi_e(\xi)$.

Using the definition of the transverse canonical form, we compute
\begin{equation}
(s_{\cH}^*\theta)_{(g,e)}(Y,\xi)
=\theta_e(\xi)
=e^{-1}\bigl([d\pi_e(\xi)]\bigr)
=e^{-1}\bigl([ds_g(Y)]\bigr).
\end{equation}
On the other hand,
\begin{equation}
t_{\cH}(g,e)=g\cdot e=\lambda_g^N\circ e.
\end{equation}
Since $g:x\rightarrow y$, the frame $g\cdot e$ lies over $y=t(g)$. Equivalently, on the arrow manifold $\cH_1$ we have the identity 
\begin{equation}\label{eq:lifted-target-projection}
\pi\circ t_{\cH}=t\circ \operatorname{pr}_1,
\end{equation}
 where $\operatorname{pr}_1:\cH_1=\cG\times_M OF(M,\mathcal O)\rightarrow \cG$. Differentiating this identity at $(g,e)$ gives
\begin{equation}
d\pi_{g\cdot e}\bigl(d(t_{\cH})_{(g,e)}(Y,\xi)\bigr)=dt_g(Y).
\end{equation}
Therefore
\begin{equation}
(t_{\cH}^*\theta)_{(g,e)}(Y,\xi)
=\theta_{g\cdot e}\bigl(d(t_{\cH})_{(g,e)}(Y,\xi)\bigr)
=(\lambda_g^N\circ e)^{-1}\bigl([dt_g(Y)]\bigr).
\end{equation}
By \eqref{eq:normal-representation-definition} and Proposition \ref{prop:normal-global}, $[dt_g(Y)]=\lambda_g^N\bigl([ds_g(Y)]\bigr)$. Thus
\begin{equation}
(t_{\cH}^*\theta)_{(g,e)}(Y,\xi)
=e^{-1}\bigl([ds_g(Y)]\bigr)
=(s_{\cH}^*\theta)_{(g,e)}(Y,\xi).
\end{equation}
Since $(g,e)$ and $(Y,\xi)$ were arbitrary, $t_{\cH}^*\theta=s_{\cH}^*\theta$.
\end{proof}

We will use the corresponding invariance of the transverse Levi-Civita connection form only at the level of local bisections (Proposition \ref{prop:molino-data-invariant}), where it follows from the preservation of the transverse metric and the naturality of the Levi-Civita connection.


We next discuss bisections, the normal representation, and invariance of Molino's structures.

\begin{definition}[{\cite{MoerdijkMrcun:IFLG}*{Sec. 5.1}}]\label{def:H-invariance}
Let $\cG\rightrightarrows M$ be a Lie groupoid. A \emph{local bisection} is a smooth map $\sigma:U\rightarrow \cG$ defined on an open set $U\subset M$ such that $s\circ\sigma=\id_U$ and $t\circ\sigma:U\rightarrow t(\sigma(U))$ is a diffeomorphism onto an open subset.
\end{definition}

For a local bisection $\sigma:U\rightarrow \cG$, we write
\begin{equation}\label{eq:bisection-frame-map}
f_\sigma:\pi^{-1}(U)\longrightarrow \pi^{-1}(t(\sigma(U))), \qquad f_\sigma(e)\colonequals \sigma(\pi(e))\cdot e,
\end{equation}
for the induced principal $O(q)$-bundle isomorphism.

Next we study the local diffeomorphisms induced by local bisections. We first show that they preserve the orbit foliation, and then identify their induced normal maps with the normal representation.

\begin{lemma}\label{lem:bisection-foliated}
Let $\sigma:U\rightarrow \cG$ be a local bisection of $\cG$ and set $\varphi=t\circ\sigma:U\rightarrow t(\sigma(U))\colonequals V$. Then $\varphi$ is orbit-preserving. For any $x_1,x_2\in U$ in the same $\cG$-orbit, $\varphi(x_1),\varphi(x_2)\in V$ lie in the same $\cG$-orbit. Moreover, $\varphi$ maps each leaf of $\mathcal O|_U$ diffeomorphically onto a leaf of $\mathcal O|_V$, and hence for all $x\in U$ we have $d\varphi_x\bigl(T_x(\mathcal O)\bigr)=T_{\varphi(x)}(\mathcal O)$.
\end{lemma}

\begin{proof}
For each $x\in U$, $\sigma(x)$ is an arrow from $x$ to $\varphi(x)$, so $\varphi$ preserves the $\cG$-orbits. Moreover, the local bisection $\sigma$ induces a Lie groupoid isomorphism
\begin{equation}
\begin{aligned}
\cG|_U &\longrightarrow \cG|_V \\
     g &\longmapsto \sigma(t(g))g\sigma(s(g))^{-1}
\end{aligned}
\end{equation}
over $\varphi$. Differentiating at the units and using compatibility with the anchors gives
\begin{equation}
d\varphi_x\bigl(T_x(\mathcal O)\bigr) = T_{\varphi(x)}(\mathcal O).
\end{equation}
Thus $\varphi$ maps the leaves of $\mathcal O|_U$ diffeomorphically onto the leaves of $\mathcal O|_V$.
\end{proof}

In the next lemma, we identify the normal map induced by a local bisection with the global normal representation.

\begin{lemma}\label{lem:bisection-normal}
Let $\cG\rightrightarrows M$ be regular with orbit foliation $\mathcal O$ and normal bundle $N=T(M)/T(\mathcal O)$. Let $\sigma:U\rightarrow \cG$ be a local bisection and $\varphi\colonequals t\circ\sigma:U\rightarrow V$. Then for every $x\in U$ we have 
\begin{equation}\label{eq:bisection-normal-map}
(d\varphi)^N_x=\lambda^N_{\sigma(x)}:N_x\longrightarrow N_{\varphi(x)}.
\end{equation}

\end{lemma}

\begin{proof}
Let $x\in U$ and $v\in T_x(M)$, and let $X=d\sigma_x(v)\in T_{\sigma(x)}(\cG)$. Since $s\circ\sigma=\id_U$, we have $ds_{\sigma(x)}(X)=v$.

Let $O$ be the orbit through $x$. Since $\cG$ is regular, we identify $N_x=\nu_x(O)$ and $N_{\varphi(x)}=\nu_{\varphi(x)}(O)$. By Proposition \ref{prop:normal-global}, the restriction of $\lambda^N$ to $O$ is the normal representation along $O$. Hence, by Definition \ref{def:normal-orbit}, $\lambda^N_{\sigma(x)}([v])=[dt_{\sigma(x)}(X)]$.

But $dt_{\sigma(x)}(X)=d(t\circ\sigma)_x(v)=d\varphi_x(v)$, hence $\lambda^N_{\sigma(x)}([v])=[d\varphi_x(v)]$. By Lemma \ref{lem:bisection-foliated}, $\varphi$ is foliated, so $[d\varphi_x(v)]=(d\varphi)^N_x([v])$.
\end{proof}

\begin{remark}[Clarification]
For a non-\'etale Lie groupoid, an arrow $g:x\rightarrow y$ need not determine a unique germ of a local diffeomorphism of $M$ near $x$: different local bisections $\sigma$ through $g$ may yield different maps $\varphi=t\circ\sigma$. By Lemma \ref{lem:bisection-normal}, for any such local bisection we have $(d\varphi)^N_x=\lambda_g^N:N_x\rightarrow N_y$. So the induced normal map depends only on the arrow $g$, not on the choice of local bisection.
\end{remark}


\begin{proposition}\label{prop:molino-data-invariant}
Under Hypothesis \ref{hyp:standing}, the lifted foliation $\widetilde{\mathcal O}$, the transverse canonical form $\theta$, and the transverse Levi-Civita connection form $\omega$ are invariant under the local diffeomorphisms of $OF(M,\mathcal O)$ induced by local bisections of $\cG$. Explicitly, for each local bisection $\sigma:U\rightarrow \cG$ with $V=t(\sigma(U))$, the induced map
\begin{equation}
f:\pi^{-1}(U)\longrightarrow \pi^{-1}(V)
\end{equation}
satisfies
\begin{equation}\label{eq:bisection-molino-invariance}
\begin{aligned}
f^*\bigl(\theta|_{\pi^{-1}(V)}\bigr)&=\theta|_{\pi^{-1}(U)},\\
f^*\bigl(\omega|_{\pi^{-1}(V)}\bigr)&=\omega|_{\pi^{-1}(U)},\\
(df)\bigl(T(\widetilde{\mathcal O})|_{\pi^{-1}(U)}\bigr)&=T(\widetilde{\mathcal O})|_{\pi^{-1}(V)}.
\end{aligned}
\end{equation}
\end{proposition}

\begin{proof}
Let $\sigma:U\rightarrow \cG$ be a local bisection and let $\varphi\colonequals t\circ\sigma:U\rightarrow V$. By Lemma \ref{lem:bisection-foliated}, $\varphi$ preserves the orbit foliation. By Lemma \ref{lem:bisection-normal}, $(d\varphi)^N_x=\lambda^N_{\sigma(x)}$ on the normal bundles. By Proposition \ref{prop:normal-global}, the maps $\lambda^N_g$ are fibrewise isometries for $g^N$. Hence $\varphi$ is a foliated transverse isometry in the sense of Definition \ref{def:local-transverse-isometry}.

We compare the two induced maps on frames. For $e\in OF_x$ with $x=\pi(e)$,
\begin{equation}
f(e)=\sigma(x)\cdot e=\lambda^N_{\sigma(x)}\circ e=(d\varphi)^N_x\circ e=OF(\varphi)(e).
\end{equation}

Hence $f=OF(\varphi)$, and the conclusion follows from Lemma \ref{lem:functoriality-molino}.
\end{proof}

We next describe each leaf of $\mathcal O$ as the quotient of a leaf of $\widetilde{\mathcal O}$ by its isotropy subgroup in $O(q)$.

\begin{proposition}\label{leafwise-Oq}
Let $(Q,\mathcal{F})$ be a foliated manifold, and let $\pi:\widehat{M}\rightarrow Q$ be a principal $O(q)$-bundle. Assume that $\widehat{M}$ carries a foliation $\widehat{\mathcal{F}}$ such that $\pi:(\widehat{M},\widehat{\mathcal{F}})\rightarrow (Q,\mathcal{F})$ is a transverse principal $O(q)$-bundle. Let $L_{\widehat M}$ be a leaf of $\widehat{\mathcal{F}}$ and let $L_Q\colonequals \pi(L_{\widehat M})$. Then
\begin{enumerate}[label=\textup{(\roman*)}]
\item The right $O(q)$-action sends leaves of $\widehat{\mathcal{F}}$ to leaves. Hence, the isotropy subgroup $O(q)_{L_{\widehat M}}\colonequals \{A\in O(q)\mid L_{\widehat M}\cdot A=L_{\widehat M}\}$ acts on $L_{\widehat M}$.

\item The restricted map $\pi|_{L_{\widehat M}}:L_{\widehat M}\rightarrow L_Q$ is a covering projection.

\item Consequently, $L_{\widehat M}/O(q)_{L_{\widehat M}}$ can be identified with $L_Q$.
\end{enumerate}
\end{proposition}

\begin{proof}
(i) follows from \cite{MoerdijkMrcun:IFLG}*{Sec. 4.2.2, (i)}: the right $O(q)$-action preserves $\widehat{\mathcal{F}}$, hence maps leaves to leaves. (ii) follows from \cite{MoerdijkMrcun:IFLG}*{Sec. 4.2.2, (ii)}.

For (iii), if $x\in L_{\widehat M}$ and $A\in O(q)_{L_{\widehat M}}$, then $\pi(x\cdot A)=\pi(x)$, so $\pi|_{L_{\widehat M}}$ factors through $L_{\widehat M}/O(q)_{L_{\widehat M}}$. Conversely, if $x,x'\in L_{\widehat M}$ satisfy $\pi(x)=\pi(x')$, then $x'=x\cdot A$ for a unique $A\in O(q)$. Since the right $O(q)$-action sends leaves to leaves and $L_{\widehat M}\cdot A$ contains $x'$, we have $L_{\widehat M}\cdot A=L_{\widehat M}$, so $A\in O(q)_{L_{\widehat M}}$. Hence $L_{\widehat M}/O(q)_{L_{\widehat M}}$ can be identified with $L_Q$.
\end{proof}

\section{The basic Lie algebroid and Molino's theory in the regular case}\label{sec:basic-algebroid-molino-theorem}

\subsection{Projectable transverse fields and the basic Lie algebroid}\label{subsec:basic-lie-algebroid}

We next review transverse vector fields and the associated structural Lie algebra \cite{Molino:RF}. We use them to describe the Lie foliation structure on the fibres. 

\begin{definition}[{\cite{MoerdijkMrcun:IFLG}*{Sec. 4.1.2}}]\label{def:structural-Lie-algebra}
Let $(M,\F)$ be a foliated manifold. Let $\mathfrak X(\F)\subset \mathfrak X(M)$ be the Lie algebra of vector fields tangent to $\F$. A vector field $X\in\mathfrak X(M)$ is \emph{projectable} if $[X,Y]\in \mathfrak X(\F)$ for all $Y\in\mathfrak X(\F)$. Let $L(M,\F)$ be the Lie algebra of projectable vector fields, and let 
\begin{equation}\label{eq:transverse-fields-quotient}
l(M,\F)\colonequals L(M,\F)/\mathfrak X(\F).
\end{equation}
 Elements of $l(M,\F)$ are called transverse vector fields.
\end{definition}

\begin{definition}[{\cite{MoerdijkMrcun:IFLG}*{Secs. 4.1.1--4.1.2}}]\label{def:homogeneous-transversely-parallelizable-basic-foliation}
Let $(M,\F)$ be a foliated manifold.

An \emph{automorphism} of $(M,\F)$ is a diffeomorphism $\phi:M\rightarrow M$ which preserves the foliation, equivalently, $d\phi_x(T_x\F)=T_{\phi(x)}\F$ for every $x\in M$. The group of automorphisms is denoted by $\Aut(M,\F)$. The foliation $\F$ is called \emph{homogeneous} if $\Aut(M,\F)$ acts transitively on $M$.

If $\F$ has codimension $q$, then $(M,\F)$ is called \emph{transversely parallelizable} if there exist transverse vector fields $\bar Y_1,\dots,\bar Y_q\in l(M,\F)$ which form a global frame of the normal bundle $N(\F)=TM/T\F$. Such a frame is called a \emph{transverse parallelism}.

For a homogeneous foliation, let $\mathfrak X_{\bas}(\F)\colonequals \{X\in\mathfrak X(M)\mid X(f)=0\text{ for all }f\in\Omega^0_{\bas}(M,\F)\}$. For each $x\in M$, put $E_x\colonequals \{X_x\mid X\in\mathfrak X_{\bas}(\F)\}$, and let $E\colonequals \bigcup_{x\in M}E_x\subset TM$. Since $\F$ is homogeneous, $E$ is an involutive subbundle of $TM$. The associated foliation $\F_{\bas}$, characterized by $T\F_{\bas}=E$, is called the \emph{basic foliation} associated to $\F$. Since $T\F\subset T\F_{\bas}$, every leaf of $\F$ is contained in a leaf of $\F_{\bas}$; the leaves of $\F_{\bas}$ are called the \emph{basic leaves}.
\end{definition}

\begin{proposition}[cf. {\cite{MoerdijkMrcun:IFLG}*{Sec. 6.4, Lemma 6.9}}]\label{prop:basic-Lie-algebroid}
Let $\F$ be a homogeneous transversely parallelizable foliation of codimension $r$ on a compact manifold $X$. Then $W\colonequals X/\F_{\bas}$ is a smooth Hausdorff manifold, $\pi_{\bas}:X\rightarrow W$ is a fibre bundle, and there exists a transitive Lie algebroid $A\colonequals b(X,\F)\rightarrow W$ such that $\Gamma(A)\cong l(X,\F)$ as a $C^\infty(W)$-module. Its anchor $\rho_A:A\rightarrow TW$ is induced by the $C^\infty(W)$-linear Lie algebra homomorphism $\varrho:l(X,\F)\rightarrow l(X,\F_{\bas})\cong \mathfrak X(W)$, and the resulting Lie algebroid $A$ is independent, up to Lie algebroid isomorphism over $W$, of the choice of transverse parallelism used to identify $l(X,\F)$ with $\Gamma(W\times\mathbb R^r)$.
\end{proposition}

\begin{proof}
1. Geometry of basic fibre bundle.

Since $\F$ is homogeneous and $X$ is compact, by \cite{MoerdijkMrcun:IFLG}*{Theorem 4.3 (iii), (v)}, the basic foliation is strictly simple, $W\colonequals X/\F_{\bas}$ is a smooth Hausdorff manifold, and $\pi_{\bas}:X\rightarrow W$ is a fibre bundle.

Now choose a transverse parallelism $\bar Y_1,\dots,\bar Y_r\in l(X,\F)$ where $r=\codim \F$ with representatives $Y_1,\dots,Y_r\in L(X,\F)$.

2. $l(X,\F)$ is a free $C^\infty(W)$-module of rank $r$.

Because $\bar Y_1,\dots,\bar Y_r$ form a global frame of the normal bundle $N(\F)$, every $\bar Y\in l(X,\F)$ can be written uniquely as $\bar Y=\sum_{i=1}^r a_i\bar Y_i$ for $a_i\in C^\infty(X)$.

Now we want to show that the $a_i$ are basic. Choose a projectable representative $Y$ of $\bar Y$, and let $T\in\mathfrak X(\F)$. Since $Y$ and each $Y_i$ are projectable, $[T,Y]$ and $[T,Y_i]$ are tangent to $\F$. Then,
\begin{equation}\label{eq:basic-coefficient-calculation}
\begin{aligned}
0=\overline{[T,Y]}=\overline{\left[T,\sum_{i=1}^r a_iY_i\right]}&=\overline{\left(\sum_{i=1}^r T(a_i)Y_i+\sum_{i=1}^r a_i[T,Y_i]\right)}\\
&=\sum_{i=1}^r T(a_i)\bar Y_i+\sum_{i=1}^r a_i\overline{[T,Y_i]}\\
&=\sum_{i=1}^r T(a_i)\bar Y_i.
\end{aligned}
\end{equation}

Because the $\bar Y_i$ are pointwise linearly independent in $N(\F)$, $T(a_i)=0$ for all $i$. Therefore each $a_i$ is $\F$-basic. By \cite{MoerdijkMrcun:IFLG}*{Theorem 4.3\textup{(iv)}}, for each $i$ there is a unique $f_i\in C^\infty(W)$ such that 
\begin{equation}\label{eq:basic-coefficient-descent}
a_i=f_i\circ\pi_{\bas}.
\end{equation}
 Thus $\bar Y=\sum_{i=1}^r f_i\bar Y_i$ with $f_i\in C^\infty(W)$.

3. Vector bundle identifications.

After choosing $\bar Y_1,\dots,\bar Y_r$, we define $A\colonequals W\times \mathbb R^r$. Let $e_1,\dots,e_r$ be the global frame of $A$, and we identify sections
\begin{equation}\label{eq:basic-section-identification}
\Gamma(A)\ni \sum f_ie_i\longleftrightarrow \sum f_i\bar Y_i\in l(X,\F).
\end{equation}

4. Anchor map identification.

Define
\begin{equation}
\varrho:l(X,\F)\longrightarrow l(X,\F_{\bas})\cong \mathfrak X(W)
\end{equation}
by projection along $\pi_{\bas}$ as follows. By \cite{MoerdijkMrcun:IFLG}*{Lemma 4.5}, $L(X,\F)\subseteq L(X,\F_{\bas})$. If $\bar Y\in l(X,\F)$ is represented by $Y\in L(X,\F)$, then $Y$ is projectable for $\F_{\bas}$. Since the leaves of $\F_{\bas}$ are the fibres of $\pi_{\bas}:X\rightarrow W$, there is a unique vector field $Y_W\in\mathfrak X(W)$ such that 
\begin{equation}\label{eq:basic-projection}
d\pi_{\bas}(Y_x)=(Y_W)_{\pi_{\bas}(x)}
\end{equation}
 for every $x\in X$. Let $\varrho(\bar Y)\colonequals Y_W$, using the identification $l(X,\F_{\bas})\cong\mathfrak X(W)$.

We check that $\varrho$ is well-defined. If $Y'=Y+U$ with $U\in \mathfrak X(\F)$, then $U$ is tangent to $\F_{\bas}$ as well, so $d\pi_{\bas}(U)=0$. Hence $Y$ and $Y'$ project to the same vector field on $W$.

Then we check that $\varrho$ is $C^\infty(W)$-linear and a Lie algebra homomorphism: For $f\in C^\infty(W)$, the section $f\bar Y$ is represented by $(f\circ\pi_{\bas})Y$. If $Y$ projects to $Y_W$, then $(f\circ\pi_{\bas})Y$ projects to $fY_W$, so 
\begin{equation}\label{eq:basic-projection-linearity}
\varrho(f\bar Y)=f\varrho(\bar Y).
\end{equation}
 If $Y,Y'\in L(X,\F)$ project to $Y_W,Y'_W\in\mathfrak X(W)$, then $[Y,Y']$ projects to $[Y_W,Y'_W]$. Hence 
\begin{equation}\label{eq:basic-projection-bracket}
\varrho([\bar Y,\bar Y'])=[\varrho(\bar Y),\varrho(\bar Y')].
\end{equation}

Under the identification $\Gamma(A)\cong l(X,\F)$, the map $\varrho$ therefore induces a bundle map
\begin{equation}
\rho_A:A\longrightarrow TW,
\end{equation}
which is the anchor.

5. Define bracket.

We define the bracket on $\Gamma(A)$ by transporting the Lie bracket on $l(X,\F)$ through the identification
\begin{equation}\label{eq:basic-Phi-identification}
\begin{aligned}
\Phi\colon \Gamma(A)&\overset{\cong}{\longrightarrow} l(X,\F), \\
\sum_{i=1}^r f_i e_i&\longmapsto \sum_{i=1}^r f_i\bar Y_i.
\end{aligned}
\end{equation}
Thus, for $\xi,\eta \in \Gamma(A)$, we write 
\begin{equation}\label{eq:basic-bracket-transport}
[\xi,\eta]_A\colonequals \Phi^{-1}\big([\Phi(\xi),\Phi(\eta)]\big).
\end{equation}

To compute the bracket on $l(X,\F)$, choose representatives $Y,Z\in L(X,\F)$ of $\bar Y,\bar Z\in l(X,\F)$. Then
\begin{equation}\label{eq:transverse-bracket}
[\bar Y,\bar Z]\colonequals \overline{[Y,Z]}.
\end{equation}

This is well defined: if $Y'=Y+U$ and $Z'=Z+V$ with $U,V\in \mathfrak X(\F)$, then $[Y',Z']-[Y,Z]=[U,Z]+[Y,V]+[U,V]\in \mathfrak X(\F)$. Skew-symmetry and the Jacobi identity follow from the corresponding properties of the Lie bracket on $l(X,\F)$.

It remains to verify the Leibniz rule. Let $\xi,\eta\in \Gamma(A)$ correspond to $\bar Y,\bar Z\in l(X,\F)$, and let $f\in C^\infty(W)$. If $Y,Z\in L(X,\F)$ represent $\bar Y,\bar Z$, then $(f\circ \pi_{\bas})Z$ represents $f\eta$. Hence
\begin{equation}
[\xi,f\eta]_A=\Phi^{-1}\!\big(\overline{[Y,(f\circ \pi_{\bas})Z]}\big).
\end{equation}
Using the Leibniz rule for vector fields on $X$, we get
\begin{equation}
[Y,(f\circ \pi_{\bas})Z]=(f\circ \pi_{\bas})[Y,Z]+Y(f\circ \pi_{\bas})\,Z.
\end{equation}
Since $\rho_A(\xi)=\varrho(\bar Y)$ is the projection of $Y$ to $W$, we have
\begin{equation}\label{eq:basic-anchor-on-functions}
Y(f\circ \pi_{\bas})=(\rho_A(\xi)f)\circ \pi_{\bas}.
\end{equation}
Therefore
\begin{equation}\label{eq:basic-Leibniz}
[\xi,f\eta]_A=f[\xi,\eta]_A+\rho_A(\xi)(f)\eta.
\end{equation}

Since $\rho_A$ is induced by the $C^\infty(W)$-linear Lie algebra homomorphism $\varrho:l(X,\F)\rightarrow \mathfrak X(W)$, the anchor preserves brackets. Hence $(A,[\,,\,]_A,\rho_A)$ is a Lie algebroid over $W$.

6. $A$ is transitive.

Recall that a Lie algebroid is transitive if its anchor
is surjective at every point
\cite{MoerdijkMrcun:IFLG}*{Sec. 6.2}.

Let $w\in W$ and $v\in T_wW$. Let $x\in X$ with $\pi_{\bas}(x)=w$. Since $\pi_{\bas}$ is a submersion, there exists $\zeta\in T_xX$ such that $d\pi_{\bas}(\zeta)=v$. Recall that we chose representatives $Y_1,\dots,Y_r\in L(X,\F)$. Then $(Y_1)_x,\dots,(Y_r)_x$ span a subspace complementary to $T_x\F$ in $T_xX$. Therefore, there exist scalars $a_1,\dots,a_r\in \mathbb R$ such that
\begin{equation}
a_1(Y_1)_x+\dots+a_r(Y_r)_x-\zeta\in T_x\F.
\end{equation}

Let $Y\colonequals a_1Y_1+\dots+a_rY_r$. Then $Y\in L(X,\F)$. Since $T_x\F\subset \ker(d\pi_{\bas})_x$, we get $d\pi_{\bas}(Y_x)=d\pi_{\bas}(\zeta)=v$. Therefore the class $\bar Y\in l(X,\F)\cong \Gamma(A)$ satisfies $\rho_A(\bar Y)(w)=v$. Thus $\rho_{A,w}$ is surjective for every $w\in W$, so $A$ is transitive.

We also get an exact sequence
\begin{equation}
0\longrightarrow \ker \rho_A\longrightarrow A\xlongrightarrow{\rho_A}TW\longrightarrow 0.
\end{equation}

This is the transitive basic Lie algebroid attached to $(X,\F)$.

7. $A$ is independent of the transverse parallelism.

Suppose $\bar Y_1,\dots,\bar Y_r$ and $\bar Y_1',\dots,\bar Y_r'$ are two transverse parallelisms of $(X,\F)$. Using these parallelisms, we obtain the identifications
\begin{equation}
\Phi,\Phi':\Gamma(W\times \R^r)\xlongrightarrow{\cong} l(X,\F),
\end{equation}
hence two Lie algebroid structures $([\,,\,],\rho)$ and $([\,,\,]',\rho')$ on $W\times \R^r$.

Let
\begin{equation}\label{eq:basic-change-of-parallelism}
T\colonequals (\Phi')^{-1}\circ \Phi:\Gamma(W\times \R^r)\longrightarrow \Gamma(W\times \R^r).
\end{equation}
Then for any sections $\xi,\eta$,
\begin{equation}\label{eq:basic-change-bracket}
T([\xi,\eta])
=(\Phi')^{-1}\big([\Phi(\xi),\Phi(\eta)]\big)
=(\Phi')^{-1}\big([\Phi'(T\xi),\Phi'(T\eta)]\big)
=[T\xi,T\eta]'.
\end{equation}
Also,
\begin{equation}\label{eq:basic-change-anchor}
\rho'(T\xi)=\varrho(\Phi'(T\xi))=\varrho(\Phi(\xi))=\rho(\xi).
\end{equation}
Since $\Phi$ and $\Phi'$ are isomorphisms of $C^\infty(W)$-modules, $T$ is $C^\infty(W)$-linear and is therefore induced by a smooth vector bundle automorphism of $W\times \mathbb R^r$. Since $T$ also preserves the bracket and the anchor, it is an isomorphism of Lie algebroids. Therefore the Lie algebroid $A$ is independent, up to isomorphism, of the chosen transverse parallelism.
\end{proof}

\subsection{Fibrewise structural Lie algebras}\label{subsec:structural-lie-algebras}

\begin{lemma}\label{lem:isotropy-structural-clean}
Let $\F$ be a homogeneous transversely parallelizable foliation of a compact manifold $X$, let $\pi_{\bas}:X\rightarrow W\colonequals X/\F_{\bas}$ be the basic fibre bundle, and let $A\colonequals b(X,\F)\rightarrow W$ be the basic Lie algebroid. For $w\in W$, write $L_w\colonequals \pi_{\bas}^{-1}(w)$. Then restriction to the fibre $L_w$ induces a Lie algebra isomorphism
\begin{equation}
(\ker\rho_A)_w \cong l(L_w,\F|_{L_w}),
\end{equation}
and $(L_w,\F|_{L_w})$ is a Lie foliation.
\end{lemma}

\begin{proof}
Let $w\in W$ and $L_w=\pi_{\bas}^{-1}(w)$. By Proposition \ref{prop:basic-Lie-algebroid}, we have that $A=b(X,\F)\rightarrow W$ is a transitive Lie algebroid with $\Gamma(A)\cong l(X,\F)$ as a $C^\infty(W)$-module, and its anchor $\rho_A$ is induced by the $C^\infty(W)$-linear Lie algebra homomorphism $\varrho:l(X,\F)\rightarrow l(X,\F_{\bas})\cong \mathfrak X(W)$.

\emph{Step 1}: We want to rewrite the isotropy fibre as a quotient.

Let
\begin{equation}\label{eq:isotropy-vanishing-spaces}
I_w\colonequals \{f\in C^\infty(W)\mid f(w)=0\},\qquad \mathfrak k_w\colonequals \{\bar Y\in l(X,\F)\mid \varrho(\bar Y)(w)=0\}.
\end{equation}

Since $\Gamma(A)\cong l(X,\F)$, under the identification $A_w\cong \Gamma(A)/I_w\Gamma(A)$, we have
\begin{equation}\label{eq:basic-fibre-quotient}
A_w \cong l(X,\F)/I_w\,l(X,\F).
\end{equation}
Under this identification, $\rho_{A,w}:A_w\rightarrow T_wW$ is induced by $\bar Y\mapsto \varrho(\bar Y)(w)$. Therefore, 
\begin{equation}\label{eq:isotropy-fibre-quotient}
(\ker\rho_A)_w\cong \mathfrak k_w/I_w\,l(X,\F).
\end{equation}

We verify that $\mathfrak k_w/I_w\,l(X,\F)$ carries a well-defined Lie algebra structure:

1. $I_w\,l(X,\F)\subset \mathfrak k_w$: if $f\in I_w$ and $\bar Y\in l(X,\F)$, then $\varrho(f\bar Y)(w)=f(w)\varrho(\bar Y)(w)=0$.

2. $\mathfrak k_w$ is a Lie subalgebra of $l(X,\F)$: let $\mathfrak m_w\colonequals \{Z\in \mathfrak X(W):Z_w=0\}$. It is easy to show that $\mathfrak m_w$ is a Lie subalgebra of $\mathfrak X(W)$. Because $\varrho$ is a Lie algebra homomorphism, $\mathfrak k_w=\varrho^{-1}(\mathfrak m_w)$ is a Lie subalgebra of $l(X,\F)$.

3. $I_w\,l(X,\F)$ is an ideal in $\mathfrak k_w$: by the Leibniz rule, if $\bar Y\in \mathfrak k_w$ and $f\bar Z\in I_w\,l(X,\F)$, then $[\bar Y,f\bar Z]=f[\bar Y,\bar Z]+(\varrho(\bar Y)f)\bar Z\in I_w\,l(X,\F)$, because $(\varrho(\bar Y)f)(w)=df_w(\varrho(\bar Y)(w))=0$.

\emph{Step 2}: We construct the restriction map to the fibre.

Let $\bar Y\in \mathfrak k_w$, choose a representative $Y\in L(X,\F)$, and let $Z\colonequals \varrho(\bar Y)\in\mathfrak X(W)$. Since $\bar Y\in\mathfrak k_w$, we have $Z_w=0$, so for every $x\in L_w$ we have
\begin{equation}
d\pi_{\bas}(Y_x)=Z_{\pi_{\bas}(x)}=Z_w=0.
\end{equation}
Hence $Y_x\in \ker(d\pi_{\bas})_x=T_xL_w$, so $Y$ restricts to a vector field $Y|_{L_w}\in \mathfrak X(L_w)$.

To check that $Y|_{L_w}$ is projectable for $\F|_{L_w}$, let $X_0\in \mathfrak X(\F|_{L_w})$. Choose a local extension $\widetilde X_0\in \mathfrak X(\F)$ near each point of $L_w$ with $\widetilde X_0|_{L_w}=X_0$. Since $Y\in L(X,\F)$, we have $[Y,\widetilde X_0]\in \mathfrak X(\F)$, hence also $[\widetilde X_0,Y]\in \mathfrak X(\F)$. Thus for every $x\in L_w$, $[\widetilde X_0,Y](x)\in T_x\F\subset T_x\F_{\bas}=T_xL_w$. This implies that $[\widetilde X_0,Y]|_{L_w}$ is a well-defined vector field on $L_w$. If $f\in C^\infty(L_w)$ and $\widetilde f$ is an extension of $f$, then
\begin{equation}
([\widetilde X_0,Y]|_{L_w})(f)=[\widetilde X_0,Y](\widetilde f)|_{L_w}.
\end{equation}
We compute that
\begin{equation}
\begin{aligned}
[\widetilde X_0,Y](\widetilde f)|_{L_w}
&=\widetilde X_0(Y\widetilde f)|_{L_w}-Y(\widetilde X_0\widetilde f)|_{L_w}\\
&=X_0((Y\widetilde f)|_{L_w})-Y|_{L_w}((\widetilde X_0\widetilde f)|_{L_w})\\
&=X_0(Y|_{L_w}(f))-Y|_{L_w}(X_0(f))\\
&=[X_0,Y|_{L_w}](f).
\end{aligned}
\end{equation}

This is to say 
\begin{equation}\label{eq:restriction-bracket}
[X_0,Y|_{L_w}]=[\widetilde X_0,Y]|_{L_w}\in \mathfrak X(\F|_{L_w}),
\end{equation}
 so $Y|_{L_w}\in L(L_w,\F|_{L_w})$, and we may define
\begin{equation}\label{eq:Res-definition}
\operatorname{Res}_w(\bar Y)\colonequals \overline{Y|_{L_w}}\in l(L_w,\F|_{L_w}).
\end{equation}

This is well defined: if $Y'$ is another representative of $\bar Y$, then $Y'-Y\in \mathfrak X(\F)$, so $(Y'-Y)|_{L_w}\in \mathfrak X(\F|_{L_w})$, hence $\overline{Y'|_{L_w}}=\overline{Y|_{L_w}}$. It is a Lie algebra homomorphism because restriction commutes with Lie brackets for vector fields tangent to $L_w$, so $\operatorname{Res}_w([\bar Y,\bar Z])=[\operatorname{Res}_w(\bar Y),\operatorname{Res}_w(\bar Z)]$.

Now let $f\in I_w$ and let $\bar Y\in l(X,\F)$ be represented by $Y\in L(X,\F)$. Since $\pi_{\bas}|_{L_w}\equiv w$ and $f\bar Y\in I_w\,l(X,\F)\subset \mathfrak k_w$, we have
\begin{equation}
\operatorname{Res}_w(f\bar Y) = \overline{((f\circ\pi_{\bas})Y)|_{L_w}} = \overline{f(w)Y|_{L_w}} = 0.
\end{equation}
Therefore $\operatorname{Res}_w$ vanishes on $I_w\,l(X,\F)$ and hence descends to a Lie algebra homomorphism
\begin{equation}\label{eq:res-definition}
\operatorname{res}_w:(\ker\rho_A)_w\longrightarrow l(L_w,\F|_{L_w}).
\end{equation}

\emph{Step 3}: It remains to prove that $\operatorname{res}_w$ is an isomorphism.

Let $x\in L_w$. Choose a transverse parallelism $\bar Y_1,\dots,\bar Y_r$ on $(X,\F)$. Then $l(X,\F)$ is a free $C^\infty(W)$-module with basis $\bar Y_1,\dots,\bar Y_r$, so the classes $[\bar Y_1]_w,\dots,[\bar Y_r]_w$ in $A_w\cong l(X,\F)/I_w\,l(X,\F)$ form a basis of $A_w$. Since $(\bar Y_1)_x,\dots,(\bar Y_r)_x$ form a basis of $N_x(\F)$, evaluation at $x$ defines a linear isomorphism
\begin{equation}\label{eq:basic-evaluation}
\begin{aligned}
\operatorname{ev}_x^X\colon A_w&\longrightarrow N_x(\F), \\
[\bar Y]&\longmapsto \bar Y_x.
\end{aligned}
\end{equation}
To see that evaluation vanishes on $I_w\,l(X,\F)$, it suffices by linearity to consider $\bar s=f\bar Y$ with $f\in I_w$ and $\bar Y\in l(X,\F)$. For $x\in L_w$, we have $(f\circ\pi_{\bas})(x)=f(w)=0$, so
\begin{equation}
\bar s_x=\overline{((f\circ \pi_{\bas})Y)_x}=0.
\end{equation}

Similarly, by \cite{MoerdijkMrcun:IFLG}*{Theorem 4.9}, $(L_w,\F|_{L_w})$ is transversely parallelizable and all its basic functions are constant, so, by the same argument as in the proof of \cite{MoerdijkMrcun:IFLG}*{Theorem 4.24}, the evaluation map
\begin{equation}\label{eq:fibre-evaluation}
\begin{aligned}
\operatorname{ev}_x^L\colon l(L_w,\F|_{L_w})&\longrightarrow N_x(\F|_{L_w}), \\
\bar Z&\longmapsto \bar Z_x,
\end{aligned}
\end{equation}
is a linear isomorphism.

Under the isomorphism $\operatorname{ev}_x^X$, we claim that the anchor $\rho_{A,w}:A_w\rightarrow T_wW$ corresponds to the map induced by $d\pi_{\bas}$ on normal spaces, $\overline{d\pi_{\bas}}:N_x(\F)\rightarrow T_wW$, $[v]\mapsto d\pi_{\bas}(v)$.

Indeed, if $\bar Y\in l(X,\F)$ is represented by $Y\in L(X,\F)$, then $d\pi_{\bas}(Y)=\varrho(\bar Y)\circ\pi_{\bas}$, and therefore 
\begin{equation}\label{eq:basic-evaluation-anchor}
\overline{d\pi_{\bas}}(\bar Y_x)=\varrho(\bar Y)(w)=\rho_{A,w}([\bar Y]).
\end{equation}
 Hence $\operatorname{ev}_x^X$ restricts to an isomorphism
\begin{equation}\label{eq:isotropy-evaluation}
(\ker\rho_A)_w \xlongrightarrow{\ \cong\ } \ker \bigl(\overline{d\pi_{\bas}}:N_x(\F)\longrightarrow T_wW\bigr)= T_xL_w/T_x(\F) = N_x(\F|_{L_w}),
\end{equation}
since $T_xL_w=\ker(d\pi_{\bas})_x$ and $T_x(\F|_{L_w})=T_x(\F)$.

Finally, for $\xi=[\bar Y]\in(\ker\rho_A)_w$, we have
\begin{equation}\label{eq:restriction-evaluation}
\operatorname{ev}_x^L(\operatorname{res}_w(\xi)) = \operatorname{ev}_x^L(\overline{Y|_{L_w}}) = \overline{Y_x} = \operatorname{ev}_x^X(\xi).
\end{equation}
Equivalently, the following square commutes:
\begin{equation}
\begin{tikzcd}[column sep=large, row sep=large]
(\ker \rho_A)_w \arrow[r, "\operatorname{res}_w"] \arrow[d, "\operatorname{ev}_x^X|_{(\ker\rho_A)_w}" left, "\cong" right]
& l(L_w,\F|_{L_w}) \arrow[d, "\operatorname{ev}_x^L" left, "\cong" right] \\
N_x(\F|_{L_w}) \arrow[r, equal]
& N_x(\F|_{L_w})
\end{tikzcd}
\end{equation}
Thus $\operatorname{ev}_x^L\circ \operatorname{res}_w=\operatorname{ev}_x^X\big|_{(\ker\rho_A)_w}$. Both vertical maps in the diagram are isomorphisms, so $\operatorname{res}_w$ is an isomorphism. Since it is already a Lie algebra homomorphism, it is a Lie algebra isomorphism.

\emph{Step 4}: We want to show that $(L_w,\F|_{L_w})$ is a Lie foliation.

Since $L_w=\pi_{\bas}^{-1}(w)$ and $X$ is compact, the fibre $L_w$ is compact. Since $L_w$ is a leaf of $\F_{\bas}$, it is connected. By \cite{MoerdijkMrcun:IFLG}*{Theorem 4.9}, the foliated manifold $(L_w,\F|_{L_w})$ is transversely parallelizable and its basic functions are constant. Hence \cite{MoerdijkMrcun:IFLG}*{Theorem 4.24} implies that $(L_w,\F|_{L_w})$ is a Lie foliation.
\end{proof}

\begin{proposition}\label{prop:algebroid-triv}
    For every $w_0\in W$, there exists an open neighborhood $U\subset W$ of $w_0$ and sections $\sigma_1,\dots,\sigma_s\in \Gamma(U,\ker \rho_A)$, where $s\colonequals \mathrm{rank}(\ker \rho_A)$, such that, for every $w\in U$, the elements $\operatorname{res}_w(\sigma_1(w)),\dots,\operatorname{res}_w(\sigma_s(w))$ form a basis of $l(L_w,\F|_{L_w})$.
\end{proposition}

\begin{proof}
Since $A\rightarrow W$ is transitive, $\rho_A:A\rightarrow TW$ is surjective at every point and has constant rank $\dim W$. Hence $\ker \rho_A\subset A$ is a smooth vector subbundle.

Shrink around $w_0$ and choose a smooth local frame $\sigma_1,\dots,\sigma_s$ of $\ker \rho_A|_U$. We can do this because there is a vector bundle isomorphism $\Theta:(\ker \rho_A)|_U\xrightarrow{\cong}U\times \RR^s$. Let $\varepsilon_1,\dots,\varepsilon_s$ be the standard basis of $\RR^s$. Using $\Theta^{-1}$, these sections are given by
\begin{equation}
\sigma_p(w)\colonequals \Theta^{-1}(w,\varepsilon_p),\qquad p=1,\dots,s.
\end{equation}
Each $\sigma_p$ is smooth because $\Theta^{-1}$ is smooth.

Choose a transverse parallelism on $(X,\F)$, and let $e_1,\dots,e_r$ be the corresponding global frame of $A$, where $r=\mathrm{rank}(A)=\codim \F$. Let $Y_1,\dots,Y_r\in L(X,\F)$ be projectable vector fields representing $e_1,\dots,e_r$, respectively.

On $U$, let \begin{equation}\label{eq:local-isotropy-coefficients}\sigma_p=\sum_{i=1}^r c_{pi}e_i,\quad c_{pi}\in C^\infty(U),\quad p=1,\dots,s.\end{equation}
Define on $\pi_{\bas}^{-1}(U)$,
\begin{equation}\label{eq:local-isotropy-representatives}
\widetilde Y_p\colonequals \sum_{i=1}^r(c_{pi}\circ \pi_{\bas})Y_i.
\end{equation}
We construct $\widetilde Y_p$ this way because the $C^\infty(W)$-module structure on $\Gamma(A)\cong l(X,\F)$ is implemented upstairs by pullback along $\pi_{\bas}$: 
\begin{equation}\label{eq:basic-module-action}
f\cdot\bar Y\longleftrightarrow(f\circ \pi_{\bas})Y.
\end{equation}
 Because $c_{pi}\circ \pi_{\bas}$ are basic and $Y_i$ are projectable, each $\widetilde Y_p$ belongs to $L(\pi_{\bas}^{-1}(U),\F)$. By construction, $\widetilde Y_p$ represents the section $\sigma_p$.

Since $\sigma_p\in \Gamma(U,\ker \rho_A)$, its anchor vanishes, so the projected vector field of $\widetilde Y_p$ on $U$ is zero. Equivalently, $d\pi_{\bas}(\widetilde Y_p)=0$, so for each $x\in L_w$ we have
\begin{equation}
\widetilde Y_p(x)\in \ker(d\pi_{\bas})_x=T_xL_w.
\end{equation}
Hence for each $w\in U$, $\widetilde Y_p|_{L_w}$ is a vector field on $L_w$. Since $\widetilde Y_p$ is projectable for $\F$, its restriction is projectable for $\F|_{L_w}$, as in the proof of Lemma \ref{lem:isotropy-structural-clean}.

By construction, $\operatorname{res}_w(\sigma_p(w))=\overline{\widetilde Y_p|_{L_w}}$. Because $\sigma_1,\dots,\sigma_s$ is a local frame of $\ker \rho_A|_U$, the vectors $\sigma_1(w),\dots,\sigma_s(w)$ form a basis of $(\ker\rho_A)_w$ for each $w\in U$. Then, by Lemma \ref{lem:isotropy-structural-clean}, the elements $\operatorname{res}_w(\sigma_1(w)),\dots,\operatorname{res}_w(\sigma_s(w))$ form a basis of $l(L_w,\F|_{L_w})$.
\end{proof}

\subsection{Descent and Maurer--Cartan forms}\label{subsec:descent-mc-forms}

We prove the following descent lemma to push bisection-induced maps down to the base $B$.

\begin{lemma}\label{lem:descend-to-B}
Let $\kappa:OF(M,\mathcal O)\rightarrow B$ be a fibre bundle whose fibres are the closures of the leaves of $\widetilde{\mathcal O}$, and assume that $\kappa$ is proper. Let $U,V\subset M$ be open and let $f:\pi^{-1}(U)\rightarrow \pi^{-1}(V)$ be a diffeomorphism sending leaves of $\widetilde{\mathcal O}|_{\pi^{-1}(U)}$ to leaves of $\widetilde{\mathcal O}|_{\pi^{-1}(V)}$. Let $B_U\colonequals \{b\in B\mid \kappa^{-1}(b)\subseteq \pi^{-1}(U)\}$, $B_V\colonequals \{b\in B\mid \kappa^{-1}(b)\subseteq \pi^{-1}(V)\}.$ Then $B_U$ and $B_V$ are open. 

For every $b\in B_U$, there exists a unique $b'\in B_V$ such that $f(\kappa^{-1}(b))=\kappa^{-1}(b')$. $b\mapsto b'$ defines a smooth diffeomorphism $\bar f:B_U\rightarrow B_V$ satisfying 
\begin{equation}\label{eq:descended-map}
\kappa\circ f=\bar f\circ \kappa
\end{equation}
 on $\kappa^{-1}(B_U)$.
\end{lemma}

\begin{proof}
Since $\kappa$ is proper, it is a closed map. Hence $B_U = B\setminus \kappa\bigl(OF(M,\mathcal O)\setminus \pi^{-1}(U)\bigr)$ is open, and similarly for $B_V$. 

Let $b\in B_U$ and choose $e\in \kappa^{-1}(b)$. Let $L$ be the leaf of $\widetilde{\mathcal O}$ through $e$. By definition of $\kappa$, $\overline L=\kappa^{-1}(b)$. Since $b\in B_U$, we have $\kappa^{-1}(b)\subseteq \pi^{-1}(U)$,  and therefore $L\subseteq \overline L=\kappa^{-1}(b)\subseteq \pi^{-1}(U)$. Hence $f(L)$ is defined. Since $L$ is a leaf of $\widetilde{\mathcal O}$ contained in $\pi^{-1}(U)$, it is also a leaf of $\widetilde{\mathcal O}|_{\pi^{-1}(U)}$, so $f(L)$ is a leaf of $\widetilde{\mathcal O}|_{\pi^{-1}(V)}$. Let $L''$ be the leaf of $\widetilde{\mathcal O}$ through $f(e)$. In the manifold topology of $L''$, the set $f(L)$ is a connected component of $L''\cap\pi^{-1}(V)$; in particular, $f(L)$ is open in $L''$.

We claim that $f(L)=L''$. Because $\kappa^{-1}(b)$ is compact and $f:\pi^{-1}(U)\rightarrow \pi^{-1}(V)$ is a homeomorphism onto its image, the set $f(\kappa^{-1}(b))$ is compact and therefore closed in $OF(M,\mathcal O)$. Since $L$ is dense in $\kappa^{-1}(b)$, the set $f(L)$ is dense in $f(\kappa^{-1}(b))$, so $\overline{f(L)}=f(\overline L)=f(\kappa^{-1}(b))\subseteq\pi^{-1}(V)$. As the inclusion of the leaf $L''$ into $OF(M,\mathcal O)$ is continuous, the closure of $f(L)$ in the manifold topology of $L''$ is contained in $\overline{f(L)}\cap L''\subseteq L''\cap\pi^{-1}(V)$; and since $f(L)$ is a connected component of $L''\cap\pi^{-1}(V)$, it is closed in $L''\cap\pi^{-1}(V)$, hence closed in $L''$. Being nonempty, open, and closed in the connected leaf $L''$, the set $f(L)$ equals $L''$. Therefore $\overline{f(L)}=\overline{L''}=\kappa^{-1}(b')$ for some $b'\in B$, and $f(\kappa^{-1}(b))=\kappa^{-1}(b')$; since $\kappa^{-1}(b')=f(\kappa^{-1}(b))\subseteq\pi^{-1}(V)$, we moreover have $b'\in B_V$.

Since $f(\kappa^{-1}(b))=\kappa^{-1}(b')$, the point $b'$ is uniquely determined by $b$. Thus $\bar f(b)\colonequals b'$ is well-defined and satisfies $\kappa\circ f=\bar f\circ\kappa$ on $\kappa^{-1}(B_U)$. Applying the same construction to $f^{-1}$ gives the inverse map $B_V\rightarrow B_U$.

Finally, $\bar f$ is smooth because $\kappa$ is a fibre bundle. For any $b\in B_U$, choose a local section $s:W\rightarrow OF(M,\mathcal O)$ of $\kappa$ with $b\in W\subseteq B_U$. Then on $W$ we have $\bar f = \kappa\circ f\circ s$, hence $\bar f$ is smooth. The same applies to $\bar f^{-1}$.
\end{proof}

\begin{lemma}\label{lem:bisection-pseudogroup}
Let $\kappa:OF(M,\mathcal O)\rightarrow B$ be the fibre bundle of Theorem \ref{thm:mol}, whose fibres are the closures of the leaves of $\widetilde{\mathcal O}$. Let $\delta:U\rightarrow \cG$ and $\tau:V\rightarrow \cG$ be local bisections such that $t(\delta(U))\subset V$. For any local bisection $\sigma:W\rightarrow \cG$, let $f_\sigma:\pi^{-1}(W)\rightarrow \pi^{-1}(t(\sigma(W)))$ be the induced map on $OF(M,\mathcal O)$, and let $B_W\colonequals \{b\in B\mid \kappa^{-1}(b)\subset \pi^{-1}(W)\}$.

Since $M$ and $O(q)$ are compact, $OF(M,\mathcal O)$ is compact; as $B$ is a Hausdorff manifold, $\kappa$ is proper. By Proposition \ref{prop:molino-data-invariant} and Lemma \ref{lem:descend-to-B}, $f_\sigma$ descends to a unique diffeomorphism $\bar f_\sigma:B_W\rightarrow B_{t(\sigma(W))}$ such that $\kappa\circ f_\sigma=\bar f_\sigma\circ\kappa$ on $\kappa^{-1}(B_W)$. Then:
\begin{enumerate}[label=\textup{(\arabic*)}]
\item On $\pi^{-1}(U)$, 
\begin{equation}\label{eq:bisection-frame-composition}
f_{\tau*\delta}=f_\tau\circ f_\delta,
\end{equation}
 where $\tau*\delta$ denotes the product bisection.
\item On $B_U$, 
\begin{equation}\label{eq:bisection-base-composition}
\bar f_{\tau*\delta}=\bar f_\tau\circ \bar f_\delta.
\end{equation}

\item For every $b\in B_U$ and every $A\in O(q)$, 
\begin{equation}\label{eq:bisection-base-equivariance}
\bar f_\delta(b\cdot A)=\bar f_\delta(b)\cdot A.
\end{equation}

\end{enumerate}
\end{lemma}

\begin{proof}
\textup{(1)} Let $e\in\pi^{-1}(U)$ and $x=\pi(e)$. Then
\begin{equation}
f_{\tau*\delta}(e)
=(\tau*\delta)(x)\cdot e
=\bigl(\tau(t(\delta(x)))\delta(x)\bigr)\cdot e
=\tau(t(\delta(x)))\cdot\bigl(\delta(x)\cdot e\bigr)
=f_\tau\bigl(f_\delta(e)\bigr),
\end{equation}
where we used that $\pi(\delta(x)\cdot e)=t(\delta(x))$ and $t(\delta(U))\subset V$.

\textup{(2)} Since $t(\delta(U))\subset V$, we have
$\bar f_\delta(B_U)\subset B_V$, so $\bar f_\tau\circ\bar f_\delta$ is
defined on $B_U$. For $e\in\kappa^{-1}(B_U)\subset\pi^{-1}(U)$,
\begin{equation}
\kappa\bigl(f_{\tau*\delta}(e)\bigr)
=\kappa\bigl(f_\tau(f_\delta(e))\bigr)
=\bar f_\tau\bigl(\kappa(f_\delta(e))\bigr)
=\bar f_\tau\bigl(\bar f_\delta(\kappa(e))\bigr)
=(\bar f_\tau\circ\bar f_\delta)\bigl(\kappa(e)\bigr).
\end{equation}
By uniqueness in Lemma \ref{lem:descend-to-B}, we have $\bar f_{\tau*\delta}=\bar f_\tau\circ\bar f_\delta$ on $B_U$.

\textup{(3)} Since $f_\delta$ is $O(q)$-equivariant and $\kappa$ is
$O(q)$-equivariant, for $e\in \kappa^{-1}(b)$ we have
\begin{equation}
\bar f_\delta(b\cdot A)
=\kappa\bigl(f_\delta(e\cdot A)\bigr)
=\kappa\bigl(f_\delta(e)\cdot A\bigr)
=\kappa\bigl(f_\delta(e)\bigr)\cdot A
=\bar f_\delta(b)\cdot A.
\end{equation}
\end{proof}

\begin{lemma}\label{lem:same-fibres-diffeo}
Let $X$ be a manifold, and let $p:X\rightarrow B$, $q:X\rightarrow W$ be surjective submersions onto Hausdorff manifolds. Assume that $p$ and $q$ have the same fibres. Then there exists a unique diffeomorphism $\Phi:B\xrightarrow{\ \cong\ }W$ such that $q=\Phi\circ p$.
\end{lemma}

\begin{proof}
Since $q$ is constant on the fibres of $p$, by the quotient property of $p$, there is a unique continuous map $\Phi:B\rightarrow W$ with $q=\Phi\circ p$. Similarly, there is a unique continuous map $\Psi:W\rightarrow B$ with $p=\Psi\circ q$. Then it is straightforward to check that $(\Psi\circ\Phi)\circ p=\Psi\circ q=p$, so $\Psi\circ\Phi=\id_B$ because $p$ is surjective. Likewise, $\Phi\circ\Psi=\id_W$. Thus $\Phi$ is a homeomorphism.

To prove smoothness, let $b\in B$. Since $p$ is a submersion, there is an open neighborhood $U$ of $b$ and a smooth local section $s:U\rightarrow X$ of $p$. On $U$, $\Phi|_U=q\circ s$, so $\Phi$ is smooth. Applying the same argument to $\Psi$, we find that $\Phi^{-1}$ is smooth. Hence $\Phi$ is a diffeomorphism.
\end{proof}


\begin{corollary}\label{cor:B-W-identification}
For $X=OF(M,\mathcal O)$ and $\F=\widetilde{\mathcal O}$, let $\kappa:X\rightarrow B$ be the fibre bundle of Theorem \ref{thm:mol} and let $\pi_{\bas}:X\rightarrow W$ be the basic fibre bundle. The two maps have the same fibres. Hence, by Lemma \ref{lem:same-fibres-diffeo}, there is a unique diffeomorphism
\begin{equation}
\Phi:B\xlongrightarrow{\ \cong\ }W
\end{equation}
such that 
\begin{equation}\label{eq:basic-base-identification}
\pi_{\bas}=\Phi\circ\kappa.
\end{equation}
 In what follows, whenever this setting is fixed, we identify $B$ and $W$ through this diffeomorphism and regard $A=b(X,\F)$ as a Lie algebroid over $B$.
\end{corollary}

\begin{proof}
By Theorem \ref{thm:mol}, the fibres of $\kappa$ are the closures of the leaves of $\F$. By \cite{MoerdijkMrcun:IFLG}*{Corollary 4.25}, the leaves of the basic foliation $\F_{\bas}$ are also the closures of the leaves of $\F$. Since $\pi_{\bas}$ is the quotient map by $\F_{\bas}$, its fibres are the leaves of $\F_{\bas}$. Thus $\kappa$ and $\pi_{\bas}$ have the same fibres, and the conclusion follows from Lemma \ref{lem:same-fibres-diffeo}.
\end{proof}

\begin{definition}[{\cite{MoerdijkMrcun:IFLG}*{Sec. 4.3.1}}]\label{def:mc}
Recall that a $\mathfrak g$-valued $1$-form $\eta$ is a Maurer--Cartan form if 
\begin{equation}\label{eq:MC-equation}
d\eta+\frac12[\eta,\eta]=0.
\end{equation}

\end{definition}

\begin{remark}[{\cite{MoerdijkMrcun:IFLG}*{Theorem 4.24}}]
If $(L,\mathcal F)$ is transversely parallelizable on a compact connected manifold and all basic functions on $(L,\mathcal F)$ are constant, then the evaluation maps $\ev_x \colon l(L,\mathcal F)\xrightarrow{\ \cong\ }N_x(\mathcal F)$ are isomorphisms, and the associated canonical Maurer--Cartan form is 
\begin{equation}\label{eq:canonical-MC-form}
(\omega_{\mathrm{MC}})_x(\xi)\colonequals \ev_x^{-1}(\pr_x^N(\xi)).
\end{equation}

\end{remark}


\begin{lemma}\label{lem:MC-naturality}
Let $(L_b,\mathcal F_b)$ and $(L_{b'},\mathcal F_{b'})$ be transversely parallelizable foliations on compact connected manifolds, and assume that all basic functions on $(L_b,\mathcal F_b)$ and $(L_{b'},\mathcal F_{b'})$ are constant. Let $\omega^{b}_{\mathrm{MC}}\in\Omega^1\bigl(L_b,l(L_b,\mathcal F_b)\bigr)$ and $\omega^{b'}_{\mathrm{MC}}\in\Omega^1\bigl(L_{b'},l(L_{b'},\mathcal F_{b'})\bigr)$ be the canonical Maurer--Cartan forms. If $\psi:(L_b,\mathcal F_b)\rightarrow(L_{b'},\mathcal F_{b'})$ is a foliated diffeomorphism, then
\begin{enumerate}[label=\textup{(\roman*)}]
\item $\psi_*$ induces a Lie algebra isomorphism $\psi':l(L_b,\mathcal F_b)\xrightarrow{\ \cong\ }l(L_{b'},\mathcal F_{b'})$, 
\begin{equation}\label{eq:transverse-pushforward}
\psi'(\bar Y)=\overline{\psi_*Y},
\end{equation}
 where $Y\in L(L_b,\mathcal F_b)$ is a projectable vector field representing $\bar Y\in l(L_b,\mathcal F_b)$.
\item 
\begin{equation}\label{eq:MC-naturality}
\psi^*(\omega^{b'}_{\mathrm{MC}})=\psi'\circ\omega^{b}_{\mathrm{MC}}.
\end{equation}

\end{enumerate}
\end{lemma}

\begin{proof}
(i) Since $\psi$ is foliated, $d\psi$ sends $T(\mathcal F_b)$ into $T(\mathcal F_{b'})$, hence $\psi_*$ sends $\mathfrak X(\mathcal F_b)$ into $\mathfrak X(\mathcal F_{b'})$. Let $Y\in L(L_b,\mathcal F_b)$ be projectable. For any $X'\in\mathfrak X(\mathcal F_{b'})$, let $X\colonequals (\psi^{-1})_*X'\in\mathfrak X(\mathcal F_b)$. Then
\begin{equation}\label{eq:projectable-pushforward}
[X',\psi_*Y]=[\psi_*X,\psi_*Y]=\psi_*[X,Y]\in \psi_*\mathfrak X(\mathcal F_b)\subseteq \mathfrak X(\mathcal F_{b'}),
\end{equation}
so $\psi_*Y\in L(L_{b'},\mathcal F_{b'})$. Therefore $\psi_*$ induces a map $\psi'\colon l(L_b,\mathcal F_b)\rightarrow l(L_{b'},\mathcal F_{b'})$ by passing to the quotient, $\psi'(\bar Y)\colonequals \overline{\psi_*Y}$.

It is a Lie algebra homomorphism because $\psi_*$ preserves Lie brackets. Its inverse is induced by $(\psi^{-1})_*$, hence $\psi'$ is a Lie algebra isomorphism.

(ii) Let $N(\mathcal F_b)\colonequals T(L_b)/T(\mathcal F_b)$ and $N(\mathcal F_{b'})\colonequals T(L_{b'})/T(\mathcal F_{b'})$ be the normal bundles, with projections $\pr^N\colon T(L_b)\rightarrow N(\mathcal F_b)$ and $\pr^{N'}\colon T(L_{b'})\rightarrow N(\mathcal F_{b'})$.

Since $\psi$ is foliated, $d\psi$ induces $(d\psi)^N_x\colon N_x(\mathcal F_b)\rightarrow N_{\psi(x)}(\mathcal F_{b'})$ by $(d\psi)^N_x([v])\colonequals [d\psi_x(v)]$. For $x\in L_b$, let $\mathrm{ev}_x\colon l(L_b,\mathcal F_b)\rightarrow N_x(\mathcal F_b)$, $\bar Y\mapsto \bar Y_x$, be the evaluation map, and similarly $\mathrm{ev}_{\psi(x)}$ for $(L_{b'},\mathcal F_{b'})$. Under our hypotheses, these are linear isomorphisms, and the canonical Maurer--Cartan forms are $(\omega^{b}_{\mathrm{MC}})_x(\xi)=\mathrm{ev}_x^{-1}\bigl(\pr^N_x(\xi)\bigr)$, $(\omega^{b'}_{\mathrm{MC}})_{\psi(x)}(\zeta)=\mathrm{ev}_{\psi(x)}^{-1}\bigl((\pr^{N'})_{\psi(x)}(\zeta)\bigr)$.

We claim that for all $x\in L_b$,
\begin{equation}\label{eq:ev-commutes-fixed}
\mathrm{ev}_{\psi(x)}\circ \psi'=(d\psi)^N_x\circ \mathrm{ev}_x.
\end{equation}
Take $\bar Y\in l(L_b,\mathcal F_b)$ represented by a projectable $Y$. Then
\begin{equation}
\begin{aligned}
\mathrm{ev}_{\psi(x)}(\psi'(\bar Y))
&=\mathrm{ev}_{\psi(x)}(\overline{\psi_*Y})
=[(\psi_*Y)_{\psi(x)}]
=[d\psi_x(Y_x)]
=(d\psi)^N_x([Y_x])
=(d\psi)^N_x(\mathrm{ev}_x(\bar Y)).
\end{aligned}
\end{equation}
Hence \eqref{eq:ev-commutes-fixed} holds. Since $\mathrm{ev}_x$ and $\mathrm{ev}_{\psi(x)}$ are isomorphisms, it follows that 
\begin{equation}\label{eq:inverse-evaluation-naturality}
\mathrm{ev}_{\psi(x)}^{-1}\circ (d\psi)^N_x=\psi'\circ \mathrm{ev}_x^{-1}.
\end{equation}

Now let $\xi\in T_x(L_b)$. Because $\psi$ is foliated, 
\begin{equation}\label{eq:normal-projection-naturality}
(\pr^{N'})_{\psi(x)}(d\psi_x(\xi))=(d\psi)^N_x(\pr^N_x(\xi)).
\end{equation}
 Therefore,
\begin{equation}
\begin{aligned}
(\psi^*\omega^{b'}_{\mathrm{MC}})_x(\xi)
&=(\omega^{b'}_{\mathrm{MC}})_{\psi(x)}(d\psi_x(\xi))\\
&=\mathrm{ev}_{\psi(x)}^{-1}\bigl((\pr^{N'})_{\psi(x)}(d\psi_x(\xi))\bigr)\\
&=\mathrm{ev}_{\psi(x)}^{-1}\bigl((d\psi)^N_x(\pr^N_x(\xi))\bigr)\\
&=\psi'\left(\mathrm{ev}_x^{-1}\bigl(\pr^N_x(\xi)\bigr)\right)\\
&=\psi'\bigl((\omega^{b}_{\mathrm{MC}})_x(\xi)\bigr),
\end{aligned}
\end{equation}
which is $\psi^*(\omega^{b'}_{\mathrm{MC}})=\psi'\circ\omega^{b}_{\mathrm{MC}}$.
\end{proof}


\begin{lemma}\label{lem:alpha-sigma}
Let $X\colonequals OF(M,\mathcal O)$ and $\F\colonequals \widetilde{\mathcal O}$, let $\kappa:X\rightarrow B$ be the fibre bundle of Theorem \ref{thm:mol}, identified with the basic fibre bundle $\pi_{\bas}:X\rightarrow W$ via Corollary \ref{cor:B-W-identification}, and let $A\colonequals b(X,\F)\rightarrow B$ be the basic Lie algebroid under this identification. Let $\sigma:U\rightarrow\cG$ be a local bisection, put $V\colonequals t(\sigma(U))$, and let $f_\sigma$ be the map in \eqref{eq:bisection-frame-map}. By Lemma \ref{lem:bisection-pseudogroup}, $f_\sigma$ descends to a diffeomorphism $\bar f_\sigma:B_U\rightarrow B_V$ satisfying $\kappa\circ f_\sigma=\bar f_\sigma\circ\kappa$ on $\kappa^{-1}(B_U)$. For $b\in B_U$, write $b'\colonequals \bar f_\sigma(b)$ and $f_{\sigma,b}\colonequals f_\sigma|_{L_b}:L_b\rightarrow L_{b'}$, and define
\begin{equation}\label{eq:alpha-sigma-definition}
\alpha_{\sigma,b}
\colonequals
\operatorname{res}_{b'}^{-1}
\circ
(f_{\sigma,b})_*
\circ
\operatorname{res}_b
:
(\ker\rho_A)_b\longrightarrow(\ker\rho_A)_{b'}.
\end{equation}
Here $(f_{\sigma,b})_*$ denotes the map on transverse vector fields defined by \eqref{eq:transverse-pushforward} with $\psi=f_{\sigma,b}$. Then each $\alpha_{\sigma,b}$ is a Lie algebra isomorphism, and these maps assemble into a smooth Lie algebra bundle isomorphism
\begin{equation}
\alpha_\sigma:(\ker\rho_A)|_{B_U}\xlongrightarrow{\ \cong\ }(\ker\rho_A)|_{B_V}
\end{equation}
covering $\bar f_\sigma$.
\end{lemma}

\begin{proof}
Since $f_\sigma$ preserves $\widetilde{\mathcal O}$ and sends $L_b$ to $L_{b'}$, the restriction $f_{\sigma,b}:L_b\rightarrow L_{b'}$ is a foliated diffeomorphism from $(L_b,\widetilde{\mathcal O}|_{L_b})$ to $(L_{b'},\widetilde{\mathcal O}|_{L_{b'}})$. By Lemma \ref{lem:MC-naturality}(i), it induces an isomorphism of structural Lie algebras
\begin{equation}
(f_{\sigma,b})_*:
 l(L_b,\widetilde{\mathcal O}|_{L_b})\longrightarrow
 l(L_{b'},\widetilde{\mathcal O}|_{L_{b'}}).
\end{equation}
Composing with the restriction isomorphisms of Lemma \ref{lem:isotropy-structural-clean} gives the Lie algebra isomorphism $\alpha_{\sigma,b}$.

We now check smoothness. Let $\cW\subseteq B_U$ be open and let $\zeta$ be a smooth local section of $\ker\rho_A$ over $\cW$. By the construction of $A$ in Proposition \ref{prop:basic-Lie-algebroid}, write $\zeta=\sum_i f_i\,e_i$ with $f_i\in C^\infty(\cW)$ in the trivialization $A\cong W\times \mathbb R^r$ determined by the transverse parallelism $\{\bar Y_1,\dots,\bar Y_r\}$; then the projectable vector field $Y\colonequals \sum_i (f_i\circ\pi_{\bas})\,Y_i$ on $\kappa^{-1}(\cW)$, where $Y_i\in L(X,\F)$ is the chosen projectable representative of $\bar Y_i$, represents $\zeta$. Since $\zeta$ takes values in $\ker\rho_A$, by the definition of the anchor, we have
\begin{equation}
d\pi_{\bas}(Y)=\rho_A(\zeta)\circ\pi_{\bas}=0.
\end{equation}
After identifying $B$ and $W$, this is $d\kappa(Y)=0$. Since $f_\sigma$ is a foliated diffeomorphism, $(f_\sigma)_*Y$ is again projectable. Moreover, since
\begin{equation}
\kappa\circ f_\sigma=\bar f_\sigma\circ \kappa,
\end{equation}
we have
\begin{equation}
d\kappa((f_\sigma)_*Y)
=
d\bar f_\sigma(d\kappa(Y))=0.
\end{equation}
Therefore $(f_\sigma)_*Y$ represents a smooth local section of $\ker\rho_A$ over $\bar f_\sigma(\cW)$, and its value at $b'=\bar f_\sigma(b)$ is
\begin{equation}
\operatorname{res}_{b'}^{-1}\Bigl(\overline{\bigl((f_\sigma)_*Y\bigr)\big|_{L_{b'}}}\Bigr)=\operatorname{res}_{b'}^{-1}\Bigl((f_{\sigma,b})_*\bigl(\overline{Y|_{L_b}}\bigr)\Bigr)=\alpha_{\sigma,b}\bigl(\zeta(b)\bigr).
\end{equation}
Hence $\alpha_\sigma$ sends smooth local sections to smooth local sections, so it is a smooth bundle map covering $\bar f_\sigma$ which is fibrewise a Lie algebra isomorphism. Applying the same construction to the inverse bisection $\sigma^{-1}$ yields the inverse bundle map, which is smooth by the same argument. Therefore $\alpha_\sigma$ is a smooth Lie algebra bundle isomorphism.
\end{proof}

\begin{proposition}\label{prop:vertical-MC}
In the situation of Lemma \ref{lem:alpha-sigma}, for $b\in B$ write $L_b\colonequals \kappa^{-1}(b)$, and let $\omega^{b}_{\mathrm{MC}}$ denote the canonical Maurer--Cartan form of the Lie foliation $(L_b,\F|_{L_b})$, valued in $l(L_b,\F|_{L_b})$.
\begin{enumerate}[label=\textup{(\roman*)}]
\item There exists a unique smooth section
\begin{equation}
\omega^{\kappa}_{\mathrm{MC}} \in \Gamma\bigl((\ker d\kappa)^*\otimes \kappa^*(\ker\rho_A)\bigr)
\end{equation}
such that, for every $b\in B$, every $x\in L_b$, and every $\xi\in T_xL_b=\ker(d\kappa)_x$,
\begin{equation}\label{eq:MC-res}
\operatorname{res}_b \bigl((\omega^{\kappa}_{\mathrm{MC}})_x(\xi)\bigr)=(\omega^{b}_{\mathrm{MC}})_x(\xi).
\end{equation}
\item The form $\omega^{\kappa}_{\mathrm{MC}}$ is equivariant under local bisections: for every local bisection $\sigma:U\rightarrow\cG$, every $b\in B_U$, every $x\in L_b$, and every $\xi\in T_xL_b$,
\begin{equation}\label{eq:vertical-MC-equivariance}
(\omega^{\kappa}_{\mathrm{MC}})_{f_\sigma(x)}\bigl((df_\sigma)_x\xi\bigr)
=
\alpha_{\sigma,b}\bigl((\omega^{\kappa}_{\mathrm{MC}})_x(\xi)\bigr).
\end{equation}
\end{enumerate}
\end{proposition}

\begin{proof}
Recall that $X$ is compact, since $M$ and $O(q)$ are, and that $(X,\F)$ is transversely parallelizable (Lemma \ref{lem:theta-omega-parallelism}).

(i) For $x\in X$, let $b\colonequals \kappa(x)$. By \eqref{eq:basic-evaluation} and \eqref{eq:basic-evaluation-anchor}, evaluation at $x$ induces a linear isomorphism $\operatorname{ev}_x^X:A_b\rightarrow N_x(\F)$, $[\bar Y]\mapsto \bar Y_x$, and, under this identification, the anchor $\rho_{A,b}:A_b\rightarrow T_bB$ corresponds to the map induced by $d\kappa:N_x(\F)\rightarrow T_bB$. Using local frames of $A$ coming from a transverse parallelism, these fibrewise maps assemble to a smooth vector bundle isomorphism $\operatorname{ev}^X:\kappa^*A\xrightarrow{\ \cong\ }N(\F)$, which identifies $\kappa^*(\ker\rho_A)$ with $\ker\bigl(d\kappa:N(\F)\rightarrow \kappa^*(TB)\bigr)$.

Since $T(\F)\subseteq \ker d\kappa$, the quotient map $TX\rightarrow N(\F)=TX/T(\F)$ restricts to
\begin{equation}\label{eq:vertical-normal-quotient}
q:\ker d\kappa\longrightarrow \ker\bigl(d\kappa:N(\F)\longrightarrow \kappa^*(TB)\bigr).
\end{equation}
Define
\begin{equation}\label{eq:vertical-MC-definition}
\omega^{\kappa}_{\mathrm{MC}}\colonequals \bigl(\operatorname{ev}^X|_{\kappa^*(\ker\rho_A)}\bigr)^{-1}\circ q.
\end{equation}
So at $x\in X$, if $\xi\in \ker(d\kappa)_x$, then $(\omega_{\mathrm{MC}}^{\kappa})_x(\xi)\in (\ker\rho_A)_b$. Because $q$ is smooth and $\operatorname{ev}^X|_{\kappa^*(\ker\rho_A)}$ is a smooth vector bundle isomorphism, the composition is smooth. So we have a smooth section of $(\ker d\kappa)^*\otimes \kappa^*(\ker\rho_A)$.

Now let $b\in B$, write $L_b=\kappa^{-1}(b)$, and note that, since $\kappa$ is constant on $L_b$, the pullback coefficient bundle restricts to the trivial bundle
\begin{equation}\label{eq:MC-coefficient-fibre}
\kappa^*(\ker \rho_A)|_{L_b}\cong L_b\times (\ker \rho_A)_b,
\end{equation}
hence $\omega_{\mathrm{MC}}^{\kappa}|_{L_b}$ is a $(\ker \rho_A)_b$-valued $1$-form on $L_b$.

For $x\in L_b$, let $\operatorname{ev}_x^{L_b}:l(L_b,\F|_{L_b})\rightarrow N_x(\F|_{L_b})$ be the evaluation isomorphism. By Lemma \ref{lem:isotropy-structural-clean}, $(L_b,\F|_{L_b})$ is a Lie foliation and $\operatorname{res}_b:(\ker\rho_A)_b\xrightarrow{\ \cong\ }l(L_b,\F|_{L_b})$ is a Lie algebra isomorphism. By \eqref{eq:restriction-evaluation}, $\operatorname{ev}_x^{L_b}\circ \operatorname{res}_b = \operatorname{ev}_x^X\big|_{(\ker\rho_A)_b}$.

For $\xi\in T_xL_b=\ker(d\kappa)_x$, let $\bar\xi$ denote the class of $\xi$ in $N_x(\F|_{L_b})=T_xL_b/T_x(\F)$. Then, by construction, $\operatorname{ev}_x^X((\omega_{\mathrm{MC}}^{\kappa})_x(\xi))=q_x(\xi)$, and under the identification $\ker\bigl(d\kappa:N_x(\F)\rightarrow T_bB\bigr)\cong N_x(\F|_{L_b})$ this is $\bar\xi$. Hence $\operatorname{ev}_x^{L_b}(\operatorname{res}_b((\omega_{\mathrm{MC}}^{\kappa})_x(\xi)))=\bar\xi$.

Since $\operatorname{ev}_x^{L_b}$ is an isomorphism, it follows that
\begin{equation}
\operatorname{res}_b \bigl((\omega^{\kappa}_{\mathrm{MC}})_x(\xi)\bigr)=(\operatorname{ev}_x^{L_b})^{-1}(\bar\xi).
\end{equation}
By \eqref{eq:canonical-MC-form}, the right hand side is $(\omega^b_{\mathrm{MC}})_x(\xi)$. Hence \eqref{eq:MC-res} holds.

Uniqueness is immediate: since each $\operatorname{res}_b$ is an isomorphism, \eqref{eq:MC-res} determines the value $(\omega^{\kappa}_{\mathrm{MC}})_x(\xi)$ for every $x\in X$ and every $\xi\in\ker(d\kappa)_x$.

(ii) Since $f_{\sigma,b}:L_b\rightarrow L_{b'}$ is a foliated diffeomorphism (Lemma \ref{lem:alpha-sigma}), we have $f_\sigma(x)\in L_{b'}$ and $(df_\sigma)_x\xi\in T_{f_\sigma(x)}L_{b'}=\ker(d\kappa)_{f_\sigma(x)}$. We first observe that both sides of \eqref{eq:vertical-MC-equivariance} lie in $(\ker\rho_A)_{b'}$: since $\kappa(f_\sigma(x))=\bar f_\sigma(\kappa(x))=b'$, the coefficient fibre $\bigl(\kappa^*(\ker\rho_A)\bigr)_{f_\sigma(x)}$ is $(\ker\rho_A)_{b'}$, so the left-hand side lies in $(\ker\rho_A)_{b'}$; the right hand side does as well, since $\alpha_{\sigma,b}$ takes values in $(\ker\rho_A)_{b'}$.

Applying $\operatorname{res}_{b'}$, we obtain
\begin{equation}
\begin{aligned}
\operatorname{res}_{b'}\Bigl((\omega^{\kappa}_{\mathrm{MC}})_{f_\sigma(x)}
  \bigl((df_\sigma)_x\xi\bigr)\Bigr)
&=(\omega^{b'}_{\mathrm{MC}})_{f_\sigma(x)}
  \bigl((df_\sigma)_x\xi\bigr)\\
&=(f_{\sigma,b})_*\bigl((\omega^{b}_{\mathrm{MC}})_x(\xi)\bigr)\\
&=(f_{\sigma,b})_*\operatorname{res}_b
  \bigl((\omega^{\kappa}_{\mathrm{MC}})_x(\xi)\bigr)\\
&=\operatorname{res}_{b'}\Bigl(\alpha_{\sigma,b}
  \bigl((\omega^{\kappa}_{\mathrm{MC}})_x(\xi)\bigr)\Bigr).
\end{aligned}
\end{equation}
The first and third equalities follow from \eqref{eq:MC-res} on $L_{b'}$ and $L_b$, respectively. For the second equality, apply Lemma \ref{lem:MC-naturality}\textup{(ii)} to $f_{\sigma,b}$. We verified the hypotheses on the two restricted foliations in Step 4 of the proof of Lemma \ref{lem:isotropy-structural-clean}. The last equality follows from \eqref{eq:alpha-sigma-definition}. Since $\operatorname{res}_{b'}$ is injective, the equivariance identity follows.
\end{proof}

\begin{remark}
A priori, each fibre $L_b$ carries its own Maurer--Cartan form, valued in the structural Lie algebra identified with $(\ker \rho_A)_b$. By Proposition \ref{prop:vertical-MC}, these forms fit together into the smooth section $\omega^\kappa_{\mathrm{MC}}$, uniquely determined by \eqref{eq:MC-res}, with coefficient bundle $\kappa^*(\ker \rho_A)$. This section is equivariant under the maps induced by local bisections of $\cG$. Here the term Maurer--Cartan is meant fibrewise: for every $b\in B$, the $l(L_b,\F|_{L_b})$-valued form $\operatorname{res}_b\circ\omega^\kappa_{\mathrm{MC}}|_{L_b}$ is $\omega^b_{\mathrm{MC}}$ and satisfies the ordinary Maurer--Cartan equation on $L_b$.
\end{remark}

\section{The Lie algebroid of \texorpdfstring{$\cG$}{G} and groupoid refinements of Molino's structures}\label{sec:groupoid-lie-algebroid}

\subsection{Construction of \texorpdfstring{$A_{\cG}=\Lie(\cG)$}{AG = Lie(G)}}
\label{subsec:construction-AG}

We return to the Lie algebroid $A_{\cG}\colonequals\operatorname{Lie}(\cG)\rightarrow M$ of the original groupoid $\cG\rightrightarrows M$. The basic Lie algebroid $A=b(OF(M,\mathcal O),\widetilde{\mathcal O})$ is defined over the leaf closure base $B$.

For each $x\in M$, the source fibre $s^{-1}(x)\subset \cG$ is the manifold of arrows whose source is $x$, and its distinguished point is the unit arrow $1_x\in \cG$. The fibre of the Lie algebroid at $x$ is $(A_{\cG})_x=T_{1_x}(s^{-1}(x))=\ker(ds)_{1_x}$. See \cite{IntLieBrac}*{Def. 1.19}. So an element of $(A_{\cG})_x$ is an infinitesimal arrow leaving $x$. Equivalently, $A_{\cG}$ is the pullback of the source-vertical bundle $\ker(ds)\subset T\cG$ along the unit map $u:M\rightarrow \cG$, $u(x)=1_x$.

Let $T^s\cG\colonequals \ker(ds)\subset T\cG$.

\begin{proposition}[cf. {\cite{MoerdijkMrcun:IFLG}*{Prop. 6.1 and the following discussion}}]\label{prop:Lie-algebroid-of-groupoid}
Let $\mathfrak X_{\mathrm{inv}}^s(\cG)$ be the space of right invariant source-vertical vector fields on $\cG$,
\begin{equation}\label{eq:right-invariant-fields}
\mathfrak X_{\mathrm{inv}}^s(\cG)\colonequals \{X\in \mathfrak X(\cG)\mid X_g\in \ker(ds)_g\quad \forall g\in \cG,\ X_{hg}=dR_g(X_h)\quad \forall (h,g)\text{ composable}\}.
\end{equation}
Then:
\begin{enumerate}[label=\textup{(\roman*)}]
\item $\mathfrak X_{\mathrm{inv}}^s(\cG)$ is a Lie subalgebra of $\mathfrak X(\cG)$.

\item Restriction to the units identifies $\mathfrak X_{\mathrm{inv}}^s(\cG)$ with $\Gamma(A_{\cG})$.
\end{enumerate}
\end{proposition}

\begin{remark}
For $a\in\Gamma(A_{\cG})$, its right invariant extension is $(a^r)_g\colonequals dR_g(a_{t(g)})$, and restriction to the units is inverse to $a\mapsto a^r$.
\end{remark}

\begin{proposition}[{\cite{IntLieBrac}*{Def. 1.21, Def. 1.23, Prop. 1.24}}]
Under the identification $\Gamma(A_{\cG})\cong \mathfrak X_{\mathrm{inv}}^s(\cG)$, the bracket on $\Gamma(A_{\cG})$ is induced from the Lie bracket of right invariant source-vertical vector fields. More precisely, if $a,b\in \Gamma(A_{\cG})$ with right invariant extensions $a^r,b^r$, then
\begin{equation}\label{eq:groupoid-bracket}
[a,b]^r=[a^r,b^r].
\end{equation}

The anchor is induced by the derivative of the target map:
\begin{equation}\label{eq:groupoid-anchor}
\rho_{\cG}:A_{\cG}\longrightarrow TM,\qquad \rho_{\cG}(\xi)=dt(\xi).
\end{equation}
Equivalently, if $a\in \Gamma(A_{\cG})$ and $a^r$ is its right invariant extension, then $a^r$ is $t$-projectable to $\rho_{\cG}(a)$, that is, 
\begin{equation}\label{eq:right-invariant-target-projection}
dt(a^r)=\rho_{\cG}(a)\circ t.
\end{equation}

Moreover, for $a,b\in \Gamma(A_{\cG})$ and $f\in C^\infty(M)$, the Leibniz rule is
\begin{equation}\label{eq:groupoid-Leibniz}
[a,fb]=f[a,b]+\rho_{\cG}(a)(f)b.
\end{equation}
\end{proposition}

\begin{proof}
The bracket part is immediate from the preceding discussion.

For the anchor, define $\rho_{\cG}:A_{\cG}\rightarrow TM$ by $\rho_{\cG}(\xi)=dt(\xi)$. Take $X\in \mathfrak X_{\mathrm{inv}}^s(\cG)$, and let $g:x\rightarrow y$. Since $g=1_yg$, we have $X_g=dR_g(X_{1_y})$. We then compute
\begin{equation}
\begin{aligned}
    dt(X_g)=dt(dR_g(X_{1_y}))&=d(t\circ R_g)(X_{1_y})\\
    &=dt(X_{1_y}),
\end{aligned}
\end{equation}
because for any $h\in s^{-1}(y)$, $t(R_g(h))=t(hg)=t(h)$. We note that the right hand side only depends on $y=t(g)$. Therefore $X$ is $t$-projectable. If we define $Y\in \mathfrak X(M)$ by $Y_y\colonequals dt(X_{1_y})$, then $dt(X_g)=Y_{t(g)}$ for all $g\in \cG$.

Now let $a\in \Gamma(A_{\cG})$ and let $a^r$ be its right invariant extension. By construction, $(a^r)_{1_x}=a_x$ for all $x\in M$, so the vector field $Y$ above is exactly $\rho_{\cG}(a)$. Hence $a^r$ is $t$-projectable to $\rho_{\cG}(a)$, that is, $dt(a^r)=\rho_{\cG}(a)\circ t$.

Next we discuss the Leibniz rule. Let $a,b\in \Gamma(A_{\cG})$ and $f\in C^\infty(M)$. We compute
\begin{equation}
\begin{aligned}
    [a,fb]^r&=[a^r,(fb)^r]\\
    &=[a^r,(f\circ t)b^r]\quad\text{since }(fb)^r=(f\circ t)b^r\\
    &=(f\circ t)[a^r,b^r]+a^r(f\circ t)b^r\\
    &=(f\circ t)[a^r,b^r]+(\rho_{\cG}(a)(f)\circ t)b^r\quad\text{because }a^r\text{ projects to }\rho_{\cG}(a)\\
    &=(f[a,b])^r+(\rho_{\cG}(a)(f)b)^r.
\end{aligned}
\end{equation}
Since the right invariant extension map is injective, we conclude $[a,fb]=f[a,b]+\rho_{\cG}(a)(f)b$. \end{proof}

\begin{remark}
Since $a^r$ and $b^r$ are $t$-related to $\rho_{\cG}(a)$ and $\rho_{\cG}(b)$, the anchor is automatically a Lie algebra homomorphism: $\rho_{\cG}([a,b])=[\rho_{\cG}(a),\rho_{\cG}(b)]$. At this point we have the full Lie algebroid structure attached to the groupoid: $(A_{\cG})_x=\ker(ds)_{1_x}$, $\rho_{\cG}=dt|_{A_{\cG}}$, with bracket induced from right invariant source-vertical vector fields. When $\cG$ is a Lie group, this right invariant bracket on $T_e\cG$ is the negative of the usual bracket defined by left-invariant vector fields.
\end{remark}


\subsection{How \texorpdfstring{$A_{\cG}$}{AG} recovers the orbit foliation}

We now explain how $A_{\cG}=\operatorname{Lie}(\cG)$ recovers the orbit foliation $\mathcal O$. 

\begin{proposition}\label{prop:LieG-recovers-orbit}
Let $\cG\rightrightarrows M$ be a regular Lie groupoid, let $A_{\cG}=\operatorname{Lie}(\cG)\rightarrow M$ be its Lie algebroid, and let $\rho_{\cG}:A_{\cG}\rightarrow TM$ be the anchor. For each $x\in M$, let $O_x$ be the orbit through $x$ and let $\cG_x\colonequals s^{-1}(x)\cap t^{-1}(x)$ be the isotropy group at $x$. We use the right invariant bracket on $\operatorname{Lie}(\cG_x)$. Then $\operatorname{Im}(\rho_{\cG,x})=T_x(O_x)$ and 
\begin{equation}\label{eq:groupoid-isotropy-Lie-algebra}
(\ker\rho_{\cG})_x=\operatorname{Lie}(\cG_x).
\end{equation}
 Consequently,
\begin{equation}\label{eq:groupoid-anchor-image}
\operatorname{Im}(\rho_{\cG})=T(\mathcal O),
\end{equation}
and there is a short exact sequence of vector bundles
\begin{equation}
0\longrightarrow \ker \rho_{\cG}\longrightarrow A_{\cG}\xlongrightarrow{\rho_{\cG}} T(\mathcal O)\longrightarrow 0.
\end{equation}
\end{proposition}

\begin{proof}
We first show that the anchor recovers the orbit directions. Let $x\in M$. By Lemma \ref{lem:orbit-submersion}, the restriction $t_x\colonequals t|_{s^{-1}(x)}:s^{-1}(x)\rightarrow O_x$ is a surjective submersion. Differentiating at the unit $1_x$, and using $(A_{\cG})_x=T_{1_x}(s^{-1}(x))=\ker(ds)_{1_x}$, we get $d(t_x)_{1_x}:(A_{\cG})_x\rightarrow T_x(O_x)$. But $d(t_x)_{1_x}=dt_{1_x}|_{\ker(ds)_{1_x}}=\rho_{\cG,x}$ since $t_x$ is the restriction of $t$. Hence $\operatorname{Im}(\rho_{\cG,x})=T_x(O_x)$.

Since $\cG$ is regular, the orbit foliation $\mathcal O$ has locally constant rank, so these spaces fit together into the tangent bundle $T(\mathcal O)\subset TM$. Thus $\operatorname{Im}(\rho_{\cG})=T(\mathcal O)$.

Next we compute the isotropy of $A_{\cG}$. The isotropy group at $x$ is $\cG_x=s^{-1}(x)\cap t^{-1}(x)$. Since $\cG_x=t_x^{-1}(x)$ and $t_x$ is a submersion, the tangent space of $\cG_x$ at the unit is
\begin{equation}
T_{1_x}\cG_x=T_{1_x}(s^{-1}(x)\cap t^{-1}(x))=\ker(ds)_{1_x}\cap \ker(dt)_{1_x}.
\end{equation}
But $(A_{\cG})_x=\ker(ds)_{1_x}$ and $\rho_{\cG,x}=dt_{1_x}|_{\ker(ds)_{1_x}}$. Hence
\begin{equation}
(\ker\rho_{\cG})_x=\ker(ds)_{1_x}\cap \ker(dt)_{1_x}=T_{1_x}\cG_x=\operatorname{Lie}(\cG_x).
\end{equation}

Since $\operatorname{Im}(\rho_{\cG})=T(\mathcal O)$ is a vector subbundle of $TM$, the anchor has locally constant rank. Therefore we obtain the exact sequence
\begin{equation}
0\longrightarrow \ker \rho_{\cG}\longrightarrow A_{\cG}\xlongrightarrow{\rho_{\cG}} T(\mathcal O)\longrightarrow 0.
\end{equation}
\end{proof}


\subsection{Normal representation and the adjoint representation up to homotopy}\label{subsec:normal-adjoint-homotopy}

\begin{lemma}\label{lem:ehresmann-connection}
Let $\cG\rightrightarrows M$ be a Lie groupoid with $M$ Hausdorff. If $\cG$ admits a smooth Riemannian metric, then it admits an Ehresmann connection $\Sigma:s^*TM\rightarrow T\cG$ satisfying 
\begin{equation}\label{eq:ehresmann-splitting}
ds\circ\Sigma=\id_{s^*TM}
\end{equation}
 and 
\begin{equation}\label{eq:ehresmann-unit}
\Sigma_{1_x}=(du)_x
\end{equation}
 for every $x\in M$.
\end{lemma}

\begin{proof}
Let $\eta^{(1)}$ be a Riemannian metric on $\cG$. The orthogonal complement $(\ker ds)^{\perp_{\eta^{(1)}}}$ determines a smooth right splitting $\Sigma:s^*TM\rightarrow T\cG$ satisfying $ds\circ\Sigma=\id_{s^*TM}$. At the units, since $ds_{1_x}\circ du_x=\id_{T_xM}$,
\begin{equation}
du_x-\Sigma_{1_x}:T_xM\longrightarrow A_{\cG,x}
\end{equation}
defines a smooth section of $\Hom(\pr_1^*TM,\pr_2^*A_{\cG})\rightarrow M\times M$ along the diagonal. Since $M$ is Hausdorff, the diagonal is closed. By \cite{Lee:ISM}*{Lemma 10.12}, there is a global smooth section $C$ satisfying $C(x,x)=du_x-\Sigma_{1_x}$. For $g\in\cG$ and $v\in T_{s(g)}M$, consider the smooth bundle map
\begin{equation}\label{eq:ehresmann-correction}
(g,v)\longmapsto\Sigma_g(v) +(dR_g)_{1_{t(g)}}\bigl(C(s(g),t(g))(v)\bigr).
\end{equation}
The second summand lies in $\ker(ds)_g$, so the map in \eqref{eq:ehresmann-correction} is a right splitting of $ds$. At $g=1_x$, its value on $v\in T_xM$ is $\Sigma_{1_x}(v)+C(x,x)(v)=du_x(v)$. It is therefore an Ehresmann connection in the sense of \cite{AriasAbadCrainic:RepUpToHomotopyGroupoids}*{Definition 2.8}.
\end{proof}

\begin{proposition}[Normal representation and the adjoint representation up to homotopy]\label{prop:normal-adjoint-cohomology}
Let $\cG\rightrightarrows M$ be a regular Lie groupoid whose arrow manifold admits a Riemannian metric, let $A_{\cG}\colonequals \Lie(\cG)$ be its Lie algebroid with anchor $\rho_{\cG}:A_{\cG}\rightarrow TM$, let $\mathcal O$ be the orbit foliation, and set $N\colonequals TM/T\mathcal O$. For the adjoint representation up to homotopy
\begin{equation}
\operatorname{Ad}(\cG)=\bigl(A_{\cG}\xlongrightarrow{\rho_{\cG}}TM\bigr),
\end{equation}
the induced representation on the degree-one cohomology bundle is canonically the normal representation $\lambda^N:\cG\curvearrowright N$.

Equivalently, under the canonical identification $H^1(\operatorname{Ad}(\cG)) \cong TM/\rho_{\cG}(A_{\cG})=TM/T\mathcal O=N$, the degree-one cohomology representation of $\operatorname{Ad}(\cG)$ is $\lambda^N$.
\end{proposition}

\begin{proof}

By Lemma \ref{lem:ehresmann-connection}, choose an Ehresmann connection $\Sigma:s^*TM\rightarrow T\mathcal G, ds\circ\Sigma=\id_{s^*TM}$, $\Sigma_{1_x}=(du)_x$ for every $x\in M$. Recall the construction of the adjoint representation up to homotopy from \cite{AriasAbadCrainic:RepUpToHomotopyGroupoids}. The adjoint complex of $\cG$ is $\operatorname{Ad}(\cG)=\bigl(A_{\cG}\xrightarrow{\rho_{\cG}}TM\bigr)$, with $A_{\cG}$ in degree $0$ and $TM$ in degree $1$ \cite{AriasAbadCrainic:RepUpToHomotopyGroupoids}*{Definition 3.13}. The connection $\Sigma$ determines quasi-actions of $\cG$ on $TM$ and on $A_{\cG}$, both denoted $\lambda^\Sigma$, together with a basic curvature
\begin{equation}
K^{\mathrm{bas}}_{\Sigma}\in\Gamma\bigl(\cG^{(2)};\Hom(s^*TM,t^*A_{\cG})\bigr).
\end{equation}
Let $r_g=(dR_g)_{1_{t(g)}}:A_{\mathcal G,t(g)}\rightarrow\ker(ds)_g$.
We use the following sign convention for the basic curvature:
\begin{equation}\label{eq:basic-curvature-definition}
K^{\mathrm{bas}}_\Sigma(g,h)(v)=r_{gh}^{-1}\Bigl(dm_{g,h}\bigl(\Sigma_g(\lambda_h^\Sigma(v)),\Sigma_h(v)\bigr)-\Sigma_{gh}(v)\Bigr), \qquad v\in T_{s(h)}M,
\end{equation}
for composable arrows $g,h$. In \cite{AriasAbadCrainic:RepUpToHomotopyGroupoids}*{Definition 2.11}, the quasi-action on $TM$ is defined by:
\begin{equation}\label{eq:adjoint-tangent-quasiaction}
\lambda_g^\Sigma(v)=dt_g\bigl(\Sigma_g(v)\bigr),\qquad v\in T_{s(g)}M.
\end{equation}
The structure operators
\begin{equation}
R_0=\rho_{\cG},\qquad R_1=\lambda^\Sigma,\qquad R_2=K^{\mathrm{bas}}_{\Sigma}
\end{equation}
give $\operatorname{Ad}(\cG)$ the structure of a unital representation up to homotopy, denoted $\operatorname{Ad}_\Sigma(\cG)$ \cite{AriasAbadCrainic:RepUpToHomotopyGroupoids}*{Proposition 3.14}, and the resulting object is independent of $\Sigma$ up to canonical isomorphism \cite{AriasAbadCrainic:RepUpToHomotopyGroupoids}*{Proposition 3.16}. Among the structure equations \cite{AriasAbadCrainic:RepUpToHomotopyGroupoids}*{Proposition 2.15}, we use the anchor equivariance
\begin{equation}\label{eq:anchor-equivariance}
\rho_{\cG}\bigl(\lambda^\Sigma_g(a)\bigr)=\lambda^\Sigma_g\bigl(\rho_{\cG}(a)\bigr),
\qquad a\in A_{\cG,s(g)},
\end{equation}
and the composition identity
\begin{equation}\label{eq:basic-curvature-defect}
\lambda^\Sigma_g\lambda^\Sigma_h(v)-\lambda^\Sigma_{gh}(v)=\rho_{\cG}\bigl(K^{\mathrm{bas}}_{\Sigma}(g,h)(v)\bigr),
\qquad v\in T_{s(h)}M,
\end{equation}
for composable arrows $g,h$ with $s(g)=t(h)$, so that the product is $gh$.

Since $\cG$ is regular, the anchor has locally constant rank and its image is the orbit distribution: $\rho_{\cG}(A_{\cG})=T\mathcal O$. Hence the complex of $\operatorname{Ad}(\cG)$ is regular in the sense of \cite{AriasAbadCrainic:RepUpToHomotopyGroupoids}*{Definition 3.31}, and its degree-one cohomology bundle is
\begin{equation}
H^1(\operatorname{Ad}(\cG))=TM/\rho_{\cG}(A_{\cG})=TM/T\mathcal O=N.
\end{equation}
By \cite{AriasAbadCrainic:RepUpToHomotopyGroupoids}*{Theorem 3.32}, the quasi-action $\lambda^\Sigma$ induces an ordinary representation of $\cG$ on $H^1(\operatorname{Ad}(\cG))\cong N$, the degree-one cohomology representation of $\operatorname{Ad}(\cG)$. We verify this directly from the structure equations and identify the resulting representation with the canonical normal representation.

Define 
\begin{equation}\label{eq:adjoint-quotient-action}
\lambda_g([v])\colonequals [\lambda_g^\Sigma(v)]\in N_{t(g)}
\end{equation}
 for $[v]\in N_{s(g)}$. This is well defined: if $v,v'\in T_{s(g)}M$ represent the same class, then $v-v'\in T_{s(g)}\mathcal O=\rho_{\cG}(A_{\cG,s(g)})$, so $v-v'=\rho_{\cG}(a)$ for some $a\in A_{\cG,s(g)}$, and by \eqref{eq:anchor-equivariance},
\begin{equation}
\lambda^\Sigma_g(v)-\lambda^\Sigma_g(v')=\lambda^\Sigma_g\bigl(\rho_{\cG}(a)\bigr)=\rho_{\cG}\bigl(\lambda^\Sigma_g(a)\bigr)\in T_{t(g)}\mathcal O.
\end{equation}
It is multiplicative: for composable arrows $g,h$, by \eqref{eq:basic-curvature-defect}, we have $\lambda^\Sigma_g\lambda^\Sigma_h(v)-\lambda^\Sigma_{gh}(v)\in\rho_{\cG}(A_{\cG})=T\mathcal O$, so $\lambda_g\lambda_h=\lambda_{gh}$ on $N$.

The induced map on $N$ is independent of the connection: if $\Sigma'$ is another Ehresmann connection, then $\Sigma_g(v)-\Sigma'_g(v)\in\ker(ds_g)$, and by Lemma \ref{lem:orbit-submersion}, $dt_g(\ker ds_g)=T_{t(g)}\mathcal O$, so $[\lambda_g^\Sigma(v)]=[\lambda_g^{\Sigma'}(v)]$ in $N_{t(g)}$.

Finally, we identify the induced representation with $\lambda^N$. Let $g:x\rightarrow y$ and let $v\in T_xM$. Choose any $Y\in T_g\cG$ such that $ds_g(Y)=v$. Then $Y-\Sigma_g(v)\in\ker(ds_g)$, so by Lemma \ref{lem:orbit-submersion} again,
\begin{equation}
[dt_g(Y)]=[dt_g(\Sigma_g(v))]=[\lambda_g^\Sigma(v)]\qquad\text{in }N_y.
\end{equation}
By Definition \ref{def:normal-orbit} and Proposition \ref{prop:normal-global}, the class $[dt_g(Y)]$ is $\lambda_g^N([v])$. Thus $[\lambda_g^\Sigma(v)] = \lambda_g^N([v])$, and since $\lambda^N:\cG\curvearrowright N$ is a smooth representation by Proposition \ref{prop:normal-global}, the degree-one cohomology representation of $\operatorname{Ad}(\cG)$ is the canonical normal representation $\lambda^N$.
\end{proof}


\subsection{The Lie algebroid of the lifted action groupoid}
\label{subsec:Lie-algebroid-lifted-action-groupoid}

\begin{proposition}\label{prop:lifted-action-groupoid-Lie-algebroid}
For the lifted action groupoid
\begin{equation}
\cH\colonequals \cG\ltimes OF(M,\mathcal O)\rightrightarrows OF(M,\mathcal O),
\end{equation}
there is a canonical Lie algebroid isomorphism
\begin{equation}
\Lie(\cH)\cong \pi^*A_{\cG},
\end{equation}
where $\pi^*A_{\cG}$ is endowed with the action Lie algebroid structure induced by the infinitesimal lifted action of $A_{\cG}$ on $OF(M,\mathcal O)$.
\end{proposition}

\begin{proof}
We use \eqref{eq:groupoid-bracket} for the right invariant bracket and \eqref{eq:groupoid-anchor} for the anchor.

For the corresponding action groupoid/action algebroid identification, see \cite{AriasAbadCrainic:RepUpToHomotopyGroupoids}*{Sec. 2.1}.

The arrow manifold of the lifted action groupoid is $\cG\times_M OF(M,\mathcal O)$, with source and target maps $s_{\cH}(g,e)=e$ and $t_{\cH}(g,e)=g\cdot e$.

Fix $e\in OF(M,\mathcal O)$, and put $x=\pi(e)$. The unit arrow at $e$ is $u_{\cH}(e)=(1_x,e)$. A tangent vector to $\cH_1$ at $(1_x,e)$ is a pair $(a,\xi)\in T_{1_x}\cG\oplus T_eOF(M,\mathcal O)$ satisfying $(ds)_{1_x}(a)=d\pi_e(\xi)$. Since $s_{\cH}(g,e)=e$, we have $(ds_{\cH})_{(1_x,e)}(a,\xi)=\xi$. Therefore
\begin{equation}\label{eq:lifted-action-algebroid-fibre}
\Lie(\cH)_e=\ker(ds_{\cH})_{(1_x,e)}=\{(a,0_e):a\in\ker(ds)_{1_x}\}.
\end{equation}
Since $(A_{\cG})_x=\ker(ds)_{1_x}$, the fibrewise maps $(A_{\cG})_{\pi(e)}\rightarrow\Lie(\cH)_e$, $a\mapsto(a,0_e)$, assemble into a canonical vector bundle isomorphism $\pi^*A_{\cG}\cong\Lie(\cH)$.

Under this identification, let $a\in(A_{\cG})_x$ and choose a smooth curve $g(\tau)$ in $s^{-1}(x)$ such that $g(0)=1_x$ and $g'(0)=a$. Then $\tau\mapsto(g(\tau),e)$ is a curve in $\cH_1$ through $(1_x,e)$ whose derivative at $\tau=0$ is $(a,0_e)$. Hence the anchor sends $(a,0_e)$ to
\begin{equation}\label{eq:lifted-action-anchor}
\left.\frac{d}{d\tau}\right|_{\tau=0}g(\tau)\cdot e.
\end{equation}
This is the value at $e$ of the infinitesimal lifted action of $a$. Moreover, since $\pi(g(\tau)\cdot e)=t(g(\tau))$, differentiating at $\tau=0$ gives
\begin{equation}\label{eq:lifted-action-anchor-projection}
d\pi_e\left(\left.\frac{d}{d\tau}\right|_{\tau=0}g(\tau)\cdot e\right)=(dt)_{1_x}(a)=\rho_{\cG}(a).
\end{equation}
Thus the infinitesimal lifted action projects to the anchor of $A_{\cG}$.

It remains to identify the bracket. Let $a$ be a local section of $A_{\cG}$. Under the vector bundle identification above, the pullback section $\pi^*a$ corresponds to $e\mapsto(a_{\pi(e)},0_e)$. Let $a^r$ denote the right invariant source-vertical vector field on $\cG$ associated to $a$. For an arrow $(g,e)\in\cH_1$, right multiplication by $(g,e)$ is the map $R_{(g,e)}:s_{\cH}^{-1}(g\cdot e)\rightarrow s_{\cH}^{-1}(e)$ given by $R_{(g,e)}(h,g\cdot e)=(hg,e)$. Under the same identification, the value of $\pi^*a$ at the target object $g\cdot e$ is $(a_{t(g)},0_{g\cdot e})$, and
\begin{equation}
(dR_{(g,e)})_{(1_{t(g)},g\cdot e)}(a_{t(g)},0_{g\cdot e})=((a^r)_g,0_e).
\end{equation}
Thus the right invariant source-vertical vector field on $\cH$ associated to $\pi^*a$ is $(a^r,0)$.

For local sections $a,b$ of $A_{\cG}$, it follows that
\begin{equation}
[(a^r,0),(b^r,0)]=([a^r,b^r],0)=([a,b]^r,0).
\end{equation}
Hence 
\begin{equation}\label{eq:lifted-action-bracket}
[\pi^*a,\pi^*b]=\pi^*[a,b]_{A_{\cG}}.
\end{equation}
 Since pullbacks of a local frame of $A_{\cG}$ form a local frame of $\pi^*A_{\cG}$, this identity, together with the anchor computed above and the Lie algebroid Leibniz rule, determines the bracket on all local sections of $\pi^*A_{\cG}$.

Thus the Lie algebroid structure transported from $\Lie(\cH)$ to $\pi^*A_{\cG}$ is exactly the action Lie algebroid structure induced by the infinitesimal lifted action. Hence $\Lie(\cH)\cong\pi^*A_{\cG}$ as Lie algebroids.
\end{proof}

\section{Molino's structure theorem for regular Riemannian groupoids}
\label{subsec:regular-groupoid-molino-theorem}

\begin{theorem}[Molino's structure theorem for regular Riemannian groupoids]\label{thm:molino-groupoid}
Let $\cG \rightrightarrows M$ be a regular Lie groupoid with $M$ compact and connected, and let $\eta^{(0)}$ be a $0$-metric on $M$. Let $\mathcal O$ be the orbit foliation, let $N\colonequals TM/T(\mathcal O)$, let $q\colonequals \rank N$, let $A_{\cG}\colonequals \Lie(\cG)$ with anchor $\rho_{\cG}:A_{\cG}\rightarrow TM$, and let $g^N$ be the quotient metric on $N$ induced by $\eta^{(0)}$. Let $\pi:OF(M,\mathcal O)\rightarrow M$ be the transverse orthonormal frame bundle of $(N,g^N)$, and let $\widetilde{\mathcal O}$ be the lifted foliation.

Then the following hold.
\begin{enumerate}[label=\textup{(\roman*)}]
\item The orbit foliation $(M,\mathcal O)$ is a Riemannian foliation. By Theorem \ref{thm:mol}, the foliated manifold $\bigl(OF(M,\mathcal O),\widetilde{\mathcal O}\bigr)$ is transversely parallelizable, with transverse canonical form $\theta$ and transverse Levi-Civita connection form $\omega$; there exist a smooth manifold $B$ with a smooth right $O(q)$-action and an $O(q)$-equivariant fibre bundle $\kappa:OF(M,\mathcal O)\rightarrow B$ whose fibres $L_b\colonequals \kappa^{-1}(b)$ are the closures of the leaves of $\widetilde{\mathcal O}$; and each restriction $\widetilde{\mathcal O}|_{L_b}$ is a Lie foliation with dense holonomy group. Moreover, $\kappa$ is proper.

\item If $\pi_{\bas}:OF(M,\mathcal O)\rightarrow W$ is the basic fibre bundle of $\bigl(OF(M,\mathcal O),\widetilde{\mathcal O}\bigr)$, and if $A\colonequals b(OF(M,\mathcal O),\widetilde{\mathcal O})\rightarrow W$ denotes the associated basic Lie algebroid, then there is a unique diffeomorphism $\Phi:B\xrightarrow{\cong}W$ such that $\pi_{\bas}=\Phi\circ\kappa$. Via $\Phi$, we regard $A$ as a transitive Lie algebroid over $B$, and restriction to the fibre induces a Lie algebra isomorphism
\begin{equation}
(\ker\rho_A)_b \xlongrightarrow{\ \cong\ } l\bigl(L_b,\widetilde{\mathcal O}|_{L_b}\bigr).
\end{equation}

\item The fibrewise Maurer--Cartan forms assemble into a unique global section
\begin{equation}
\omega^{\kappa}_{\mathrm{MC}}
\in
\Gamma\bigl((\ker d\kappa)^*\otimes \kappa^*(\ker\rho_A)\bigr)
\end{equation}
such that, for every $b\in B$, every $x\in L_b$, and every $\xi\in T_xL_b$,
\begin{equation}
\operatorname{res}_b\bigl((\omega^{\kappa}_{\mathrm{MC}})_x(\xi)\bigr)=(\omega^{b}_{\mathrm{MC}})_x(\xi),
\end{equation}
where $\omega^{b}_{\mathrm{MC}}$ is the Maurer--Cartan form of the Lie foliation $(L_b,\widetilde{\mathcal O}|_{L_b})$ and $\operatorname{res}_b$ is the restriction isomorphism in \textup{(ii)}.

\item If the arrow manifold of $\cG$ admits a Riemannian metric, then the normal representation $\lambda^N:\cG\curvearrowright N$ is the degree-one cohomology representation of the adjoint representation up to homotopy
\begin{equation}
\operatorname{Ad}(\cG)=\bigl(A_{\cG}\xlongrightarrow{\rho_{\cG}}TM\bigr).
\end{equation}
It lifts canonically to a smooth left action of $\cG$ on $OF(M,\mathcal O)$ with moment map $\pi$, given by $g\cdot e=\lambda_g^N\circ e$, and this action commutes with the right $O(q)$-action. The associated action groupoid
\begin{equation}
\cH\colonequals \cG\ltimes OF(M,\mathcal O)\rightrightarrows OF(M,\mathcal O)
\end{equation}
is a Lie groupoid and satisfies $\Lie(\cH)\cong\pi^*A_{\cG}$ and $t_{\cH}^*\theta=s_{\cH}^*\theta$, where $\pi^*A_{\cG}$ is endowed with the action Lie algebroid structure induced by the lifted action.

\item For every local bisection $\sigma:U\rightarrow \cG$, the induced map $f_\sigma:\pi^{-1}(U)\rightarrow\pi^{-1}(t(\sigma(U)))$, $f_\sigma(e)\colonequals \sigma(\pi(e))\cdot e$, preserves $\widetilde{\mathcal O}$, $\theta$, and $\omega$. It sends each $\kappa$-fibre contained in $\pi^{-1}(U)$ onto a $\kappa$-fibre, and hence descends to a local $O(q)$-equivariant diffeomorphism $\bar f_\sigma$ of $B$. For each $b$ in the domain of $\bar f_\sigma$, the map
\begin{equation}
\alpha_{\sigma,b}\colonequals \operatorname{res}_{\bar f_\sigma(b)}^{-1}\circ(f_\sigma|_{L_b})_*\circ \operatorname{res}_{b}:
(\ker\rho_A)_b\longrightarrow(\ker\rho_A)_{\bar f_\sigma(b)}
\end{equation}
is a Lie algebra isomorphism, and these maps assemble into a smooth Lie algebra bundle isomorphism $\alpha_\sigma$ covering $\bar f_\sigma$. Moreover, for $x\in L_b$ and $\xi\in T_xL_b$,
\begin{equation}
(\omega^{\kappa}_{\mathrm{MC}})_{f_\sigma(x)}\bigl((df_\sigma)_x\xi\bigr)
=
\alpha_{\sigma,b}\bigl((\omega^{\kappa}_{\mathrm{MC}})_x(\xi)\bigr).
\end{equation}
\end{enumerate}
\end{theorem}

We now assemble the preceding results into a proof of the main theorem.

\begin{proof}[Proof of Theorem \ref{thm:molino-groupoid}]

Let $g^M\colonequals \eta^{(0)}$. By Proposition \ref{prop:0-metric}, the normal representation acts by fibrewise isometries for the induced metric $g^N$ on $N$. By Proposition \ref{prop:RF}, $(M,\mathcal O)$ is a compact connected Riemannian foliation.

Following Definitions \ref{def:OF}, \ref{def:lifted-foliation}, and \ref{def:theta-omega}, we obtain the transverse orthonormal frame bundle $\pi:OF(M,\mathcal O)\rightarrow M$, the lifted foliation $\widetilde{\mathcal O}$, the transverse canonical form $\theta$, the transverse Levi-Civita connection $\nabla^{\mathrm{tr}}$ on $N$, and the associated principal connection form $\omega$. By Proposition \ref{prop:transverse-LC-local}, $\omega$ is well-defined, vanishes on $T(\widetilde{\mathcal O})$, and is invariant under vector fields tangent to $\widetilde{\mathcal O}$. By Lemma \ref{lem:theta-omega-parallelism}, we obtain a transverse parallelism of $\bigl(OF(M,\mathcal O),\widetilde{\mathcal O}\bigr)$.

Applying Theorem \ref{thm:mol} to the compact connected Riemannian foliation $(M,\mathcal O)$ gives the smooth manifold $B$, the smooth right $O(q)$-action, and the $O(q)$-equivariant fibre bundle $\kappa:OF(M,\mathcal O)\rightarrow B$ whose fibres are the closures of the leaves of $\widetilde{\mathcal O}$. Since $M$ and $O(q)$ are compact, the total space $OF(M,\mathcal O)$ is compact; as $B$ is a Hausdorff manifold, $\kappa$ is proper. In particular, $\kappa$ satisfies the properness hypothesis of Lemma \ref{lem:descend-to-B}.

Let $X\colonequals OF(M,\mathcal O)$ and $\F\colonequals \widetilde{\mathcal O}$. As observed in \cite{MoerdijkMrcun:IFLG}*{proof of Theorem 4.26}, $(X,\F)$ is homogeneous. Since $X$ is compact and $(X,\F)$ is transversely parallelizable, we obtain the transitive basic Lie algebroid $A\colonequals b(X,\F)\rightarrow W\colonequals X/\F_{\bas}$ by Proposition \ref{prop:basic-Lie-algebroid}. By Corollary \ref{cor:B-W-identification}, there is a unique diffeomorphism $\Phi:B\xrightarrow{\cong}W$ such that $\pi_{\bas}=\Phi\circ\kappa$. Using $\Phi$ to regard $A$ as a Lie algebroid over $B$, we identify $(\ker\rho_A)_b$ with $l(L_b,\widetilde{\mathcal O}|_{L_b})$ by restriction, as in Lemma \ref{lem:isotropy-structural-clean}.

By Proposition \ref{prop:vertical-MC}\textup{(i)}, there is a unique global section $\omega^{\kappa}_{\mathrm{MC}}$ formed by the fibrewise Maurer--Cartan forms, with defining property \eqref{eq:MC-res}. If the arrow manifold of $\cG$ admits a Riemannian metric, the identification of the normal representation with the degree-one cohomology representation of $\operatorname{Ad}(\cG)$ follows from Proposition \ref{prop:normal-adjoint-cohomology}.

In Proposition \ref{prop:action-axioms}, we construct the smooth left action of $\cG$ on $OF(M,\mathcal O)$ given by $g\cdot e=\lambda_g^N\circ e$, show that it commutes with the right $O(q)$-action, and prove that it is the unique such action inducing $\lambda^N$ on the associated bundle $OF(M,\mathcal O)\times_{O(q)}\RR^q\cong N$. In Proposition \ref{prop:lifted-action-groupoid-Lie-algebroid}, we identify the Lie algebroid of the lifted action groupoid $\cH$ with the action Lie algebroid on the pullback bundle $\pi^*A_{\cG}$. In Proposition \ref{prop:lifted-action-groupoid}, we prove that $t_{\cH}^*\theta=s_{\cH}^*\theta$.

Finally, let $\sigma:U\rightarrow \cG$ be a local bisection, and let $V\colonequals t(\sigma(U))$. By Proposition \ref{prop:molino-data-invariant}, the induced map $f_\sigma$ preserves $\widetilde{\mathcal O}$, $\theta$, and $\omega$; in particular, $f_\sigma$ is foliated for $\widetilde{\mathcal O}$ and sends leaves of $\widetilde{\mathcal O}|_{\pi^{-1}(U)}$ to leaves of $\widetilde{\mathcal O}|_{\pi^{-1}(V)}$. Applying Lemma \ref{lem:descend-to-B}, we obtain the descended diffeomorphism $\bar f_\sigma:B_U\rightarrow B_V$. The descended maps are compatible with products and inverses and are $O(q)$-equivariant by Lemmas \ref{lem:descend-to-B} and \ref{lem:bisection-pseudogroup}. In Lemma \ref{lem:alpha-sigma}, we construct the smooth Lie algebra bundle isomorphism $\alpha_\sigma$ covering $\bar f_\sigma$, and in Proposition \ref{prop:vertical-MC}\textup{(ii)}, we prove the equivariance formula for $\omega^{\kappa}_{\mathrm{MC}}$.
\end{proof}


\begin{remark}
The basic Lie algebroid $A\rightarrow W$ and the Lie algebroid $A_{\cG}\rightarrow M$ fit into the short exact sequences of vector bundles
\begin{gather}
0\longrightarrow \ker \rho_A\longrightarrow A\xlongrightarrow{\rho_A}TW\longrightarrow 0,\\
0\longrightarrow \ker \rho_{\cG}\longrightarrow A_{\cG}\xlongrightarrow{\rho_{\cG}}T(\mathcal O)\longrightarrow 0.
\end{gather}

Thus the image of $\rho_{\cG}$ is the tangent bundle $T(\mathcal O)$ of the orbit foliation, while the image of $\rho_A$ is $TW$. For $x\in M$, the fibre $(\ker\rho_{\cG})_x$ is the Lie algebra of the isotropy group $\cG_x$. For $b\in B$ and $w=\Phi(b)$, restriction to $L_b$ identifies $(\ker\rho_A)_w$ with the structural Lie algebra of the Lie foliation on $L_b$.


\end{remark}


\section{Molino’s structures under Riemannian Morita equivalence}\label{sec:riemannian-morita-invariance}

To compare two groupoids related by a Riemannian Morita equivalence bibundle $P$, we first construct a Lie groupoid over $P$ with Morita fibrations to both groupoids. In the proper effective \'etale case, we compare the quotients of the orthonormal frame bundles by the lifted groupoid actions. In the regular case, we compare the leaf closure bases of the lifted orbit foliations and the basic Lie algebroids over those bases.

\begin{lemma}\label{lem:bibundle-pullback-groupoid}
Let $\cG'\curvearrowright P\curvearrowleft\cG$ be a Morita equivalence bibundle with moment maps $\alpha':P\rightarrow M'$ and $\alpha:P\rightarrow M$. Let $\mathcal K_P\rightrightarrows P$ be the pullback groupoid of $\cG$ through $\alpha$: its arrows are the triples $(p_1,g,p_0)$ with $s(g)=\alpha(p_0)$ and $t(g)=\alpha(p_1)$, with source $p_0$, target $p_1$, and multiplication $(p_2,h,p_1)(p_1,g,p_0)=(p_2,hg,p_0)$. Then:
\begin{enumerate}[label=\textup{(\arabic*)}]
\item $z:\mathcal K_P\rightarrow\cG$, $(p_1,g,p_0)\mapsto g$, is a Morita fibration with object map $\alpha$.
\item For each arrow $(p_1,g,p_0)$ of $\mathcal K_P$, there is a unique $g'\in\cG'$ with
\begin{equation}\label{eq:bibundle-comparison-arrow}
p_1\cdot g=g'\cdot p_0,
\end{equation}
and $z':\mathcal K_P\rightarrow\cG'$, $(p_1,g,p_0)\mapsto g'$, is a Morita fibration with object map $\alpha'$.
\item For $p\in P$, the maps $g'\mapsto g'\cdot p$ and $g\mapsto p\cdot g$ are diffeomorphisms
\begin{equation}\label{eq:bibundle-moment-fibres}
s_{\cG'}^{-1}(\alpha'(p))\xlongrightarrow{\ \cong\ }\alpha^{-1}(\alpha(p)),\qquad t_{\cG}^{-1}(\alpha(p))\xlongrightarrow{\ \cong\ }(\alpha')^{-1}(\alpha'(p)).
\end{equation}
In particular, if $\cG'$ and $\cG$ have connected source fibres, then $\alpha$ and $\alpha'$ have connected fibres.
\end{enumerate}
\end{lemma}

\begin{proof}
Both moment maps are surjective submersions, since $\alpha'$ is the bundle projection of the principal right $\cG$-action and $\alpha$ that of the principal left $\cG'$-action. Hence $\mathcal K_P$ is a Lie groupoid.

(1) By construction, the square formed by $z$ and the maps $(s,t)$ is a pullback, so $z$ is fully faithful. Since its object map is $\alpha$, it is a Morita fibration.

(2) For an arrow $(p_1,g,p_0)$, we have $\alpha(p_1\cdot g)=s(g)=\alpha(p_0)$. By principality of the left $\cG'$-action, there is a unique $g'\in\cG'$ with $p_1\cdot g=g'\cdot p_0$, depending smoothly on $(p_1,g,p_0)$. Moreover, $s(g')=\alpha'(p_0)$ and $t(g')=\alpha'(p_1)$. If $p_2\cdot h=h'\cdot p_1$, then, because the two actions commute,
\begin{equation}
p_2\cdot(hg)=(p_2\cdot h)\cdot g=h'\cdot(p_1\cdot g)=(h'g')\cdot p_0,
\end{equation}
so $z'$ preserves multiplication; by uniqueness, it preserves units and inverses.

Conversely, for an arrow $(p_1,g',p_0)$ in the pullback groupoid of $\cG'$ through $\alpha'$, by principality of the right $\cG$-action, there is a unique $g$ with $p_1\cdot g=g'\cdot p_0$, depending smoothly on $(p_1,g',p_0)$. Thus the map $(p_1,g,p_0)\mapsto(p_1,g',p_0)$ is an isomorphism of Lie groupoids over $P$ from $\mathcal K_P$ to this pullback groupoid. Under this identification, $z'$ is the projection to $\cG'$, hence a Morita fibration since $\alpha'$ is a surjective submersion.

(3) By principality of the left action over $M$, the map $\cG'\times_{M'}P\rightarrow P\times_MP$, $(g',p)\mapsto(g'\cdot p,p)$, is a diffeomorphism; restricting to the fibre over $p$ gives the first diffeomorphism. The second follows in the same way from principality of the right action over $M'$. Since inversion identifies $t_{\cG}^{-1}(\alpha(p))$ with $s_{\cG}^{-1}(\alpha(p))$, we are done.
\end{proof}

\subsection{Proper \'etale groupoid case}\label{sec:proper-etale-case}

Let $\cE\rightrightarrows T$ be a proper effective \'etale Riemannian groupoid, and let $q\colonequals \dim T$. We denote by $\pi_T:OF(T)\rightarrow T$ the orthonormal frame bundle of $T$. For an \'etale Lie groupoid, every $0$-metric extends uniquely to a $2$-metric \cite{MdHF-LieGpdMetrics}*{Sec. 3.2, Example 4.1}.

Since $\cE$ is \'etale, the connected components of its orbits are points. Thus the orbit foliation has $T\mathcal O=0$, $N=TT$, and $OF(T,\mathcal O)=OF(T)$. Under this identification, the normal representation is the derivative of the local diffeomorphism germ determined by an arrow.

For an arrow $g:x\rightarrow y$ in $\cE$, choose a local bisection $\sigma:U\rightarrow \cE$ through $g$, so that $\sigma(x)=g$. Then $t\circ\sigma:U\rightarrow T$ is a local diffeomorphism sending $x$ to $y$. Since $\cE$ is \'etale, the germ at $x$ of $t\circ\sigma$ is independent of the chosen local bisection. We denote this germ by $\psi_g:(T,x)\rightarrow (T,y)$. By \cite{MdHF-LieGpdMetrics}*{Lemma 3.5}, $\psi_g$ is a germ of a local isometry for the given $0$-metric on $T$. Hence $d_x\psi_g:T_xT\rightarrow T_yT$ is a linear isometry. Under the identification $N=TT$, this is $\lambda_g^N$.

We denote the action groupoid of Proposition \ref{prop:lifted-action-groupoid} in this case by $\widehat{\cE}\rightrightarrows OF(T)$. Thus
\begin{equation}\label{eq:etale-frame-structure}
\widehat{\cE}_1=\cE_1\times_T OF(T),\qquad \widehat s(g,u)=u,\qquad \widehat t(g,u)=d_{\pi_T(u)}\psi_g\circ u.
\end{equation}
By Proposition \ref{prop:action-axioms}, the right $O(q)$-action commutes with the lifted $\cE$-action. Hence, for every $A\in O(q)$, the maps $u\mapsto uA$ and $(g,u)\mapsto(g,uA)$ define a Lie groupoid automorphism of $\widehat{\cE}$, with inverse obtained by replacing $A$ with $A^{-1}$.

\begin{proposition}\label{prop:lifted-frame-groupoid-properties}
The lifted frame groupoid $\widehat{\cE}\rightrightarrows OF(T)$ is Hausdorff, proper, free, and \'{e}tale.

Therefore the orbit space $B_{\cE}\colonequals OF(T)/\widehat{\cE}$ has a unique smooth manifold structure such that the quotient map $\kappa_{\cE}:OF(T)\rightarrow B_{\cE}$ is a surjective local diffeomorphism. Its fibre through $u\in OF(T)$ is the $\widehat{\cE}$-orbit
\begin{equation}\label{eq:etale-frame-fibre}
\kappa_{\cE}^{-1}(\kappa_{\cE}(u))=\widehat{\cE}\cdot u=\{d_{\pi_T(u)}\psi_g\circ u : g\in\cE_1, s(g)=\pi_T(u)\}.
\end{equation}
Moreover, the right $O(q)$-action on $OF(T)$ descends to a smooth right $O(q)$-action on $B_{\cE}$, defined by
\begin{equation}\label{eq:etale-quotient-Oq-action}
\kappa_{\cE}(u)\cdot A\colonequals \kappa_{\cE}(uA).
\end{equation}
\end{proposition}

\begin{proof}
We first show that $\widehat{\cE}$ is \textit{\'{e}tale}. The source map $\widehat s$ is the pullback of the local diffeomorphism $s:\cE_1\rightarrow T$ along $\pi_T$, hence is a local diffeomorphism. Since inversion is a diffeomorphism, $\widehat t$ is also a local diffeomorphism, so $\widehat{\cE}$ is \'etale.

We next show \textit{properness}. Let $K\subset OF(T)\times OF(T)$ be compact. Since $\cE$ is proper, $(s,t)^{-1}((\pi_T\times\pi_T)(K))$ is compact, and $\operatorname{pr}_1(K)$ is compact. The set
\begin{equation}
D\colonequals \{(g,u)\in (s,t)^{-1}((\pi_T\times\pi_T)(K)) \times \operatorname{pr}_1(K):s(g)=\pi_T(u)\}
\end{equation}
is closed in $(s,t)^{-1}((\pi_T\times\pi_T)(K)) \times \operatorname{pr}_1(K)$, hence compact. The map $D\rightarrow OF(T)\times OF(T)$, $(g,u)\mapsto (u,d_{\pi_T(u)}\psi_g\circ u)$, is continuous, and the inverse image of the closed set $K$ is $(\widehat s,\widehat t)^{-1}(K)$. Hence $(\widehat s,\widehat t)^{-1}(K)$ is compact, so $\widehat{\cE}$ is proper.

We now show \textit{freeness}. Suppose $(g,u)$ fixes $u\in OF(T)_x$. Then $g\in\cE_x$ and $d_x\psi_g\circ u=u$. Since $u$ is a linear isomorphism, $d_x\psi_g=\operatorname{id}_{T_xT}$.

The germ $\psi_g$ is represented by a local isometry fixing $x$ with $d_x\psi_g=\operatorname{id}_{T_xT}$. After shrinking this representative, choose $\epsilon>0$ such that $\exp_x:B(0,\epsilon)\subset T_xT\rightarrow W\subset T$ is a diffeomorphism onto an open neighbourhood of $x$ contained in the domain of this representative. By naturality of the exponential map under local isometries,
\begin{equation}\label{eq:etale-exponential-naturality}
\psi_g(\exp_x v)=\exp_x(d_x\psi_g(v))=\exp_x(v)
\end{equation}
for all $v\in B(0,\epsilon)$. Hence $\psi_g$ is the identity germ at $x$. Since $\cE$ is effective, $g=1_x$. Therefore $\widehat{\cE}$ is free.

Since $\widehat{\cE}$ has trivial isotropy, the map $(\widehat s,\widehat t):\widehat{\cE}_1\rightarrow OF(T)\times OF(T)$ is continuous and injective. Its target is Hausdorff. Hence $\widehat{\cE}_1$ is Hausdorff.

We have shown that $\widehat{\cE}$ is proper, free, and \'{e}tale. Since $\widehat{\cE}$ is proper, the orbit space $B_{\cE}=OF(T)/\widehat{\cE}$ is Hausdorff. By \cite{HenriquesMetzler:NoneffectiveOrbifolds}*{Proposition 4.2 and Corollary 4.3}, $B_{\cE}$ has a smooth manifold structure for which $\kappa_{\cE}:OF(T)\rightarrow B_{\cE}$ is a surjective local diffeomorphism. This structure is the unique one with this property.

The fibre formula follows immediately from the definition of the orbit space. Finally, right translation by any $A\in O(q)$ sends $\widehat{\cE}$-orbits to $\widehat{\cE}$-orbits, so $\kappa_{\cE}(u)\cdot A\colonequals \kappa_{\cE}(uA)$ defines a right $O(q)$-action on $B_{\cE}$.

The composite of the descended action with $\kappa_{\cE}\times\operatorname{id}_{O(q)}$ is the smooth map $(u,A)\mapsto\kappa_{\cE}(uA)$. Using local inverses of the surjective local diffeomorphism $\kappa_{\cE}\times\operatorname{id}_{O(q)}$, we see that the descended action is smooth.
\end{proof}

Now let $\cE'\rightrightarrows T'$ and $\cE\rightrightarrows T$ be proper effective \'etale Riemannian groupoids, and let $\cE'\curvearrowright P\curvearrowleft \cE$ be a Riemannian Morita equivalence bibundle with moment maps $\alpha':P\rightarrow T'$ and $\alpha:P\rightarrow T$.

\begin{proposition}\label{prop:OF-P-Morita-bibundle}
Let $\cE'\curvearrowright P\curvearrowleft \cE$ be a Riemannian Morita equivalence bibundle as above. Then the following hold.
\begin{enumerate}[label=\textup{(\arabic*)}]
\item $\dim T=\dim P=\dim T'\equalscolon q$, and the moment maps $\alpha$ and $\alpha'$ are surjective local isometries. They induce $O(q)$-equivariant surjective local diffeomorphisms of orthonormal frame bundles
\begin{equation}\label{eq:etale-bibundle-frame-map}
OF(\alpha):OF(P)\longrightarrow OF(T),\qquad OF(\alpha)(\xi)=d_p\alpha\circ\xi,
\end{equation}
and
\begin{equation}\label{eq:etale-bibundle-frame-map-prime}
OF(\alpha'):OF(P)\longrightarrow OF(T'),\qquad OF(\alpha')(\xi)=d_p\alpha'\circ\xi,
\end{equation}
where $\xi\in OF_p(P)$.

\item The frame bundle $OF(P)$ is an $O(q)$-equivariant Morita equivalence bibundle $\widehat{\cE}'\curvearrowright OF(P)\curvearrowleft \widehat{\cE}$ between the lifted frame groupoids, with moment maps $OF(\alpha')$ and $OF(\alpha)$.

\item There is a unique $O(q)$-equivariant diffeomorphism $\Phi_P:B_{\cE}\rightarrow B_{\cE'}$ such that
\begin{equation}\label{eq:etale-morita-base-map}
\Phi_P\circ\kappa_{\cE}\circ OF(\alpha)=\kappa_{\cE'}\circ OF(\alpha')
\end{equation}
on $OF(P)$.
\end{enumerate}
\end{proposition}

\begin{proof}
(1) Since $\cE$ and $\cE'$ are \'etale, the fibres of $\alpha$ and $\alpha'$ are discrete by Lemma \ref{lem:bibundle-pullback-groupoid}\textup{(3)}. Both maps are surjective submersions, so $\dim T=\dim P=\dim T'=q$.

Because the bibundle is Riemannian, the maps $\alpha$ and $\alpha'$ are Riemannian submersions. Since they are submersions between manifolds of equal dimension, they are local diffeomorphisms, hence local isometries. Therefore they induce surjective local diffeomorphisms of orthonormal frame bundles. The maps $OF(\alpha)$ and $OF(\alpha')$ are $O(q)$-equivariant:
\begin{equation}\label{eq:etale-frame-map-equivariance}
OF(\alpha)(\xi A)=d_p\alpha\circ\xi\circ A=OF(\alpha)(\xi)A,\qquad OF(\alpha')(\xi A)=OF(\alpha')(\xi)A.
\end{equation}

(2) We \textit{lift the right $\cE$-action on $P$}. Let $(g,u)\in\widehat{\cE}_1$, where $g:x\rightarrow y$ and $u\in OF_x(T)$, and let $\xi\in OF_p(P)$ satisfy $OF(\alpha)(\xi)=\widehat t(g,u)=d_x\psi_g\circ u$; in particular, $\alpha(p)=y=t(g)$, so the right action by $g$ is defined. Let $R_g$ be the germ of the local right action map on $P$ induced by a local bisection through $g$, so $R_g(p)=p\cdot g$. It satisfies $\alpha\circ R_g=\psi_{g^{-1}}\circ\alpha$ and $\alpha'\circ R_g=\alpha'$; differentiating the first identity at $p$ gives
\begin{equation}\label{eq:etale-right-action-differential}
d_{p\cdot g}\alpha\circ d_pR_g=d_y\psi_{g^{-1}}\circ d_p\alpha.
\end{equation}
Define
\begin{equation}\label{eq:etale-lifted-right-action}
\xi\cdot(g,u)\colonequals d_pR_g\circ\xi.
\end{equation}
Then, as linear maps,
\begin{equation}\label{eq:etale-lifted-right-moment}
\begin{aligned}
d_{p\cdot g}\alpha\circ\bigl(\xi\cdot(g,u)\bigr)&=d_y\psi_{g^{-1}}\circ d_p\alpha\circ\xi\\
&=d_y\psi_{g^{-1}}\circ d_x\psi_g\circ u\\
&=u.
\end{aligned}
\end{equation}
Since $d_{p\cdot g}\alpha$ is a linear isometry and $u$ is an orthonormal frame, $\xi\cdot(g,u)=(d_{p\cdot g}\alpha)^{-1}\circ u$ is an orthonormal frame; thus $\xi\cdot(g,u)\in OF_{p\cdot g}(P)$ and $OF(\alpha)(\xi\cdot(g,u))=u$. Also,
\begin{equation}\label{eq:etale-lifted-right-invariance}
\begin{aligned}
OF(\alpha')(\xi\cdot(g,u))&=d_{p\cdot g}\alpha'\circ d_pR_g\circ\xi\\
&=d_p\alpha'\circ\xi=OF(\alpha')(\xi),
\end{aligned}
\end{equation}
using $\alpha'\circ R_g=\alpha'$. So this defines a right action of $\widehat{\cE}$ on $OF(P)$ with moment map $OF(\alpha)$.

Similarly, we \textit{lift the left $\cE'$-action}. Let $(g',u')\in\widehat{\cE}'_1$, where $g':x'\rightarrow y'$ and $u'\in OF_{x'}(T')$, and let $\xi\in OF_p(P)$ satisfy $OF(\alpha')(\xi)=u'$. Let $L_{g'}$ denote the germ of the local left action map, $L_{g'}(p)=g'\cdot p$; it satisfies $\alpha'\circ L_{g'}=\psi_{g'}\circ\alpha'$ and $\alpha\circ L_{g'}=\alpha$. Define
\begin{equation}\label{eq:etale-lifted-left-action}
(g',u')\cdot\xi\colonequals d_pL_{g'}\circ\xi.
\end{equation}
By the same argument, $(g',u')\cdot\xi\in OF_{g'\cdot p}(P)$, with $OF(\alpha')\bigl((g',u')\cdot\xi\bigr)=d_{x'}\psi_{g'}\circ u'$ and $OF(\alpha)\bigl((g',u')\cdot\xi\bigr)=OF(\alpha)(\xi)$. So this defines a left action of $\widehat{\cE}'$ on $OF(P)$ with moment map $OF(\alpha')$.

The \textit{two lifted actions commute} because the $\cE'$- and $\cE$-actions on $P$ commute: $g'\cdot(p\cdot g)=(g'\cdot p)\cdot g$. Differentiating the corresponding identity for the local maps induced by bisections gives $d_{p\cdot g}L_{g'}\circ d_pR_g=d_{g'\cdot p}R_g\circ d_pL_{g'}$, so $(g',u')\cdot\bigl(\xi\cdot(g,u)\bigr)=\bigl((g',u')\cdot\xi\bigr)\cdot(g,u)$. The associativity and unit identities follow in the same way from the corresponding identities for the original actions on $P$, and smoothness follows locally by differentiating the smooth action maps induced by local bisections.

Next we discuss \textit{principality}. Take frames $\xi_1\in OF_{p_1}(P)$ and $\xi_2\in OF_{p_2}(P)$ with $OF(\alpha')(\xi_1)=OF(\alpha')(\xi_2)$; then $\alpha'(p_1)=\alpha'(p_2)$. Since the right $\cE$-action on $P$ is principal over $T'$, there is a unique arrow $g:x\rightarrow y$ in $\cE$ such that $p_2=p_1\cdot g$. Since $\alpha'\circ R_g=\alpha'$, we get $d_{p_2}\alpha'\circ d_{p_1}R_g\circ\xi_1=d_{p_1}\alpha'\circ\xi_1=d_{p_2}\alpha'\circ\xi_2$. Because $d_{p_2}\alpha'$ is an isomorphism, $d_{p_1}R_g\circ\xi_1=\xi_2$. Applying $d_{p_2}\alpha$ and using \eqref{eq:etale-right-action-differential}, we obtain $OF(\alpha)(\xi_2)=d_y\psi_{g^{-1}}\circ OF(\alpha)(\xi_1)$, that is, $OF(\alpha)(\xi_1)=d_x\psi_g\circ OF(\alpha)(\xi_2)$; hence the lifted action is defined and $\xi_2=\xi_1\cdot\widehat g$ with $\widehat g\colonequals (g,OF(\alpha)(\xi_2))$.

\textit{Uniqueness of $\widehat g$} follows from the uniqueness of $g$. The assignment $(\xi_1,\xi_2)\mapsto\widehat g$ is smooth, because $g$ is obtained from the inverse of the diffeomorphism $P\times_T\cE_1\rightarrow P\times_{T'}P$, $(p,g)\mapsto(p,p\cdot g)$, and $OF(\alpha)(\xi_2)$ depends smoothly on $\xi_2$. Hence the lifted right $\widehat{\cE}$-action is principal.

The proof of principality of the lifted left action is symmetric: if $OF(\alpha)(\xi_1)=OF(\alpha)(\xi_2)$, then $\alpha(p_1)=\alpha(p_2)$, so there is a unique arrow $g'$ in $\cE'$ with $p_1=g'\cdot p_2$. Differentiating $\alpha\circ L_{g'}=\alpha$ gives $\xi_1=d_{p_2}L_{g'}\circ\xi_2=(g',u')\cdot\xi_2$ with $u'\colonequals OF(\alpha')(\xi_2)$. Therefore $\widehat{\cE}'\curvearrowright OF(P)\curvearrowleft\widehat{\cE}$ is a Morita equivalence bibundle.

Finally, we check \textit{$O(q)$-equivariance}. The right $O(q)$-action on each arrow manifold is $(g,u)A=(g,uA)$, and the moment maps $OF(\alpha)$ and $OF(\alpha')$ are $O(q)$-equivariant by part (1). For the lifted actions, $(\xi A)\cdot\bigl((g,u)A\bigr)=d_pR_g\circ\xi\circ A=\bigl(\xi\cdot(g,u)\bigr)A$, and similarly $\bigl((g',u')A\bigr)\cdot(\xi A)=\bigl((g',u')\cdot\xi\bigr)A$. Thus $OF(P)$ is an $O(q)$-equivariant Morita equivalence bibundle.

(3) Define $\Phi_P$ by requiring $\Phi_P\circ\kappa_{\cE}\circ OF(\alpha)=\kappa_{\cE'}\circ OF(\alpha')$ on $OF(P)$. Since $OF(\alpha):OF(P)\rightarrow OF(T)$ and $\kappa_{\cE}:OF(T)\rightarrow B_{\cE}$ are surjective local diffeomorphisms, so is the composite $\kappa_{\cE}\circ OF(\alpha):OF(P)\rightarrow B_{\cE}$; hence this requirement determines $\Phi_P$ uniquely once it is well-defined.

Suppose $\xi_0,\xi_1\in OF(P)$ have the same image under $\kappa_{\cE}\circ OF(\alpha)$, and put $u_i\colonequals OF(\alpha)(\xi_i)$ for $i=0,1$. Then $u_0,u_1$ lie in the same $\widehat{\cE}$-orbit, so there is an arrow $(g,u_1):u_1\rightarrow u_0$. The lifted right action therefore gives
\begin{equation}
OF(\alpha)\bigl(\xi_0\cdot(g,u_1)\bigr)=u_1=OF(\alpha)(\xi_1).
\end{equation}
By principality of the lifted left action in (2), there is an arrow of $\widehat{\cE}'$ carrying $\xi_1$ to $\xi_0\cdot(g,u_1)$. Their $OF(\alpha')$-images therefore lie in the same $\widehat{\cE}'$-orbit. Since the lifted right action preserves $OF(\alpha')$, we obtain
\begin{equation}
\begin{aligned}
\kappa_{\cE'}(OF(\alpha')(\xi_0))&=\kappa_{\cE'}\bigl(OF(\alpha')(\xi_0\cdot(g,u_1))\bigr)\\
&=\kappa_{\cE'}(OF(\alpha')(\xi_1)).
\end{aligned}
\end{equation}
Thus $\kappa_{\cE'}\circ OF(\alpha')$ is constant on the fibres of $\kappa_{\cE}\circ OF(\alpha)$, and $\Phi_P$ is well-defined.

Using the lifted left action and principality of the lifted right action in the same way, we find that $\kappa_{\cE}\circ OF(\alpha)$ is constant on the fibres of $\kappa_{\cE'}\circ OF(\alpha')$. Hence the two composites have the same fibres. By part (1) and Proposition \ref{prop:lifted-frame-groupoid-properties}, both are surjective local diffeomorphisms onto Hausdorff manifolds. Therefore, by Lemma \ref{lem:same-fibres-diffeo}, $\Phi_P:B_{\cE}\rightarrow B_{\cE'}$ is a diffeomorphism.

Finally, $\kappa_{\cE}\circ OF(\alpha)$ and $\kappa_{\cE'}\circ OF(\alpha')$ are $O(q)$-equivariant by part (1) and Proposition \ref{prop:lifted-frame-groupoid-properties}. By \eqref{eq:etale-morita-base-map} and the surjectivity of $\kappa_{\cE}\circ OF(\alpha)$, $\Phi_P$ is $O(q)$-equivariant.
\end{proof}

We next show that the canonical form and the Levi-Civita connection form of $OF(T)$ are invariant under the lifted frame groupoid, so that they descend to $B_{\cE}$. We also show that the maps induced on the orthonormal frame bundles by a Riemannian Morita equivalence bibundle pull back these forms to the corresponding forms on $OF(P)$.

\begin{lemma}\label{lem:etale-theta-omega-morita}
Let $\cE\rightrightarrows T$ be a proper effective \'etale Riemannian groupoid, let $\theta_T\in\Omega^1(OF(T),\RR^q)$ and $\omega_T\in\Omega^1(OF(T),\mathfrak o(q))$ be the canonical form and the Levi-Civita connection form of $\pi_T:OF(T)\rightarrow T$, and let $\kappa_{\cE}:OF(T)\rightarrow B_{\cE}$ be the quotient map.

\begin{enumerate}[label=\textup{(\arabic*)}]
\item On $\widehat{\cE}_1$ we have $\widehat t^{\,*}\theta_T=\widehat s^{\,*}\theta_T$ and $\widehat t^{\,*}\omega_T=\widehat s^{\,*}\omega_T$. Consequently, there are unique forms $\theta_{B_{\cE}}\in\Omega^1(B_{\cE},\RR^q)$ and $\omega_{B_{\cE}}\in\Omega^1(B_{\cE},\mathfrak o(q))$ with
\begin{equation}\label{eq:etale-descended-forms}
\kappa_{\cE}^*\theta_{B_{\cE}}=\theta_T,\qquad \kappa_{\cE}^*\omega_{B_{\cE}}=\omega_T,
\end{equation}
and for every $b\in B_{\cE}$, the map $(\theta_{B_{\cE}},\omega_{B_{\cE}})_b:T_bB_{\cE}\rightarrow\RR^q\oplus\mathfrak o(q)$ is a linear isomorphism.

\item In the situation of Proposition \ref{prop:OF-P-Morita-bibundle}, let $\theta_{T'}$, $\theta_P$ and $\omega_{T'}$, $\omega_P$ be the canonical forms and the Levi-Civita connection forms of $OF(T')$ and $OF(P)$. Then
\begin{equation}\label{eq:etale-bibundle-form-pullbacks}
\begin{aligned}
OF(\alpha)^*\theta_T&=\theta_P=OF(\alpha')^*\theta_{T'},\\
OF(\alpha)^*\omega_T&=\omega_P=OF(\alpha')^*\omega_{T'}
\end{aligned}
\end{equation}
on $OF(P)$, and $\Phi_P^*\theta_{B_{\cE'}}=\theta_{B_{\cE}}$ and $\Phi_P^*\omega_{B_{\cE'}}=\omega_{B_{\cE}}$.
\end{enumerate}
\end{lemma}

\begin{proof}
(1) Let $g:x\rightarrow y$ be an arrow of $\cE$, and choose a local bisection $\sigma:U\rightarrow\cE$ through $g$, with $U$ so small that $h\colonequals t\circ\sigma:U\rightarrow V$ is a diffeomorphism onto an open subset $V\subset T$. By \cite{MdHF-LieGpdMetrics}*{Lemma 3.5}, $h$ is a local isometry for the given $0$-metric on $T$, and its frame lift $OF(h):OF(T)|_U\rightarrow OF(T)|_V$, $OF(h)(u)\colonequals d_{\pi_T(u)}h\circ u$, satisfies $\pi_T\circ OF(h)=h\circ\pi_T$. Applying \eqref{eq:frame-lift-naturality} to $h$ for the foliation of $T$ by points, we obtain $OF(h)^*\theta_T=\theta_T$ and $OF(h)^*\omega_T=\omega_T$ on $OF(T)|_U$.

For $u\in OF_x(T)$, an open neighbourhood of $(g,u)$ in $\widehat{\cE}_1$ is parametrized by $v\mapsto(\sigma(\pi_T(v)),v)$, $v\in OF(T)|_U$. Under this parametrization, the source and target maps are given by $v\mapsto v$ and $v\mapsto OF(h)(v)$, respectively. Hence $\widehat t^{\,*}\theta_T=\widehat s^{\,*}\theta_T$ and $\widehat t^{\,*}\omega_T=\widehat s^{\,*}\omega_T$ near $(g,u)$, and such neighbourhoods cover $\widehat{\cE}_1$.

Since $\kappa_{\cE}$ is a surjective local diffeomorphism, the descended forms are unique if they exist. For $b\in B_{\cE}$, $v\in T_bB_{\cE}$, and $u\in\kappa_{\cE}^{-1}(b)$, put
\begin{equation}\label{eq:etale-descended-form-definitions}
\begin{aligned}
(\theta_{B_{\cE}})_b(v)&\colonequals (\theta_T)_u\bigl((d\kappa_{\cE})_u^{-1}v\bigr),\\
(\omega_{B_{\cE}})_b(v)&\colonequals (\omega_T)_u\bigl((d\kappa_{\cE})_u^{-1}v\bigr).
\end{aligned}
\end{equation}
If $u_0,u_1\in\kappa_{\cE}^{-1}(b)$, by \eqref{eq:etale-frame-fibre}, there is an arrow $g:x\rightarrow y$, where $x=\pi_T(u_0)$ and $y=\pi_T(u_1)$, such that $u_1=d_x\psi_g\circ u_0$. Choose a local bisection $\sigma:U\rightarrow\cE$ through $g$ and put $h=t\circ\sigma$. Then $OF(h)(u_0)=u_1$, and $\kappa_{\cE}\circ OF(h)=\kappa_{\cE}$ on $OF(T)|_U$, because $OF(h)$ maps each frame into its $\widehat{\cE}$-orbit. Differentiating gives $(d\kappa_{\cE})_{u_1}^{-1}v=dOF(h)_{u_0}\bigl((d\kappa_{\cE})_{u_0}^{-1}v\bigr)$, and by the invariance of $\theta_T$ and $\omega_T$ under $OF(h)$, the two definitions do not depend on $u$.

On an open set $W\subset OF(T)$ on which $\kappa_{\cE}$ is a diffeomorphism onto an open subset, the descended forms are $((\kappa_{\cE}|_W)^{-1})^*\theta_T$ and $((\kappa_{\cE}|_W)^{-1})^*\omega_T$, hence smooth, and $\kappa_{\cE}^*\theta_{B_{\cE}}=\theta_T$ and $\kappa_{\cE}^*\omega_{B_{\cE}}=\omega_T$.

Finally, by Lemma \ref{lem:theta-omega-parallelism}, applied to the foliation of $T$ by points, $(\theta_T,\omega_T)_u$ is a linear isomorphism for every $u\in OF(T)$. Since $(d\kappa_{\cE})_u$ is an isomorphism, by \eqref{eq:etale-descended-forms} the map $(\theta_{B_{\cE}},\omega_{B_{\cE}})_{\kappa_{\cE}(u)}$ is also a linear isomorphism.

(2) By Proposition \ref{prop:OF-P-Morita-bibundle}\textup{(1)}, the maps $\alpha$ and $\alpha'$ are local isometries. Applying the calculation for the canonical form and the naturality of the Levi-Civita connection in the proof of Lemma \ref{lem:functoriality-molino} locally, we therefore obtain
\begin{equation}
\begin{aligned}
OF(\alpha)^*\theta_T&=\theta_P=OF(\alpha')^*\theta_{T'},\\
OF(\alpha)^*\omega_T&=\omega_P=OF(\alpha')^*\omega_{T'}.
\end{aligned}
\end{equation}

For the last part, by Proposition \ref{prop:OF-P-Morita-bibundle}\textup{(3)}, we have $\Phi_P\circ\kappa_{\cE}\circ OF(\alpha)=\kappa_{\cE'}\circ OF(\alpha')$. Hence
\begin{equation}\label{eq:etale-base-form-comparison}
\begin{aligned}
(\kappa_{\cE}\circ OF(\alpha))^*\Phi_P^*\theta_{B_{\cE'}}&=\theta_P\\
&=(\kappa_{\cE}\circ OF(\alpha))^*\theta_{B_{\cE}}.
\end{aligned}
\end{equation}
Since $\kappa_{\cE}\circ OF(\alpha)$ is a surjective local diffeomorphism, pullback of forms along it is injective, so $\Phi_P^*\theta_{B_{\cE'}}=\theta_{B_{\cE}}$. By the same argument, we obtain $\Phi_P^*\omega_{B_{\cE'}}=\omega_{B_{\cE}}$.
\end{proof}

For the right $O(q)$-action on $B_{\cE}$, we write $B_{\cE}\rtimes O(q)\rightrightarrows B_{\cE}$ for the action groupoid with arrow manifold $B_{\cE}\times O(q)$, source $s(b,A)=bA$, target $t(b,A)=b$, multiplication $(b,A)(bA,C)=(b,AC)$, units $1_b=(b,\id)$, and inverses $(b,A)^{-1}=(bA,A^{-1})$. With this convention, a right action of $B_{\cE}\rtimes O(q)$ with moment map $\alpha$ is defined for $\alpha(p)=t(b,A)=b$ and satisfies $\alpha(p\cdot(b,A))=s(b,A)=bA$, as in Definition \ref{def:morita-data-prelim}\textup{(ii)}.

\begin{cor}\label{cor:molino-proper-etale}
Let $\cE\rightrightarrows T$ be a proper effective \'etale Riemannian groupoid, and let $B_{\cE}=OF(T)/\widehat{\cE}$ with its descended right $O(q)$-action.

\begin{enumerate}[label=\textup{(\arabic*)}]
\item The map $B_{\cE}/O(q)\rightarrow T/\cE$, $[\kappa_{\cE}(u)]\mapsto[\pi_T(u)]$, is a homeomorphism.

\item For $u\in OF_x(T)$, the map $\cE_x\rightarrow O(q)$, $g\mapsto u^{-1}\circ d_x\psi_g\circ u$, is an injective group homomorphism whose image is the stabilizer of $\kappa_{\cE}(u)$ in $O(q)$. In particular, all stabilizers of the descended $O(q)$-action are finite.

\item $OF(T)$ is a Morita equivalence bibundle $\cE\curvearrowright OF(T)\curvearrowleft\bigl(B_{\cE}\rtimes O(q)\bigr)$ with moment maps $\pi_T$ and $\kappa_{\cE}$ and actions $g\cdot u=d_{\pi_T(u)}\psi_g\circ u$ and $u\cdot(b,A)=uA$ for $\kappa_{\cE}(u)=b$. In particular, $\cE$ is Morita equivalent to $B_{\cE}\rtimes O(q)$; compare \cite{HenriquesMetzler:NoneffectiveOrbifolds}*{Prop. 5.2}.

\item If $\cE'\rightrightarrows T'$ is a proper effective \'etale Riemannian groupoid and $\cE'\curvearrowright P\curvearrowleft\cE$ is a Riemannian Morita equivalence bibundle, then the $O(q)$-equivariant diffeomorphism $\Phi_P:B_{\cE}\rightarrow B_{\cE'}$ of Proposition \ref{prop:OF-P-Morita-bibundle} induces a Lie groupoid isomorphism $B_{\cE}\rtimes O(q)\rightarrow B_{\cE'}\rtimes O(q)$, given by $b\mapsto\Phi_P(b)$ on objects and $(b,A)\mapsto(\Phi_P(b),A)$ on arrows.
\end{enumerate}
\end{cor}

\begin{proof}
(1) If $\kappa_{\cE}(u)=\kappa_{\cE}(v)$, then $v=d_{\pi_T(u)}\psi_g\circ u$ for an arrow $g:\pi_T(u)\rightarrow\pi_T(v)$, so $\pi_T(u)$ and $\pi_T(v)$ have the same class in $T/\cE$; also $\kappa_{\cE}(u)A=\kappa_{\cE}(uA)$ and $\pi_T(uA)=\pi_T(u)$. Hence the map is well-defined, and it is surjective because every point of $T$ admits an orthonormal frame. If $[\pi_T(u)]=[\pi_T(v)]$, choose an arrow $g:\pi_T(u)\rightarrow\pi_T(v)$; the frames $g\cdot u\colonequals d_{\pi_T(u)}\psi_g\circ u$ and $v$ lie over the same point, so $v=(g\cdot u)A$ for a unique $A\in O(q)$, and $\kappa_{\cE}(v)=\kappa_{\cE}(u)A$. Thus the map is injective.

The quotient map $T\rightarrow T/\cE$ is open, because the saturation $t(s^{-1}(U))$ of an open subset $U\subset T$ is open; the quotient map $B_{\cE}\rightarrow B_{\cE}/O(q)$ is open, because the saturation of an open subset is the union of its right translates. Since $\pi_T$ and $\kappa_{\cE}$ are open, the two composites $OF(T)\rightarrow T/\cE$ and $OF(T)\rightarrow B_{\cE}/O(q)$ are continuous open surjections, and the bijection just constructed identifies their fibres. Hence both spaces carry the quotient topology of $OF(T)$, and the bijection is a homeomorphism.

(2) Fix $u\in OF_x(T)$. By Proposition \ref{prop:lifted-frame-groupoid-properties}, $\kappa_{\cE}(u)A=\kappa_{\cE}(u)$ if and only if $uA=d_x\psi_g\circ u$ for some arrow $g$ with $s(g)=x$; since $\pi_T(uA)=x$, such a $g$ lies in $\cE_x$, and then $A=u^{-1}\circ d_x\psi_g\circ u$. By the identity $d_x\psi_{hg}=d_x\psi_h\circ d_x\psi_g$ for $g,h\in\cE_x$, the map $g\mapsto u^{-1}\circ d_x\psi_g\circ u$ is a group homomorphism, and its kernel is trivial because $\widehat{\cE}$ is free. Hence it identifies $\cE_x$ with the stabilizer of $\kappa_{\cE}(u)$. Finally, $\cE_x$ is compact by properness and discrete by \'etaleness, hence finite.

(3) Both moment maps $\pi_T$ and $\kappa_{\cE}$ are surjective submersions. The left action is the lifted action of Proposition \ref{prop:action-axioms}; it satisfies $\pi_T(g\cdot u)=t(g)$ and $\kappa_{\cE}(g\cdot u)=\kappa_{\cE}(u)$. The right action is smooth, and for $\kappa_{\cE}(u)=b$ we have $\kappa_{\cE}(u\cdot(b,A))=\kappa_{\cE}(uA)=bA=s(b,A)$, $\pi_T(uA)=\pi_T(u)$, $u\cdot(b,\id)=u$, and $(u\cdot(b,A))\cdot(bA,C)=u(AC)$, so it is an action. The two actions commute because $g\cdot(uA)=(g\cdot u)A$.

The right action is principal over $T$: after substituting $b=\kappa_{\cE}(u)$, the map $OF(T)\times O(q)\rightarrow OF(T)\times_T OF(T)$, $(u,A)\mapsto(u,uA)$, is a diffeomorphism since $\pi_T:OF(T)\rightarrow T$ is a principal $O(q)$-bundle.

The left action is principal over $B_{\cE}$: the map $\cE_1\times_T OF(T)\rightarrow OF(T)\times_{B_{\cE}}OF(T)$, $(g,u)\mapsto(g\cdot u,u)$, is surjective because the fibres of $\kappa_{\cE}$ are the $\widehat{\cE}$-orbits, and injective because $\widehat{\cE}$ is free. The projection of the domain onto its second factor is $\widehat s$, a local diffeomorphism since $\widehat{\cE}$ is \'etale. The projection of the codomain onto its second factor is the pullback of the local diffeomorphism $\kappa_{\cE}$ along $\kappa_{\cE}$, hence a local diffeomorphism as well. This smooth map commutes with these projections, so it is a bijective local diffeomorphism, hence a diffeomorphism.

(4) By the $O(q)$-equivariance of $\Phi_P$, the maps in the statement define a Lie groupoid isomorphism. Its inverse is given by $\Phi_P^{-1}$ on objects and $(b',A)\mapsto(\Phi_P^{-1}(b'),A)$ on arrows.
\end{proof}

The Morita equivalence of Corollary \ref{cor:molino-proper-etale}\textup{(3)} can also be made Riemannian by choosing a compatible $2$-metric on $B_{\cE}\rtimes O(q)$.

\begin{cor}\label{cor:etale-compatible-2metric}
Let $\cE\rightrightarrows T$ be a proper effective \'etale Riemannian groupoid. Then $B_{\cE}\rtimes O(q)$ admits a $2$-metric $\eta'$ with the following property. If $\eta'^{(0)}$ denotes its induced $0$-metric on $B_{\cE}$, then the metric $\kappa_{\cE}^*\eta'^{(0)}$ on $OF(T)$ makes the bibundle of Corollary \ref{cor:molino-proper-etale}\textup{(3)} a Riemannian Morita equivalence bibundle.
\end{cor}

\begin{proof}
Let $\eta^{(0)}$ be the given $0$-metric on $T$, and let $\eta$ be its unique extension to a $2$-metric on $\cE$.

We first verify that $\cE_1$ is Hausdorff. Each germ $\psi_g$ is represented by a local isometry for $\eta^{(0)}$, so the map
\begin{equation}
g\longmapsto \bigl(s(g),t(g),d_{s(g)}\psi_g\bigr)
\end{equation}
takes values in the Hausdorff manifold of linear isometries between tangent spaces of $T$. Since $\cE$ is \'etale, we have
\begin{equation}\label{eq:etale-normal-representation}
d_{s(g)}\psi_g=dt_g\circ(ds_g)^{-1},
\end{equation}
so this map is smooth.

Suppose that $g,h:x\rightarrow y$ have the same image. Then
\begin{equation}
d_x\psi_{h^{-1}g}=(d_x\psi_h)^{-1}\circ d_x\psi_g=\id_{T_xT}.
\end{equation}
By the exponential map argument in the proof of Proposition \ref{prop:lifted-frame-groupoid-properties}, $\psi_{h^{-1}g}$ is the identity germ. Since $\cE$ is effective, $h^{-1}g=1_x$, hence $g=h$. Thus the map is injective. Disjoint open neighbourhoods of the images of two distinct arrows pull back to disjoint open neighbourhoods of the arrows, so $\cE_1$ is Hausdorff.

Let $\mathcal K\rightrightarrows OF(T)$ be the pullback groupoid of $\cE$ through $\pi_T$. Using the isomorphism between the two pullback groupoids in Lemma \ref{lem:bibundle-pullback-groupoid} for the bibundle of Corollary \ref{cor:molino-proper-etale}\textup{(3)}, we obtain Morita fibrations
\begin{equation}
z:\mathcal K\longrightarrow\cE,
\qquad
z':\mathcal K\longrightarrow B_{\cE}\rtimes O(q)
\end{equation}
with object maps $\pi_T$ and $\kappa_{\cE}$, respectively. The arrow manifold of $\mathcal K$ is a fibre product inside $OF(T)\times\cE_1\times OF(T)$, and is therefore Hausdorff. The arrow manifold $B_{\cE}\times O(q)$ is Hausdorff as well.

We next identify the normal spaces of the $\mathcal K$-orbits. By the definition of the pullback groupoid, for every $u\in OF(T)$,
\begin{equation}\label{eq:etale-pullback-orbit}
\mathcal K\cdot u=\pi_T^{-1}\bigl(\cE\cdot\pi_T(u)\bigr).
\end{equation}
Since $\cE$ is proper and \'etale, its orbits are discrete subsets of $T$, so $T_u(\mathcal K\cdot u)=\ker(d\pi_T)_u$. The differential of $\pi_T$ therefore identifies the normal space to the orbit at $u$ with $T_{\pi_T(u)}T$:
\begin{equation}\label{eq:etale-common-normal-identification}
\begin{aligned}
T_uOF(T)/\ker(d\pi_T)_u &\longrightarrow T_{\pi_T(u)}T, \\
[v] &\longmapsto (d\pi_T)_u(v).
\end{aligned}
\end{equation}

By \cite{MdHFStackMetrics}*{Prop. 6.3.1}, there is a $2$-metric $\tilde\eta$ on $\mathcal K$ making $z$ a Riemannian submersion in the sense of \cite{MdHFStackMetrics}*{Definition 3.2.1}. In particular, $\pi_T:\bigl(OF(T),\tilde\eta^{(0)}\bigr)\rightarrow\bigl(T,\eta^{(0)}\bigr)$ is a Riemannian submersion. Hence, for each $u\in OF(T)$, the map $[v]\mapsto(d\pi_T)_u(v)$ is an isometry for the metric induced by $\tilde\eta^{(0)}$ on the normal space and the metric $\eta^{(0)}$ on $T_{\pi_T(u)}T$.

Applying \cite{MdHFStackMetrics}*{Prop. 6.3.2 and the proof of Thm. 6.3.3} to $z'$ and $\tilde\eta$, we obtain a $2$-metric $\eta'$ on $B_{\cE}\rtimes O(q)$ and a $2$-metric $\tilde\eta'$ on $\mathcal K$ making $z'$ a Riemannian submersion. Moreover, $\tilde\eta^{(0)}$ and $\tilde\eta'^{(0)}$ induce the same metric on the normal bundle of every $\mathcal K$-orbit. Hence, for every $u\in OF(T)$, the map $[v]\mapsto(d\pi_T)_u(v)$ remains an isometry for the metric induced by $\tilde\eta'^{(0)}$ and the given metric $\eta^{(0)}$. It follows from \cite{MdHF-LieGpdMetrics}*{Sec. 2.1} that $\pi_T:(OF(T),\tilde\eta'^{(0)})\rightarrow(T,\eta^{(0)})$ is a Riemannian submersion.

Finally, the object map $\kappa_{\cE}$ of $z'$ is a Riemannian submersion for $\tilde\eta'^{(0)}$ and $\eta'^{(0)}$. Since $\kappa_{\cE}$ is a local diffeomorphism, this means $\tilde\eta'^{(0)}=\kappa_{\cE}^*\eta'^{(0)}$. Both moment maps are therefore Riemannian submersions for the metric $\kappa_{\cE}^*\eta'^{(0)}$, as required by Definition \ref{def:morita-data-prelim}\textup{(iii)}.
\end{proof}

\begin{remark}\label{rem:no-effectiveness-regular}
We used \textit{effectiveness} only in the proper \'etale subsection, where it implies that the lifted frame groupoid $\widehat{\cE}\rightrightarrows OF(T)$ is free (Proposition \ref{prop:lifted-frame-groupoid-properties}).

We did not use this hypothesis in Theorem \ref{thm:molino-groupoid}. There we work directly with the full Lie groupoid $\cG\rightrightarrows M$, local bisections, and the normal representation $\lambda^N$. Isotropy acting trivially on transverse directions is therefore retained as part of the groupoid data in the regular theory.
\end{remark}


We now return to the regular setting.

\subsection{Frame bundles and Morita fibrations}

\begin{lemma}\label{lem:frame-bundle-pullback}
Let $\phi:\bigl(\widetilde{\cG}\rightrightarrows \widetilde M\bigr)\rightarrow\bigl(\cG\rightrightarrows M\bigr)$ be a Morita fibration, and write $\phi_0\colonequals\phi^{(0)}:\widetilde M\rightarrow M$ and $\phi_1\colonequals\phi^{(1)}:\widetilde{\cG}\rightarrow\cG$. Then $\widetilde{\cG}$ is regular if and only if $\cG$ is regular. In this case the orbit foliations $\mathcal O_{\widetilde{\cG}}$ and $\mathcal O$ satisfy $T\mathcal O_{\widetilde{\cG}}=(d\phi_0)^{-1}(T\mathcal O)$, and $d\phi_0$ descends to a vector bundle isomorphism
\begin{equation}\label{eq:morita-normal-map}
\overline{d\phi_0}:N(\mathcal O_{\widetilde{\cG}})\xlongrightarrow{\ \cong\ }\phi_0^*N(\mathcal O),
\qquad
(\overline{d\phi_0})_{\tilde m}([v])=[(d\phi_0)_{\tilde m}(v)].
\end{equation}
In particular, if $\mathcal O$ has constant codimension $q$, then so does $\mathcal O_{\widetilde{\cG}}$.

Assume now that $\mathcal O$ has constant codimension $q$, and equip $\widetilde M$ and $M$ with Riemannian metrics whose induced normal metrics make $\overline{d\phi_0}$ a fibrewise isometry; this holds whenever $\phi_0$ is a Riemannian submersion. Then there is a canonical $O(q)$-equivariant principal bundle isomorphism
\begin{equation}\label{eq:morita-frame-isomorphism}
\begin{aligned}
F_\phi:OF(\widetilde M,\mathcal O_{\widetilde{\cG}}) &\xlongrightarrow{\ \cong\ }\phi_0^*OF(M,\mathcal O), \\
F_\phi(u) &=\bigl(\tilde m,(\overline{d\phi_0})_{\tilde m}\circ u\bigr),
\end{aligned}
\end{equation}
where $u\in OF_{\tilde m}(\widetilde M,\mathcal O_{\widetilde{\cG}})$. The map
\begin{equation}\label{eq:morita-frame-map}
\widehat\phi\colonequals\pr_2\circ F_\phi:OF(\widetilde M,\mathcal O_{\widetilde{\cG}})\longrightarrow OF(M,\mathcal O)
\end{equation}
is an $O(q)$-equivariant surjective submersion with $\pi_M\circ\widehat\phi=\phi_0\circ\pi_{\widetilde M}$, where $\pi_M$ and $\pi_{\widetilde M}$ are the bundle projections, and the square
\begin{equation}\label{eq:morita-frame-pullback-square}
\begin{tikzcd}[row sep=3.2em, column sep=4.5em]
OF(\widetilde M,\mathcal O_{\widetilde{\cG}}) \arrow[r, "\widehat\phi"] \arrow[d, "\pi_{\widetilde M}"'] & OF(M,\mathcal O) \arrow[d, "\pi_M"] \\
\widetilde M \arrow[r, "\phi_0"'] & M
\end{tikzcd}
\end{equation}
is a pullback. For every $e\in OF(M,\mathcal O)$, the projection $\pi_{\widetilde M}$ restricts to a diffeomorphism $\widehat\phi^{-1}(e)\rightarrow\phi_0^{-1}(\pi_M(e))$ with inverse $\tilde m\mapsto(\overline{d\phi_0})_{\tilde m}^{-1}\circ e$. If $\phi_0$ is proper, then $\widehat\phi$ is proper.
\end{lemma}

\begin{proof}
We first compare the orbits. Let $\tilde m\in\widetilde M$ and $x\colonequals\phi_0(\tilde m)$. Let $O_x\subset M$ be the $\cG$-orbit through $x$, and let $\widetilde O_{\tilde m}\subset\widetilde M$ be the $\widetilde{\cG}$-orbit through $\tilde m$.

Since $\phi$ is fully faithful, the square
\begin{equation}\label{eq:morita-full-faithfulness-square}
\begin{tikzcd}[row sep=3.8em, column sep=7em]
\widetilde{\cG} \arrow[r, "\phi_1"] \arrow[d, "{(s_{\widetilde{\cG}},t_{\widetilde{\cG}})}"'] & \cG \arrow[d, "{(s_{\cG},t_{\cG})}"] \\
\widetilde M\times\widetilde M \arrow[r, "{\phi_0\times\phi_0}"'] & M\times M
\end{tikzcd}
\end{equation}
is a pullback. Hence, for each pair $\tilde m,\tilde m'\in\widetilde M$, the induced map $\widetilde{\cG}(\tilde m,\tilde m')\rightarrow\cG(\phi_0(\tilde m),\phi_0(\tilde m'))$ is a diffeomorphism.

By the preceding diffeomorphism, $\tilde m'\in\widetilde O_{\tilde m}$ if and only if $\phi_0(\tilde m')\in O_x$. Hence $\widetilde O_{\tilde m}=\phi_0^{-1}(O_x)$.

Differentiating \eqref{eq:morita-full-faithfulness-square} at the unit arrows and then restricting to the kernels of the source differentials gives
\begin{equation}\label{eq:morita-pullback-algebroid-fibre}
(A_{\widetilde{\cG}})_{\tilde m}\cong\left\{(a,v)\in(A_{\cG})_x\oplus T_{\tilde m}\widetilde M\,\middle|\,\rho_{\cG,x}(a)=(d\phi_0)_{\tilde m}(v)\right\}.
\end{equation}
Under this identification, the anchor of $A_{\widetilde{\cG}}$ is $(a,v)\mapsto v$. It follows that
\begin{equation}\label{eq:morita-anchor-image}
\begin{aligned}
\operatorname{im}\rho_{\widetilde{\cG},\tilde m} &=\left\{v\in T_{\tilde m}\widetilde M\,\middle|\,(d\phi_0)_{\tilde m}(v)\in\operatorname{im}\rho_{\cG,x}\right\} \\
&=(d\phi_0)_{\tilde m}^{-1}(\operatorname{im}\rho_{\cG,x}),
\end{aligned}
\end{equation}
and, since $\phi_0$ is a submersion,
\begin{equation}\label{eq:morita-anchor-rank}
\rank\rho_{\widetilde{\cG},\tilde m}=\dim\ker(d\phi_0)_{\tilde m}+\rank\rho_{\cG,x}.
\end{equation}

If $\cG$ is regular, then $T\mathcal O=\operatorname{im}\rho_{\cG}$ is a smooth subbundle of $TM$, and the inverse image of $\phi_0^*T\mathcal O$ under the surjective bundle map $d\phi_0:T\widetilde M\rightarrow\phi_0^*TM$ is a smooth subbundle of $T\widetilde M$; hence $\rank\rho_{\widetilde{\cG}}$ is locally constant and $\widetilde{\cG}$ is regular. Conversely, suppose that $\widetilde{\cG}$ is regular and choose a smooth local section of $\phi_0$ over a neighbourhood of $x$. By \eqref{eq:morita-anchor-rank}, the rank of $\rho_{\cG}$ on this neighbourhood is the difference of two locally constant functions. Hence $\cG$ is regular. In the regular case, by \eqref{eq:morita-anchor-image}, we have $T_{\tilde m}\mathcal O_{\widetilde{\cG}}=(d\phi_0)_{\tilde m}^{-1}(T_x\mathcal O)$. Consequently, $(d\phi_0)_{\tilde m}\bigl(T_{\tilde m}\mathcal O_{\widetilde{\cG}}\bigr)=T_x\mathcal O$, and $\ker(d\phi_0)_{\tilde m}\subset T_{\tilde m}\mathcal O_{\widetilde{\cG}}$.

Writing $V_{\tilde m}\colonequals\ker(d\phi_0)_{\tilde m}$ and using what we just discussed, we obtain the exact sequence
\begin{equation}\label{eq:morita-orbit-tangent-sequence}
0\longrightarrow V_{\tilde m}\longrightarrow T_{\tilde m}\mathcal O_{\widetilde{\cG}}\xlongrightarrow{\ (d\phi_0)_{\tilde m}\ }T_x\mathcal O\longrightarrow 0.
\end{equation}

Since $\phi_0$ is a submersion, the two orbit foliations have the same codimension at corresponding points by \eqref{eq:morita-orbit-tangent-sequence}.

Since $(d\phi_0)_{\tilde m}$ sends $T_{\tilde m}\mathcal O_{\widetilde{\cG}}$ into $T_x\mathcal O$, the induced map $(\overline{d\phi_0})_{\tilde m}$ on normal quotients is well-defined.

Thus the following diagram commutes, with the natural inclusion and quotient maps as horizontal arrows:
\begin{equation}
\begin{tikzcd}[row sep=3.2em, column sep=4.5em]
T_{\tilde m}\mathcal O_{\widetilde{\cG}} \arrow[r, hook] \arrow[d, two heads, "{(d\phi_0)_{\tilde m}}"'] & T_{\tilde m}\widetilde M \arrow[r] \arrow[d, two heads, "{(d\phi_0)_{\tilde m}}"] & N_{\tilde m}(\mathcal O_{\widetilde{\cG}}) \arrow[d, "{(\overline{d\phi_0})_{\tilde m}}", "\cong"'] \\
T_x\mathcal O \arrow[r, hook] & T_xM \arrow[r] & N_x(\mathcal O).
\end{tikzcd}
\end{equation}
By the equality $(d\phi_0)_{\tilde m}^{-1}(T_x\mathcal O)=T_{\tilde m}\mathcal O_{\widetilde{\cG}}$, the induced map on quotients is injective, while it is surjective because $(d\phi_0)_{\tilde m}:T_{\tilde m}\widetilde M\rightarrow T_xM$ is surjective.

Since $d\phi_0$ is smooth and the orbit distributions are smooth subbundles, the induced maps assemble into a smooth vector bundle isomorphism $\overline{d\phi_0}$.

We next show that this normal isomorphism is an isometry when $\phi_0$ is a Riemannian submersion for the given metrics. Let $H_{\tilde m}\colonequals V_{\tilde m}^{\perp}\subset T_{\tilde m}\widetilde M$ be the horizontal space of this Riemannian submersion. Then
\begin{equation}
(d\phi_0)_{\tilde m}|_{H_{\tilde m}}:H_{\tilde m}\longrightarrow T_xM
\end{equation}
is a linear isometry. Also, because $V_{\tilde m}\subset T_{\tilde m}\mathcal O_{\widetilde{\cG}}$, we have $(T_{\tilde m}\mathcal O_{\widetilde{\cG}})^\perp\subset H_{\tilde m}$.

The metric part of the argument is summarized by the following diagram:
\begin{equation}
\begin{tikzcd}[row sep=3.2em, column sep=4.5em]
(T_{\tilde m}\mathcal O_{\widetilde{\cG}})^\perp \arrow[r, hook] \arrow[d, "{(d\phi_0)_{\tilde m}}"'] & H_{\tilde m}=V_{\tilde m}^\perp \arrow[d, "{(d\phi_0)_{\tilde m}}", "\cong"'] \\
(T_x\mathcal O)^\perp \arrow[r, hook] & T_xM.
\end{tikzcd}
\end{equation}
We now justify the left vertical arrow. Let $v\in(T_{\tilde m}\mathcal O_{\widetilde{\cG}})^\perp$ and $w\in T_x\mathcal O$. Choose $\widetilde w\in T_{\tilde m}\mathcal O_{\widetilde{\cG}}$ such that $(d\phi_0)_{\tilde m}(\widetilde w)=w$. Decompose $\widetilde w$ with respect to the Riemannian submersion: $\widetilde w=\widetilde w^h+\widetilde w^v$, where $\widetilde w^h\in H_{\tilde m}$ and $\widetilde w^v\in V_{\tilde m}$. Since $V_{\tilde m}\subset T_{\tilde m}\mathcal O_{\widetilde{\cG}}$, we also have $\widetilde w^h=\widetilde w-\widetilde w^v\in T_{\tilde m}\mathcal O_{\widetilde{\cG}}$. Moreover, $(d\phi_0)_{\tilde m}(\widetilde w^h)=w$.

Both $v$ and $\widetilde w^h$ lie in $H_{\tilde m}$, and $(d\phi_0)_{\tilde m}|_{H_{\tilde m}}$ is an isometry. Therefore
\begin{equation}
\begin{aligned}
\big\langle(d\phi_0)_{\tilde m}(v),w\big\rangle_M &=\big\langle(d\phi_0)_{\tilde m}(v),(d\phi_0)_{\tilde m}(\widetilde w^h)\big\rangle_M \\
&=\langle v,\widetilde w^h\rangle_{\widetilde M} \\
&=0,
\end{aligned}
\end{equation}
because $v\perp T_{\tilde m}\mathcal O_{\widetilde{\cG}}$ and $\widetilde w^h\in T_{\tilde m}\mathcal O_{\widetilde{\cG}}$. Thus $(d\phi_0)_{\tilde m}\bigl((T_{\tilde m}\mathcal O_{\widetilde{\cG}})^\perp\bigr)\subset(T_x\mathcal O)^\perp$. Since $(\overline{d\phi_0})_{\tilde m}$ is already a linear isomorphism on normal quotients, the restriction is also a linear isomorphism:
\begin{equation}
(d\phi_0)_{\tilde m}:(T_{\tilde m}\mathcal O_{\widetilde{\cG}})^\perp\xlongrightarrow{\cong}(T_x\mathcal O)^\perp.
\end{equation}
It is also an isometry, because it is the restriction of the linear isometry $(d\phi_0)_{\tilde m}|_{H_{\tilde m}}:H_{\tilde m}\rightarrow T_xM$.

Using the metric identifications $N_{\tilde m}(\mathcal O_{\widetilde{\cG}})\cong(T_{\tilde m}\mathcal O_{\widetilde{\cG}})^\perp$ and $N_x(\mathcal O)\cong(T_x\mathcal O)^\perp$, we conclude that $(\overline{d\phi_0})_{\tilde m}$ is a linear isometry for the induced normal metrics. Hence $\overline{d\phi_0}$ is fibrewise an isometry.

Now assume only that $\overline{d\phi_0}$ is fibrewise isometric, and construct the induced maps on transverse orthonormal frame bundles. The pullback principal bundle is $\phi_0^*OF(M,\mathcal O)=\widetilde M\times_M OF(M,\mathcal O)$. Define
\begin{equation}
F_\phi:OF(\widetilde M,\mathcal O_{\widetilde{\cG}})\longrightarrow\phi_0^*OF(M,\mathcal O)
\end{equation}
by $F_\phi(u)=\bigl(\tilde m,(\overline{d\phi_0})_{\tilde m}\circ u\bigr)$, where $u\in OF_{\tilde m}(\widetilde M,\mathcal O_{\widetilde{\cG}})$.

This fits into the diagram
\begin{equation}
\begin{tikzcd}[row sep=3.2em, column sep=4.5em]
OF(\widetilde M,\mathcal O_{\widetilde{\cG}}) \arrow[r, "{F_\phi}"] \arrow[dr, "\pi_{\widetilde M}"'] & \phi_0^*OF(M,\mathcal O) \arrow[d, "\pr_1"] \arrow[r, "\pr_2"] & OF(M,\mathcal O) \arrow[d, "\pi_M"] \\
& \widetilde M \arrow[r, "\phi_0"'] & M
\end{tikzcd}.
\end{equation}
Since each $(\overline{d\phi_0})_{\tilde m}$ is an isometry, $(\overline{d\phi_0})_{\tilde m}\circ u$ is indeed an orthonormal frame over $\phi_0(\tilde m)$.

The map $F_\phi$ is smooth because it is induced by the smooth vector bundle isomorphism $\overline{d\phi_0}$. It is $O(q)$-equivariant since, for $A\in O(q)$,
\begin{equation}
\begin{aligned}
F_\phi(u\cdot A) &=\bigl(\tilde m,(\overline{d\phi_0})_{\tilde m}\circ u\circ A\bigr) \\
&=\bigl(\tilde m,((\overline{d\phi_0})_{\tilde m}\circ u)\circ A\bigr) \\
&=F_\phi(u)\cdot A.
\end{aligned}
\end{equation}

It is bijective on each fibre, with inverse $F_\phi^{-1}(\tilde m,e)=(\overline{d\phi_0})_{\tilde m}^{-1}\circ e$. The inverse is smooth because $\overline{d\phi_0}$ is a smooth vector bundle isomorphism. Therefore $F_\phi$ is a principal $O(q)$-bundle isomorphism over $\widetilde M$.

Under $F_\phi$, the map $\widehat\phi=\pr_2\circ F_\phi$ becomes the projection $\widetilde M\times_M OF(M,\mathcal O)\rightarrow OF(M,\mathcal O)$, which is the pullback of the surjective submersion $\phi_0$ along $\pi_M$. Hence $\widehat\phi$ is a surjective submersion, and it is $O(q)$-equivariant because $F_\phi$ and $\pr_2$ are. The square \eqref{eq:morita-frame-pullback-square} is a pullback, and $\widehat\phi^{-1}(e)=F_\phi^{-1}\bigl(\phi_0^{-1}(\pi_M(e))\times\{e\}\bigr)$ is mapped by $\pi_{\widetilde M}$ diffeomorphically onto $\phi_0^{-1}(\pi_M(e))$, with the stated inverse. If $\phi_0$ is proper and $K\subset OF(M,\mathcal O)$ is compact, then $\widehat\phi^{-1}(K)$ corresponds under $F_\phi$ to a closed subset of the compact set $\phi_0^{-1}(\pi_M(K))\times K$, hence is compact; so $\widehat\phi$ is proper.
\end{proof}

\begin{proposition}\label{prop:frame-action-morita-fibration}
Under the hypotheses of the second part of Lemma \ref{lem:frame-bundle-pullback}, assume that the metric on $M$ is a $0$-metric for $\cG$. Then the metric on $\widetilde M$ is a $0$-metric for $\widetilde{\cG}$; more precisely, the normal representations of Definition \ref{def:normal-orbit} satisfy
\begin{equation}\label{eq:morita-normal-equivariance}
(\overline{d\phi_0})_{t(h)}\circ\lambda_h=\lambda_{\phi_1(h)}\circ(\overline{d\phi_0})_{s(h)}
\qquad\text{for every }h\in\widetilde{\cG}.
\end{equation}
Moreover, $\widehat\phi(h\cdot u)=\phi_1(h)\cdot\widehat\phi(u)$ for the lifted actions of Proposition \ref{prop:action-axioms}, and the maps $\widehat\phi$ on objects and $(h,u)\mapsto(\phi_1(h),\widehat\phi(u))$ on arrows define an $O(q)$-equivariant Morita fibration
\begin{equation}
\widetilde{\cG}\ltimes OF(\widetilde M,\mathcal O_{\widetilde{\cG}})\longrightarrow\cG\ltimes OF(M,\mathcal O)
\end{equation}
between the lifted action groupoids of Proposition \ref{prop:lifted-action-groupoid}.
\end{proposition}

\begin{proof}
Let $h:\tilde m\rightarrow\tilde m'$ be an arrow of $\widetilde{\cG}$, let $v\in T_{\tilde m}\widetilde M$, and choose $\xi\in T_h\widetilde{\cG}$ with $ds_{\widetilde{\cG}}(\xi)=v$, so that $\lambda_h([v])=[dt_{\widetilde{\cG}}(\xi)]$ by Definition \ref{def:normal-orbit}. Differentiating $s_{\cG}\circ\phi_1=\phi_0\circ s_{\widetilde{\cG}}$ and $t_{\cG}\circ\phi_1=\phi_0\circ t_{\widetilde{\cG}}$ gives $ds_{\cG}(d\phi_1\xi)=d\phi_0(v)$ and $dt_{\cG}(d\phi_1\xi)=d\phi_0(dt_{\widetilde{\cG}}\xi)$, so $\lambda_{\phi_1(h)}([d\phi_0(v)])=[d\phi_0(dt_{\widetilde{\cG}}\xi)]=(\overline{d\phi_0})_{\tilde m'}(\lambda_h([v]))$. This is the naturality of the normal representation \cite{MdHF-LieGpdMetrics}*{Remark 2.6}.

Since $\overline{d\phi_0}$ is fibrewise isometric and $\lambda_{\phi_1(h)}$ is an isometry, $\lambda_h=(\overline{d\phi_0})_{\tilde m'}^{-1}\circ\lambda_{\phi_1(h)}\circ(\overline{d\phi_0})_{\tilde m}$ is an isometry, which is the $0$-metric condition on $\widetilde M$.

By \eqref{eq:morita-normal-equivariance}, we have $\widehat\phi(h\cdot u)=(\overline{d\phi_0})_{\tilde m'}\circ\lambda_h\circ u=\lambda_{\phi_1(h)}\circ\widehat\phi(u)=\phi_1(h)\cdot\widehat\phi(u)$. Hence the proposed arrow map preserves sources and targets, and it preserves multiplication because $(\phi_1(k),\widehat\phi(h\cdot u))(\phi_1(h),\widehat\phi(u))=(\phi_1(kh),\widehat\phi(u))$; units and inverses are preserved because $\phi$ preserves them. The arrow map is smooth because $\phi_1$ and $\widehat\phi$ are smooth.

The object map $\widehat\phi$ is a surjective submersion by Lemma \ref{lem:frame-bundle-pullback}. For full faithfulness, let $u_0,u_1\in OF(\widetilde M,\mathcal O_{\widetilde{\cG}})$ and let $g$ be an arrow of $\cG$ with $g\cdot\widehat\phi(u_0)=\widehat\phi(u_1)$. Writing $\tilde m_i\colonequals\pi_{\widetilde M}(u_i)$, we have $s(g)=\phi_0(\tilde m_0)$ and $t(g)=\phi_0(\tilde m_1)$, so by full faithfulness of $\phi$, there is a unique arrow $h:\tilde m_0\rightarrow\tilde m_1$ with $\phi_1(h)=g$. Then $\widehat\phi(h\cdot u_0)=g\cdot\widehat\phi(u_0)=\widehat\phi(u_1)$, and $\widehat\phi$ is injective on the frame fibre over $\tilde m_1$, so $h\cdot u_0=u_1$. Thus the smooth map $(h,u_0)\mapsto(h\cdot u_0,\phi_1(h),u_0)$ is bijective onto the set of triples $(u_1,g,u_0)$ satisfying $g\cdot\widehat\phi(u_0)=\widehat\phi(u_1)$. Its inverse sends $(u_1,g,u_0)$ to $(h,u_0)$. By \eqref{eq:morita-full-faithfulness-square} for $\phi$, $h$ depends smoothly on $(\tilde m_1,g,\tilde m_0)$, so this inverse is smooth. Finally, right translation acts on arrows by $(h,u)A=(h,uA)$, and $\widehat\phi(uA)=\widehat\phi(u)A$, so the Morita fibration is $O(q)$-equivariant.
\end{proof}

\begin{remark}
By \eqref{eq:morita-full-faithfulness-square} for $\phi$, the $\widetilde{\cG}$-orbit through $\tilde m$ is $\phi_0^{-1}(O_{\phi_0(\tilde m)})$. When the fibres of $\phi_0$ are disconnected, the inverse image of a connected leaf need not be a single connected leaf. Note also that $F_\phi$ is an isomorphism onto a pullback principal bundle, whereas $\widehat\phi$ is a surjective submersion onto the original frame bundle.
\end{remark}

\subsection{Comparison of Molino's structures in the regular case}

\begin{lemma}\label{lem:morita-parallelism}
Under the hypotheses of the second part of Lemma \ref{lem:frame-bundle-pullback}, assume moreover that $(M,\mathcal O)$ is a Riemannian foliation for the given metric on $M$; by Proposition \ref{prop:RF}, this holds when that metric is a $0$-metric for $\cG$. Then $(\widetilde M,\mathcal O_{\widetilde{\cG}})$ is a Riemannian foliation for the given metric on $\widetilde M$.

Let $\widetilde{\mathcal O}_{\widetilde{\cG}}$ and $\widetilde{\mathcal O}$ be the lifted foliations on $OF(\widetilde M,\mathcal O_{\widetilde{\cG}})$ and $OF(M,\mathcal O)$, respectively, and let $(\theta_{\widetilde{\cG}},\omega_{\widetilde{\cG}})$ and $(\theta,\omega)$ be the corresponding transverse canonical forms and transverse Levi-Civita connection forms. Then $\widehat\phi$ is a foliated surjective submersion, and $T\widetilde{\mathcal O}_{\widetilde{\cG}}=(d\widehat\phi)^{-1}(T\widetilde{\mathcal O})$. Moreover, $\widehat\phi^*\theta=\theta_{\widetilde{\cG}}$ and $\widehat\phi^*\omega=\omega_{\widetilde{\cG}}$.
\end{lemma}

\begin{proof}
Let $(s_i:U_i\rightarrow T_i)$ be a Haefliger cocycle for $\mathcal O$ with transverse metrics $g_i$ as in Remark \ref{rem:riemannian-cocycle}. By Lemma \ref{lem:frame-bundle-pullback}, on $\phi_0^{-1}(U_i)$ we have $\ker d(s_i\circ\phi_0)=(d\phi_0)^{-1}(\ker ds_i)=(d\phi_0)^{-1}(T\mathcal O)=T\mathcal O_{\widetilde{\cG}}$, and the induced map on normal spaces is $(d(s_i\circ\phi_0))^N=(ds_i)^N\circ\overline{d\phi_0}$, which is isometric. Choose foliation charts in $\phi_0^{-1}(U_i)$ small enough that the fibres of $s_i\circ\phi_0$ in each chart are its plaques. The restrictions of $s_i\circ\phi_0$ to these charts form a Haefliger cocycle for $\mathcal O_{\widetilde{\cG}}$. The transverse metrics are the restrictions of the $g_i$, and the transition maps are restrictions of the local isometries. Thus the metric on $\widetilde M$ is bundle-like for $\mathcal O_{\widetilde{\cG}}$. Using this cocycle in Definitions \ref{def:lifted-foliation} and \ref{def:theta-omega}, we write $\tilde s_i$ for the lift of $s_i$ and, on each such chart, $\tilde s_i^{\widetilde{\cG}}$ for the lift of the corresponding restriction of $s_i\circ\phi_0$.

By definition, $\tilde s_i(e)=(ds_i)^N_{\pi_M(e)}\circ e$ and $\tilde s_i^{\widetilde{\cG}}(u)=(d(s_i\circ\phi_0))^N_{\pi_{\widetilde M}(u)}\circ u$. Then for $u$ based at $\tilde m$ and $x=\phi_0(\tilde m)$, we have
\begin{equation}\label{eq:morita-lifted-cocycle}
\begin{aligned}
\tilde s_i^{\widetilde{\cG}}(u) &=(ds_i)^N_x\circ(\overline{d\phi_0})_{\tilde m}\circ u \\
&=\tilde s_i((\overline{d\phi_0})_{\tilde m}\circ u) \\
&=\tilde s_i(\widehat\phi(u)),
\end{aligned}
\end{equation}
so $\tilde s_i^{\widetilde{\cG}}=\tilde s_i\circ\widehat\phi$ on the domain of $\tilde s_i^{\widetilde{\cG}}$.

By Lemma \ref{lem:frame-bundle-pullback}, $\widehat\phi$ is a surjective submersion. It is foliated: locally the lifted foliation on $OF(M,\mathcal O)$ is given by the fibres of $\tilde s_i$, while the lifted foliation on $OF(\widetilde M,\mathcal O_{\widetilde{\cG}})$ is given by the fibres of $\tilde s_i^{\widetilde{\cG}}=\tilde s_i\circ\widehat\phi$. Taking kernels of the differentials gives $T\widetilde{\mathcal O}_{\widetilde{\cG}}=(d\widehat\phi)^{-1}(T\widetilde{\mathcal O})$.

By Definition \ref{def:theta-omega} and Proposition \ref{prop:transverse-LC-local}, $\theta$ and $\omega$ are locally the pullbacks along $\tilde s_i$ of the canonical form and the Levi-Civita connection form on $OF(T_i)$. On the domain of $\tilde s_i^{\widetilde{\cG}}$, $\theta_{\widetilde{\cG}}$ and $\omega_{\widetilde{\cG}}$ are the pullbacks of the same forms along $\tilde s_i^{\widetilde{\cG}}$. Using the identity $\tilde s_i^{\widetilde{\cG}}=\tilde s_i\circ\widehat\phi$, we therefore obtain
\begin{equation}\label{eq:morita-frame-form-pullbacks}
\widehat\phi^*\theta=\theta_{\widetilde{\cG}},
\qquad
\widehat\phi^*\omega=\omega_{\widetilde{\cG}}
\end{equation}
on each such domain. Since these domains cover $OF(\widetilde M,\mathcal O_{\widetilde{\cG}})$, both identities hold globally.

\end{proof}


\begin{corollary}\label{cor:theta-omega-morita}
Let $\cG\rightrightarrows M$ and $\cG'\rightrightarrows M'$ be Lie groupoids with $0$-metrics $\eta$ on $M$ and $\eta'$ on $M'$. Let $\mathcal K_P\rightrightarrows P$ be a Lie groupoid with $P$ Hausdorff, and let
\begin{equation}
\cG'\xlongleftarrow{\ z'\ }\mathcal K_P\xlongrightarrow{\ z\ }\cG
\end{equation}
be Morita fibrations with maps on objects $\alpha':P\rightarrow M'$ and $\alpha:P\rightarrow M$. Assume that $\cG$ is regular and that its orbit foliation $\mathcal O$ has constant codimension $q$. Then $\mathcal K_P$ and $\cG'$ are regular, their orbit foliations $\mathcal O_P$ and $\mathcal O'$ have codimension $q$, and $T\mathcal O_P=(d\alpha)^{-1}(T\mathcal O)=(d\alpha')^{-1}(T\mathcal O')$.

Let $\eta_P$ be a Riemannian metric on $P$ for which $\alpha$ and $\alpha'$ are Riemannian submersions onto $M$ and $M'$ equipped with $\eta$ and $\eta'$; for a Riemannian Morita equivalence bibundle in the sense of Definition \ref{def:morita-data-prelim}\textup{(iii)}, this is part of the data. Let
\begin{equation}
\widehat\alpha:OF(P,\mathcal O_P)\longrightarrow OF(M,\mathcal O),
\qquad
\widehat\alpha':OF(P,\mathcal O_P)\longrightarrow OF(M',\mathcal O')
\end{equation}
be the frame maps of Lemma \ref{lem:frame-bundle-pullback} for $z$ and $z'$. Then:
\begin{enumerate}[label=\textup{(\arabic*)}]
\item $\eta_P$ is a $0$-metric for $\mathcal K_P$, and $(P,\mathcal O_P)$ is a Riemannian foliation for $\eta_P$. The maps $\overline{d\alpha}$ and $\overline{d\alpha'}$ induce $O(q)$-equivariant principal bundle isomorphisms over $P$,
\begin{equation}
OF(P,\mathcal O_P)\cong\alpha^*OF(M,\mathcal O),
\qquad
OF(P,\mathcal O_P)\cong(\alpha')^*OF(M',\mathcal O'),
\end{equation}
and $\widehat\alpha$ and $\widehat\alpha'$ are $O(q)$-equivariant surjective submersions inducing Morita fibrations of the lifted action groupoids.

\item $\widehat\alpha$ and $\widehat\alpha'$ are foliated,
\begin{equation}\label{eq:morita-lifted-tangent-preimage}
T\widetilde{\mathcal O}_P=(d\widehat\alpha)^{-1}(T\widetilde{\mathcal O})=(d\widehat\alpha')^{-1}(T\widetilde{\mathcal O}'),
\end{equation}
and
\begin{equation}\label{eq:morita-three-form-pullbacks}
\widehat\alpha^*\theta=\theta_P=(\widehat\alpha')^*\theta',
\qquad
\widehat\alpha^*\omega=\omega_P=(\widehat\alpha')^*\omega'
\end{equation}
as forms on $OF(P,\mathcal O_P)$. Here $(\theta_P,\omega_P)$ and $(\theta',\omega')$ denote the transverse canonical forms and transverse Levi-Civita connection forms on $OF(P,\mathcal O_P)$ and $OF(M',\mathcal O')$, respectively.
\end{enumerate}
\end{corollary}

\begin{proof}
The Morita fibrations fit into the diagram of Lie groupoids
\begin{equation}\label{eq:morita-fibrations-diagram}
\begin{tikzcd}[column sep=4em]
\cG' \arrow[d] & \mathcal K_P \arrow[l,"z'"'] \arrow[r,"z"] \arrow[d] & \cG \arrow[d] \\
M' & P \arrow[l,"\alpha'"'] \arrow[r,"\alpha"] & M.
\end{tikzcd}
\end{equation}

Since $\alpha$ and $\alpha'$ are Riemannian submersions for $\eta_P$, by Lemma \ref{lem:frame-bundle-pullback}, $\overline{d\alpha}$ and $\overline{d\alpha'}$ are fibrewise isometric. Applying the second part of that lemma, we therefore obtain the two principal bundle isomorphisms and the frame maps.

Since $\eta$ and $\eta'$ are $0$-metrics, $(M,\mathcal O)$ and $(M',\mathcal O')$ are Riemannian foliations by Proposition \ref{prop:RF}. By Proposition \ref{prop:frame-action-morita-fibration}, applied to $z$ and $z'$, $\eta_P$ is a $0$-metric, and we obtain the two Morita fibrations of the lifted action groupoids; by Lemma \ref{lem:morita-parallelism}, applied to $z$ and $z'$, $(P,\mathcal O_P)$ is a Riemannian foliation for $\eta_P$, and we obtain the identities in (2) for the lifted foliations and for the forms $\theta$ and $\omega$.
\end{proof}

\begin{lemma}\label{lem:leaf closures-submersion}
Let $f:(X,\F_X)\rightarrow(Y,\F_Y)$ be a smooth foliated surjective submersion between Hausdorff manifolds satisfying
\begin{equation}\label{eq:leaf-submersion-tangent-preimage}
T\F_X=(df)^{-1}(T\F_Y).
\end{equation}
Let $\ell_X$ be a leaf of $\F_X$, and let $\ell_Y$ be the leaf of $\F_Y$ containing $f(\ell_X)$. If $\overline{\ell_X}$ and $\overline{\ell_Y}$ are embedded submanifolds and $f(\overline{\ell_X})$ is closed in $Y$, then $f|_{\overline{\ell_X}}:\overline{\ell_X}\rightarrow\overline{\ell_Y}$ is a surjective submersion and, for every $x\in\overline{\ell_X}$,
\begin{equation}\label{eq:leaf closure-tangent-preimage}
(df)_x^{-1}(T_{f(x)}\overline{\ell_Y})=T_x\overline{\ell_X}.
\end{equation}
\end{lemma}

\begin{proof}
Since leaves are weakly embedded by \cite{Lee:ISM}*{Theorem 19.17}, the restrictions of $f$ to leaves are smooth as maps into the corresponding leaves. Since $f$ is a submersion, for every $x\in X$ we have
\begin{equation}\label{eq:leaf-submersion-tangent-image}
(df)_x(T_x\F_X)=(df)_x\bigl((df)_x^{-1}(T_{f(x)}\F_Y)\bigr)=T_{f(x)}\F_Y.
\end{equation}
Thus the restrictions are submersions, and their images are open in the leaf topology.

In a foliation chart $U$, the intersection $\ell_X\cap U$ is a union of plaques \cite{Lee:ISM}*{p. 501}. Taking the closure in $U$, we find that $\overline{\ell_X}\cap U$ is also a union of plaques. Hence $\overline{\ell_X}\cap\ell$ is open in $\ell$ for every leaf $\ell$ of $\F_X$. This intersection is also closed in $\ell$, so connectedness of $\ell$ implies that it is either empty or all of $\ell$. Thus $\overline{\ell_X}$ is saturated. By continuity, $f(\overline{\ell_X})\subset\overline{\ell_Y}$. Since $f$ is foliated and $\overline{\ell_X}$ is saturated, $f(\overline{\ell_X})\cap\ell_Y$ is a union of images of leaves of $\F_X$ mapping into $\ell_Y$, and is therefore open in $\ell_Y$ by the preceding paragraph. It is also closed in $\ell_Y$, since $f(\overline{\ell_X})$ is closed in $Y$ and the inclusion of $\ell_Y$ is continuous. The intersection contains $f(\ell_X)$, so it is nonempty. Since $\ell_Y$ is connected, $\ell_Y\subset f(\overline{\ell_X})$. Taking closures in $Y$ and using the closedness of $f(\overline{\ell_X})$, we obtain $\overline{\ell_Y}\subset f(\overline{\ell_X})$. Hence $f(\overline{\ell_X})=\overline{\ell_Y}$.

The restriction $f|_{\overline{\ell_X}}:\overline{\ell_X}\rightarrow\overline{\ell_Y}$ is smooth by \cite{Lee:ISM}*{Corollary 5.30}. Along $\overline{\ell_X}$ we have $\ker df\subset T\F_X\subset T\overline{\ell_X}$, so, for $x\in\overline{\ell_X}$,
\begin{equation}
\ker d(f|_{\overline{\ell_X}})_x=T_x\overline{\ell_X}\cap\ker(df)_x=\ker(df)_x.
\end{equation}
Since $f$ is a submersion, $\dim\ker(df)_x$ is constant. Hence $f|_{\overline{\ell_X}}$ has constant rank. The restriction is surjective, so it is a submersion by \cite{Lee:ISM}*{Theorem 4.14(a)}. Since $d(f|_{\overline{\ell_X}})_x$ is surjective and $\ker(df)_x\subset T_x\overline{\ell_X}$,
\begin{equation}
(df)_x^{-1}(T_{f(x)}\overline{\ell_Y})=T_x\overline{\ell_X}+\ker(df)_x=T_x\overline{\ell_X}.
\end{equation}
\end{proof}

We continue in the setting of Corollary \ref{cor:theta-omega-morita}.

\begin{proposition}\label{prop:regular-morita-invariance}
Let $\kappa:OF(M,\mathcal O)\rightarrow B$ be a smooth surjective submersion onto a Hausdorff manifold whose fibres are the leaf closures of $\widetilde{\mathcal O}$. Suppose that one of the following conditions holds.
\begin{enumerate}[label=\textup{(\alph*)}]
\item Both $\alpha$ and $\alpha'$ have connected fibres.

\item There are smooth surjective submersions onto Hausdorff manifolds
\begin{equation}
\kappa':OF(M',\mathcal O')\longrightarrow B',
\qquad
\kappa_P:OF(P,\mathcal O_P)\longrightarrow B_P,
\end{equation}
whose fibres are the leaf closures of $\widetilde{\mathcal O}'$ and $\widetilde{\mathcal O}_P$, respectively. Every fibre of $\kappa_P$ is compact.

\item The submersions $\kappa'$ and $\kappa_P$ in \textup{(b)} exist, with fibres the leaf closures of $\widetilde{\mathcal O}'$ and $\widetilde{\mathcal O}_P$, respectively. Both $\alpha$ and $\alpha'$ are proper.
\end{enumerate}
In case \textup{(a)}, there is a unique smooth surjective submersion $\kappa':OF(M',\mathcal O')\rightarrow B$ satisfying
\begin{equation}\label{eq:morita-case-a-quotient}
\kappa'\circ\widehat\alpha'=\kappa\circ\widehat\alpha.
\end{equation}
Let $B_P=B'=B$ and
\begin{equation}\label{eq:morita-case-a-kappa-P}
\kappa_P\colonequals\kappa\circ\widehat\alpha.
\end{equation}
The fibres of $\kappa_P$ and $\kappa'$ are the leaf closures of $\widetilde{\mathcal O}_P$ and $\widetilde{\mathcal O}'$, respectively.

In all three cases, the following statements hold.
\begin{enumerate}[label=\textup{(\arabic*)}]
\item The right $O(q)$-actions on the frame bundles descend to smooth right actions on $B$, $B'$, and $B_P$. There are unique $O(q)$-equivariant surjective local diffeomorphisms $\rho:B_P\rightarrow B$ and $\rho':B_P\rightarrow B'$ satisfying
\begin{equation}\label{eq:morita-base-comparison}
\rho\circ\kappa_P=\kappa\circ\widehat\alpha,
\qquad
\rho'\circ\kappa_P=\kappa'\circ\widehat\alpha'.
\end{equation}
Moreover,
\begin{equation}\label{eq:morita-quotient-kernels}
\ker d\kappa_P=(d\widehat\alpha)^{-1}(\ker d\kappa)=(d\widehat\alpha')^{-1}(\ker d\kappa').
\end{equation}

\item For $b_P\in B_P$, let $b\colonequals\rho(b_P)$ and $b'\colonequals\rho'(b_P)$, and write
\begin{equation}\label{eq:morita-quotient-fibres}
L_P\colonequals\kappa_P^{-1}(b_P),\qquad L_b\colonequals\kappa^{-1}(b),\qquad L'_{b'}\colonequals(\kappa')^{-1}(b').
\end{equation}
The restrictions
\begin{equation}\label{eq:morita-fibre-submersions}
\widehat\alpha|_{L_P}:L_P\longrightarrow L_b,
\qquad
\widehat\alpha'|_{L_P}:L_P\longrightarrow L'_{b'}
\end{equation}
are surjective submersions. The foliation $\widetilde{\mathcal O}_P|_{L_P}$ is the pullback of $\widetilde{\mathcal O}|_{L_b}$ under $\widehat\alpha|_{L_P}$ and of $\widetilde{\mathcal O}'|_{L'_{b'}}$ under $\widehat\alpha'|_{L_P}$.

\item In case \textup{(a)}, $\rho$ and $\rho'$ are diffeomorphisms. In any of the three cases, if $\alpha$ is proper, then $\rho$ is a proper covering map with finite fibres. The same holds for $\alpha'$ and $\rho'$.
\end{enumerate}
\end{proposition}

\begin{proof}
By Corollary \ref{cor:theta-omega-morita}, $\widehat\alpha$ and $\widehat\alpha'$ are foliated $O(q)$-equivariant surjective submersions satisfying
\begin{equation}
T\widetilde{\mathcal O}_P=(d\widehat\alpha)^{-1}(T\widetilde{\mathcal O})=(d\widehat\alpha')^{-1}(T\widetilde{\mathcal O}').
\end{equation}
By the argument leading to \eqref{eq:leaf-submersion-tangent-image}, the restrictions of $\widehat\alpha$ and $\widehat\alpha'$ to leaves are smooth submersions. Their images are therefore open in the corresponding leaves.

\emph{Case \textup{(a)}.} The fibres of $\widehat\alpha$ and $\widehat\alpha'$ are connected by Lemma \ref{lem:frame-bundle-pullback}. So they are the leaves of the foliations with tangent bundles $\ker d\widehat\alpha$ and $\ker d\widehat\alpha'$, respectively, by \cite{MoerdijkMrcun:IFLG}*{Examples 1.1(2)}. Since both kernels are contained in $T\widetilde{\mathcal O}_P$, every fibre of either map is contained in a leaf of $\widetilde{\mathcal O}_P$, by \cite{MoerdijkMrcun:IFLG}*{Sec. 1.2, Remark (4)}.

Let $\ell$ be a leaf of $\widetilde{\mathcal O}$. The set $\widehat\alpha^{-1}(\ell)$ is a union of leaves of $\widetilde{\mathcal O}_P$. Since every fibre of $\widehat\alpha$ is contained in a leaf of $\widetilde{\mathcal O}_P$, the leaves contained in $\widehat\alpha^{-1}(\ell)$ have pairwise disjoint images. These images are open in $\ell$, and surjectivity of $\widehat\alpha$ implies that they cover $\ell$. Connectedness of $\ell$ therefore implies that $\widehat\alpha^{-1}(\ell)$ is a single leaf. The same argument applies to $\widehat\alpha'$.

Since $\widehat\alpha$ is continuous and open,
\begin{equation}\label{eq:morita-leaf closure-preimage}
\overline{\widehat\alpha^{-1}(\ell)}=\widehat\alpha^{-1}(\overline\ell).
\end{equation}
Thus $\kappa_P\colonequals\kappa\circ\widehat\alpha$ is a smooth surjective submersion whose fibres are the closures of the leaves of $\widetilde{\mathcal O}_P$.

Each fibre of $\widehat\alpha'$ is contained in a leaf of $\widetilde{\mathcal O}_P$, so $\kappa_P$ is constant on that fibre. By \cite{Lee:ISM}*{Theorem 4.30}, it descends uniquely to a smooth map
\begin{equation}
\kappa':OF(M',\mathcal O')\longrightarrow B,\qquad \kappa'\circ\widehat\alpha'=\kappa_P.
\end{equation}
Since $\kappa_P$ is surjective, so is $\kappa'$, and $d\kappa'\circ d\widehat\alpha'=d\kappa_P$ implies that $\kappa'$ is a submersion.

For a leaf $\ell'$ of $\widetilde{\mathcal O}'$, the set $(\widehat\alpha')^{-1}(\ell')$ is a leaf whose closure is $\kappa_P^{-1}(b)$ for some $b\in B$. Since $\widehat\alpha'$ is continuous and open,
\begin{equation}
\begin{aligned}
(\widehat\alpha')^{-1}(\overline{\ell'}) &=\overline{(\widehat\alpha')^{-1}(\ell')} \\
&=\kappa_P^{-1}(b) \\
&=(\widehat\alpha')^{-1}\bigl((\kappa')^{-1}(b)\bigr).
\end{aligned}
\end{equation}
Surjectivity of $\widehat\alpha'$ implies $\overline{\ell'}=(\kappa')^{-1}(b)$. Thus the fibres of $\kappa'$ are the closures of the leaves of $\widetilde{\mathcal O}'$. Let $B_P=B'=B$.

In all three cases, the frame maps are continuous and foliated, so they send leaf closures into leaf closures. Hence $\kappa\circ\widehat\alpha$ and $\kappa'\circ\widehat\alpha'$ are constant on the fibres of $\kappa_P$. By \cite{Lee:ISM}*{Theorem 4.30}, they descend uniquely to smooth maps $\rho:B_P\rightarrow B$ and $\rho':B_P\rightarrow B'$ satisfying
\begin{equation}
\rho\circ\kappa_P=\kappa\circ\widehat\alpha,\qquad \rho'\circ\kappa_P=\kappa'\circ\widehat\alpha'.
\end{equation}
Since $\kappa\circ\widehat\alpha$ and $\kappa'\circ\widehat\alpha'$ are surjective submersions, $\rho$ and $\rho'$ are surjective submersions.

In case \textup{(a)}, $\kappa\circ\widehat\alpha=\kappa_P=\kappa'\circ\widehat\alpha'$ by \eqref{eq:morita-case-a-kappa-P} and \eqref{eq:morita-case-a-quotient}. Thus $\rho\circ\kappa_P=\kappa_P=\rho'\circ\kappa_P$. Since $\kappa_P$ is surjective, $\rho=\rho'=\id_B$.

Fix $b_P\in B_P$, let $b=\rho(b_P)$ and $b'=\rho'(b_P)$, and write $L_P=\kappa_P^{-1}(b_P)$, $L_b=\kappa^{-1}(b)$, and $L'_{b'}=(\kappa')^{-1}(b')$. By \eqref{eq:morita-base-comparison}, we have $\widehat\alpha(L_P)\subset L_b$ and $\widehat\alpha'(L_P)\subset L'_{b'}$.

In case \textup{(a)},
\begin{equation}\label{eq:morita-case-a-fibre-preimages}
L_P=\widehat\alpha^{-1}(L_b)=(\widehat\alpha')^{-1}(L'_{b'}),
\end{equation}
so both inclusions are equalities by surjectivity of the frame maps. In particular, both images are closed. In case \textup{(b)}, both images are compact and hence closed in their respective frame bundles. In case \textup{(c)}, the frame maps are proper by Lemma \ref{lem:frame-bundle-pullback}. Since $L_P$ is closed, both images are closed.

As fibres of smooth submersions, $L_P,L_b,L'_{b'}$ are embedded submanifolds. Choose a leaf $\ell_P$ of $\widetilde{\mathcal O}_P$ with $\overline{\ell_P}=L_P$. Since the fibres of $\kappa$ and $\kappa'$ are the leaf closures, the leaves containing $\widehat\alpha(\ell_P)$ and $\widehat\alpha'(\ell_P)$ have closures $L_b$ and $L'_{b'}$, respectively. By Lemma \ref{lem:leaf closures-submersion}, the restrictions $\widehat\alpha|_{L_P}:L_P\rightarrow L_b$ and $\widehat\alpha'|_{L_P}:L_P\rightarrow L'_{b'}$ are surjective submersions. Since the tangent spaces of the fibres are the kernels of the corresponding differentials and $b_P$ is arbitrary, we also obtain from \eqref{eq:leaf closure-tangent-preimage}
\begin{equation}
\ker d\kappa_P=(d\widehat\alpha)^{-1}(\ker d\kappa)=(d\widehat\alpha')^{-1}(\ker d\kappa').
\end{equation}

By Corollary \ref{cor:theta-omega-morita}\textup{(2)}, the restricted foliations satisfy
\begin{equation}\label{eq:morita-restricted-foliations}
\begin{aligned}
\widetilde{\mathcal O}_P|_{L_P} &=(\widehat\alpha|_{L_P})^*(\widetilde{\mathcal O}|_{L_b}) \\
&=(\widehat\alpha'|_{L_P})^*(\widetilde{\mathcal O}'|_{L'_{b'}}).
\end{aligned}
\end{equation}

Differentiating $\rho\circ\kappa_P=\kappa\circ\widehat\alpha$ and using $(d\widehat\alpha)^{-1}(\ker d\kappa)=\ker d\kappa_P$, we obtain
\begin{equation}
\ker(d\rho\circ d\kappa_P)=\ker(d\kappa\circ d\widehat\alpha)=\ker d\kappa_P.
\end{equation}
Since $d\kappa_P$ is surjective, $d\rho$ is injective. Thus $\rho$ is both an immersion and a submersion, and hence a local diffeomorphism. The same argument applies to $\rho'$.

By Proposition \ref{leafwise-Oq}\textup{(i)}, the right $O(q)$-actions send leaves to leaves, and therefore send leaf closures to leaf closures. They descend to right actions on $B$, $B'$, and $B_P$. These actions are smooth because $\kappa$, $\kappa'$, and $\kappa_P$ admit smooth local sections. For $u\in OF(P,\mathcal O_P)$ and $A\in O(q)$,
\begin{equation}
\rho(\kappa_P(u)\cdot A)=\kappa(\widehat\alpha(u)\cdot A)=\rho(\kappa_P(u))\cdot A.
\end{equation}
Thus $\rho$ is equivariant. The same calculation applies to $\rho'$.

Finally, suppose that $\alpha$ is proper. Then $\widehat\alpha$ is proper by Lemma \ref{lem:frame-bundle-pullback}. Let $C\subset B$ be compact. Choose finitely many smooth local sections $s_i:U_i\longrightarrow OF(M,\mathcal O)$ of $\kappa$ and compact sets $C_i\subset C\cap U_i$ such that $C=\bigcup_i C_i$. Then $K\colonequals\bigcup_i s_i(C_i)$ is compact and $\kappa(K)=C$. Since $\kappa(K)=C$ and $\widehat\alpha$ maps each fibre of $\kappa_P$ onto the corresponding fibre of $\kappa$, we have
\begin{equation}
\rho^{-1}(C)=\kappa_P\bigl(\widehat\alpha^{-1}(K)\bigr).
\end{equation}
Since $\widehat\alpha$ is proper, $\widehat\alpha^{-1}(K)$ is compact. Its image under $\kappa_P$ is compact, so $\rho$ is proper.

The map $\rho$ is a proper surjective local diffeomorphism, $B_P$ is Hausdorff, and $B$ is locally compact. Hence $\rho$ is a covering map by \cite{HenriquesMetzler:NoneffectiveOrbifolds}*{Lemma 2.5}. Its fibres are compact and discrete, hence finite. If $\alpha'$ is proper, then by the same argument, $\rho'$ is a proper covering map with finite fibres.
\end{proof}

\begin{remark}
For complete Riemannian foliations, Lin, Loizides, Sjamaar, and Song obtain an $O(q)$-equivariant isometry between the Molino manifolds from a complete isometric weak equivalence \cite{LinLoizidesSjamaarSong:RiemannianFoliationsQuantization}*{Sec. 2.7, Proposition 2.16}.
\end{remark}

\begin{corollary}\label{cor:regular-morita-compact-bases}
Assume that $M$ is connected.
\begin{enumerate}[label=\textup{(\arabic*)}]
\item If $M$ is compact and both $\alpha$ and $\alpha'$ have connected fibres, then the leaf closures of $\widetilde{\mathcal O}$, $\widetilde{\mathcal O}'$, and $\widetilde{\mathcal O}_P$ are the fibres of smooth surjective submersions $\kappa$, $\kappa'$, $\kappa_P$ onto Hausdorff manifolds, and the bases are related by $O(q)$-equivariant diffeomorphisms $B\xleftarrow{\ \cong\ }B_P\xrightarrow{\ \cong\ }B'$.
\item If $P$ is compact, then $M$ and $M'$ are compact, and the leaf closures of the three lifted foliations are the fibres of $O(q)$-equivariant smooth surjective submersions $\kappa$, $\kappa'$, $\kappa_P$ onto Hausdorff manifolds. Each submersion restricts to a locally trivial fibre bundle over each connected component of its base. The maps $B\xleftarrow{\ \rho\ }B_P\xrightarrow{\ \rho'\ }B'$ are proper covering maps with finite fibres.
\end{enumerate}
\end{corollary}

\begin{proof}
\textup{(1)} By Theorem \ref{thm:mol}, there is an $O(q)$-equivariant fibre bundle $\kappa:OF(M,\mathcal O)\rightarrow B$ whose fibres are the leaf closures of $\widetilde{\mathcal O}$. The conclusion then follows from case \textup{(a)} of Proposition \ref{prop:regular-morita-invariance}.

\textup{(2)} Since $P$ is compact and $\alpha$ and $\alpha'$ are surjective, $M$ and $M'$ are compact. Applying Theorem \ref{thm:mol} to $(M,\mathcal O)$, we choose an $O(q)$-equivariant fibre bundle $\kappa:OF(M,\mathcal O)\rightarrow B$ whose fibres are the leaf closures of $\widetilde{\mathcal O}$.

Let $P_j$ be a connected component of $P$; there are finitely many. The image $\alpha(P_j)$ is open, because $\alpha$ is a submersion, and closed, because $P_j$ is compact. Since $M$ is connected, $\alpha(P_j)=M$. Similarly, $\alpha'(P_j)$ is open, closed, and connected, hence a connected component $M'_i$ of $M'$. These images cover $M'$ because $\alpha'$ is surjective.

The restrictions of $z$ and $z'$ to $\mathcal K_P|_{P_j}$ are Morita fibrations onto $\cG$ and $\cG'|_{M'_i}$, since full faithfulness restricts and the restricted object maps are surjective submersions. The restricted object maps remain Riemannian submersions.

Applying Theorem \ref{thm:mol} to each connected component of $P$ and $M'$, we choose $O(q)$-equivariant fibre bundles $\kappa_{P_j}:OF(P_j,\mathcal O_P)\rightarrow B_{P_j}$ and $\kappa'_i:OF(M'_i,\mathcal O')\rightarrow B'_i$ whose fibres are the compact leaf closures. To compare the dimensions of these bases, we apply Proposition \ref{prop:regular-morita-invariance}\textup{(1)} in case \textup{(b)} to the restricted Morita fibrations. For each $j$, with $M'_i=\alpha'(P_j)$, we obtain surjective local diffeomorphisms $B\leftarrow B_{P_j}\rightarrow B'_i$. Hence $\dim B_{P_j}=\dim B'_i=\dim B$.

Thus $B_P\colonequals\bigsqcup_jB_{P_j}$ and $B'\colonequals\bigsqcup_iB'_i$ are Hausdorff manifolds of dimension $\dim B$. The maps $\kappa_P\colonequals\bigsqcup_j\kappa_{P_j}$ and $\kappa'\colonequals\bigsqcup_i\kappa'_i$ are $O(q)$-equivariant surjective submersions whose fibres are the leaf closures. Their restrictions over each connected component of the base are locally trivial fibre bundles. Since $P$ and $O(q)$ are compact, $OF(P,\mathcal O_P)$, and hence $B_P$, is compact.

We can now apply case \textup{(b)} of Proposition \ref{prop:regular-morita-invariance} to the Morita fibrations in diagram \eqref{eq:morita-fibrations-diagram}, obtaining $\rho$ and $\rho'$. Since $P$ is compact, $\alpha$ and $\alpha'$ are proper. By Proposition \ref{prop:regular-morita-invariance}\textup{(3)}, $\rho$ and $\rho'$ are proper covering maps with finite fibres.
\end{proof}


\subsection{Basic Lie algebroids and structural Lie algebras}

In each case of Proposition \ref{prop:regular-morita-invariance}, we have smooth surjective submersions $\kappa,\kappa',\kappa_P$ whose fibres are the leaf closures of the three lifted foliations. We first extend the basic Lie algebroid and fibrewise Maurer--Cartan constructions to transversely parallelizable foliations with such submersions.

\begin{lemma}\label{lem:basic-algebroid-closure-submersion}
Let $\F$ be a transversely parallelizable foliation of codimension $r$ on a Hausdorff manifold $X$, and let $\kappa:X\rightarrow B$ be a smooth surjective submersion onto a Hausdorff manifold whose fibre through every point is the closure of the leaf of $\F$ through that point. Then:
\begin{enumerate}[label=\textup{(\roman*)}]
\item There is a transitive Lie algebroid $A\colonequals b(X,\F)\rightarrow B$ with
\begin{equation}\label{eq:basic-closure-sections}
\Gamma(A|_U)\cong l\bigl(\kappa^{-1}(U),\F|_{\kappa^{-1}(U)}\bigr)
\end{equation}
as $C^\infty(U)$-modules for every open $U\subset B$, compatibly with restriction. Evaluation defines a smooth vector bundle isomorphism
\begin{equation}\label{eq:basic-closure-evaluation}
\operatorname{ev}^X:\kappa^*A\xlongrightarrow{\cong}N(\F),
\end{equation}
under which the pullback of the anchor $\rho_A$ corresponds to the map induced by $d\kappa$ on $N(\F)$. The Lie algebroid $A$ is independent of the chosen transverse parallelism up to an isomorphism preserving the identifications of local sections and the evaluation maps.

\item For $b\in B$, put $L_b\colonequals \kappa^{-1}(b)$. Restriction to $L_b$ defines a Lie algebra isomorphism
\begin{equation}
\operatorname{res}_b:(\ker\rho_A)_b\xlongrightarrow{\cong}l(L_b,\F|_{L_b}).
\end{equation}
The foliation $\F|_{L_b}$ is a Lie foliation with canonical Maurer--Cartan form $\omega^b_{\mathrm{MC}}$.

\item There is a unique smooth section $\omega^\kappa_{\mathrm{MC}}\in \Gamma\bigl((\ker d\kappa)^*\otimes\kappa^*(\ker\rho_A)\bigr)$ such that
\begin{equation}\label{eq:basic-closure-MC-res}
\operatorname{res}_b\circ\bigl(\omega^\kappa_{\mathrm{MC}}|_{L_b}\bigr)=\omega^b_{\mathrm{MC}}\qquad (b\in B).
\end{equation}
It is given by
\begin{equation}\label{eq:basic-closure-MC-formula}
(\omega^\kappa_{\mathrm{MC}})_x(\xi)=(\operatorname{ev}^X_x)^{-1}\bigl(\pr^N_x(\xi)\bigr),\qquad \xi\in\ker(d\kappa)_x,
\end{equation}
where $\pr^N:TX\rightarrow N(\F)$ is the quotient projection.
\end{enumerate}
\end{lemma}

\begin{proof}
(i) A smooth $\F$-basic function is constant on every leaf and, by continuity, on every leaf closure. It is therefore constant on every fibre of $\kappa$ and descends uniquely to a smooth function on $B$ by \cite{Lee:ISM}*{Theorem 4.30}. Conversely, since $T\F\subset\ker d\kappa$, every pullback of a smooth function on $B$ is basic. The same argument applies over every open $U\subset B$.

Choose a transverse parallelism $\bar Y_1,\dots,\bar Y_r$ with projectable representatives $Y_1,\dots,Y_r$. By the coefficient argument in Step 2 of the proof of Proposition \ref{prop:basic-Lie-algebroid}, the transverse vector fields on $\kappa^{-1}(U)$ are the combinations $\sum_i(f_i\circ\kappa)\bar Y_i$, with uniquely determined $f_i\in C^\infty(U)$. Thus $A\colonequals B\times\RR^r$ has the identifications of local sections in \textup{(i)}, compatibly with restriction.

Evaluation at $x\in\kappa^{-1}(b)$ sends the standard basis of $A_b$ to $(\bar Y_1)_x,\dots,(\bar Y_r)_x$, a basis of $N_x(\F)$. Hence $\operatorname{ev}^X:\kappa^*A\rightarrow N(\F)$ is a smooth vector bundle isomorphism.

For the anchor, let $Y$ be projectable on $\kappa^{-1}(U)$ and $f\in C^\infty(U)$. For $T\in\mathfrak X(\F|_{\kappa^{-1}(U)})$,
\begin{equation}
T\bigl(Y(f\circ\kappa)\bigr)=[T,Y](f\circ\kappa)+Y\bigl(T(f\circ\kappa)\bigr)=0,
\end{equation}
so $Y(f\circ\kappa)$ is basic and descends to $U$. Applying this to local coordinate functions gives a unique smooth vector field $\varrho(\bar Y)$ on $U$ with $d\kappa(Y)=\varrho(\bar Y)\circ\kappa$. It depends only on $\bar Y$, since fields tangent to $\F$ lie in $\ker d\kappa$. Moreover, $\varrho((f\circ\kappa)\bar Y)=f\varrho(\bar Y)$.

The bracket of two projectable fields projects to the bracket of their projections, so $\varrho$ also preserves brackets. Applying the arguments in Steps 5 and 7 of the proof of Proposition \ref{prop:basic-Lie-algebroid}, with $\kappa$ in place of $\pi_{\bas}$, we obtain the Lie algebroid structure on $A$ and its independence of the chosen transverse parallelism. The isomorphisms obtained from a change of transverse parallelism preserve the identifications of local sections with transverse vector fields, and hence the evaluation maps.

The anchor $\rho_A:A\rightarrow TB$ satisfies
\begin{equation}\label{eq:basic-closure-anchor}
\rho_A(a)=\overline{d\kappa}_x(\operatorname{ev}^X_x(a))
\end{equation}
for $x\in\kappa^{-1}(b)$ and $a\in A_b$, where $\overline{d\kappa}_x:N_x(\F)\rightarrow T_bB$ is the map induced by $d\kappa_x$. Since $\kappa$ is a submersion and $\operatorname{ev}^X_x$ is an isomorphism, $\rho_A$ is surjective, so $A$ is transitive.

(ii) Since $\rho_A$ is surjective, $\ker\rho_A$ is a smooth subbundle. For $a\in(\ker\rho_A)_b$, choose a local section of $\ker\rho_A$ extending $a$, with projectable representative $Y$. Then $d\kappa(Y)=0$, so $Y$ is tangent to the fibres of $\kappa$. By \textup{(i)}, the representatives of any two local extensions of $a$ differ by a field tangent to $\F$ along $L_b$. By the restriction argument in Step 2 of the proof of Lemma \ref{lem:isotropy-structural-clean}, we obtain a well-defined Lie algebra homomorphism
\begin{equation}\label{eq:basic-closure-restriction}
\begin{aligned}
\operatorname{res}_b\colon (\ker\rho_A)_b&\longrightarrow l(L_b,\F|_{L_b}), \\
a &\longmapsto \overline{Y|_{L_b}}.
\end{aligned}
\end{equation}

By \textup{(i)}, for every $x\in L_b$, evaluation restricts to an isomorphism
\begin{equation}\label{eq:basic-closure-isotropy-evaluation}
\operatorname{ev}^X_x:(\ker\rho_A)_b\xlongrightarrow{\cong}T_xL_b/T_x\F=N_x(\F|_{L_b}),
\end{equation}
and, by construction,
\begin{equation}\label{eq:basic-closure-ev-res}
\operatorname{ev}^{L_b}_x\circ\operatorname{res}_b=\operatorname{ev}^X_x|_{(\ker\rho_A)_b}.
\end{equation}
Thus the restrictions of a basis of $(\ker\rho_A)_b$ form a transverse parallelism on $(L_b,\F|_{L_b})$.

Every leaf of $\F|_{L_b}$ is dense in $L_b$ by the hypothesis on $\kappa$. The coefficients of any transverse vector field in the restricted parallelism are basic by the coefficient argument in \textup{(i)}, and hence constant on $L_b$. Therefore $\operatorname{res}_b$ is surjective. It is injective by \eqref{eq:basic-closure-ev-res}. Thus $\operatorname{ev}^{L_b}_x$ is an isomorphism for every $x\in L_b$.

Define
\begin{equation}\label{eq:closure-fibre-MC-definition}
(\omega^b_{\mathrm{MC}})_x(\xi)\colonequals(\operatorname{ev}^{L_b}_x)^{-1}\bigl(\pr^N_x(\xi)\bigr),\qquad \xi\in T_xL_b.
\end{equation}
This is a smooth $l(L_b,\F|_{L_b})$-valued $1$-form on $L_b$, pointwise surjective, with kernel $T(\F|_{L_b})$. For every projectable vector field $Y$ on $L_b$, $\omega^b_{\mathrm{MC}}(Y)=\bar Y$ is constant. Since projectable vector fields span $TL_b$ pointwise, the Maurer--Cartan calculation in the proof of \cite{MoerdijkMrcun:IFLG}*{Theorem 4.24} applies. Hence $\F|_{L_b}$ is a Lie foliation with canonical Maurer--Cartan form $\omega^b_{\mathrm{MC}}$.

(iii) By \textup{(i)}, $\operatorname{ev}^X$ restricts to a smooth vector bundle isomorphism $\kappa^*(\ker\rho_A)\cong\ker d\kappa/T\F$. The formula \eqref{eq:basic-closure-MC-formula} therefore defines a smooth section of $(\ker d\kappa)^*\otimes\kappa^*(\ker\rho_A)$.

Applying $\operatorname{res}_b$ to \eqref{eq:basic-closure-MC-formula} and using \eqref{eq:basic-closure-ev-res}, we obtain \eqref{eq:basic-closure-MC-res}. Since $\operatorname{res}_b$ is injective, \eqref{eq:basic-closure-MC-res} determines $\omega^\kappa_{\mathrm{MC}}$ uniquely.

In the compact homogeneous case, the basic fibre bundle $\pi_{\bas}:X\rightarrow W$ and $\kappa$ have the same fibres by \cite{MoerdijkMrcun:IFLG}*{Corollary 4.25}, so we identify $B$ with $W$ by Lemma \ref{lem:same-fibres-diffeo}. Under this identification, the modules of sections, bracket, anchor, and evaluation maps are those in Proposition \ref{prop:basic-Lie-algebroid} and Lemma \ref{lem:isotropy-structural-clean}. For each $b\in B$, the canonical form $\omega^b_{\mathrm{MC}}$ on $L_b$ is given by \eqref{eq:canonical-MC-form}. The global vertical section $\omega^\kappa_{\mathrm{MC}}$ is given by \eqref{eq:vertical-MC-definition}.
\end{proof}

\begin{lemma}\label{lem:morita-basic-algebroid-comparison}
Assume the hypotheses of Proposition \ref{prop:regular-morita-invariance}, in any of its three cases. Then the lifted foliations have basic Lie algebroids
\begin{equation}
\begin{gathered}
A\colonequals b\bigl(OF(M,\mathcal O),\widetilde{\mathcal O}\bigr)\longrightarrow B,\qquad
A'\colonequals b\bigl(OF(M',\mathcal O'),\widetilde{\mathcal O}'\bigr)\longrightarrow B',\\
A_P\colonequals b\bigl(OF(P,\mathcal O_P),\widetilde{\mathcal O}_P\bigr)\longrightarrow B_P,
\end{gathered}
\end{equation}
and $(\widehat\alpha,\rho)$ and $(\widehat\alpha',\rho')$ induce Lie algebroid isomorphisms over $B_P$,
\begin{equation}
\Psi:\rho^*A\longrightarrow A_P,\qquad
\Psi':(\rho')^*A'\longrightarrow A_P.
\end{equation}
The pullback Lie algebroid structures are transported locally through the local diffeomorphisms $\rho$ and $\rho'$. The isomorphism $\Psi$ is characterized by
\begin{equation}\label{eq:morita-basic-evaluation}
\overline{d\widehat\alpha}_u
\circ\operatorname{ev}^P_u
\circ\Psi_{b_P}
=
\operatorname{ev}_{\widehat\alpha(u)}
\end{equation}
for every $b_P\in B_P$ and $u\in\kappa_P^{-1}(b_P)$, where the evaluation isomorphisms are those of Lemma \ref{lem:basic-algebroid-closure-submersion}\textup{(i)} and $\overline{d\widehat\alpha}_u$ is the isomorphism induced by $d\widehat\alpha_u$ on the normal spaces of the lifted foliations. The corresponding identity characterizes $\Psi'$.
\end{lemma}

\begin{proof} 
The lifted foliations are transversely parallelizable by Lemma \ref{lem:theta-omega-parallelism}. By Proposition \ref{prop:regular-morita-invariance}, we have the three leaf closure submersions $\kappa,\kappa',\kappa_P$. Hence, by Lemma \ref{lem:basic-algebroid-closure-submersion}, we obtain $A$, $A'$, and $A_P$.

We prove the statement for $(\widehat\alpha,\rho)$. By Proposition \ref{prop:regular-morita-invariance}, $\rho:B_P\rightarrow B$ is a surjective local diffeomorphism and $\kappa\circ\widehat\alpha=\rho\circ\kappa_P$. Let $V\subset B_P$ be an open set such that $\rho|_V:V\rightarrow U\colonequals \rho(V)$ is a diffeomorphism. On $V$, every section of $\rho^*A$ is uniquely of the form $\rho^*\sigma$, with $\sigma\in\Gamma(A|_U)$, and the Lie algebroid structure on $\rho^*A|_V$ is transported from $A|_U$ through $\rho|_V$.

By Corollary \ref{cor:theta-omega-morita}\textup{(1)--(2)}, $\widehat\alpha$ is a submersion and $T\widetilde{\mathcal O}_P=(d\widehat\alpha)^{-1}(T\widetilde{\mathcal O})$. Hence the induced map $\overline{d\widehat\alpha}$ on normal bundles is smooth and fibrewise invertible. For $\sigma\in\Gamma(A|_U)$, consider the smooth section of $N(\widetilde{\mathcal O}_P)$ on $\kappa_P^{-1}(V)$ given by
\begin{equation}\label{eq:morita-transverse-section}
u\longmapsto
(\overline{d\widehat\alpha}_u)^{-1}\operatorname{ev}_{\widehat\alpha(u)}\bigl(\sigma(\kappa(\widehat\alpha(u)))\bigr).
\end{equation}

Let $\tilde s_i:OF(M,\mathcal O)|_{U_i}\rightarrow OF(T_i)$ be the lifted submersions. The maps $\tilde s_i$ and $\tilde s_i\circ\widehat\alpha$ locally define $\widetilde{\mathcal O}$ and $\widetilde{\mathcal O}_P$, respectively. After restricting to foliation charts, the transverse vector field corresponding to $\sigma$ projects along $\tilde s_i$ to a vector field on an open subset of $OF(T_i)$. Since $d(\tilde s_i\circ\widehat\alpha) =d\tilde s_i\circ d\widehat\alpha$, the section above projects along $\tilde s_i\circ\widehat\alpha$ to the same vector field. It is therefore a transverse vector field \cite{MoerdijkMrcun:IFLG}*{Sec. 4.1.2, Examples 4.4(1), (3)}.

By Lemma \ref{lem:basic-algebroid-closure-submersion}\textup{(i)}, it determines a unique section of $A_P|_V$, which we denote by $\Psi_V(\rho^*\sigma)$. By \eqref{eq:morita-transverse-section} and the identity $\kappa\circ\widehat\alpha=\rho\circ\kappa_P$, $\Psi_V$ is $C^\infty(V)$-linear and commutes with restriction. Hence it defines a smooth vector bundle map $\Psi_V:\rho^*A|_V\rightarrow A_P|_V$.

Fix $b_P\in V$ and choose $u\in\kappa_P^{-1}(b_P)$. By the definition of $\Psi_V$ and $(\rho^*A)_{b_P}=A_{\rho(b_P)}$, we have
\begin{equation}\label{eq:morita-local-evaluation}
\overline{d\widehat\alpha}_u
\circ\operatorname{ev}^P_u
\circ(\Psi_V)_{b_P}
=
\operatorname{ev}_{\widehat\alpha(u)}.
\end{equation}
Since the evaluation maps and $\overline{d\widehat\alpha}_u$ are isomorphisms, $(\Psi_V)_{b_P}$ is an isomorphism uniquely determined by this identity. Thus the maps $\Psi_V$ agree on overlaps and glue to a vector bundle isomorphism $\Psi:\rho^*A\rightarrow A_P$.

Differentiating $\kappa\circ\widehat\alpha=\rho\circ\kappa_P$ and using \eqref{eq:basic-closure-anchor} and \eqref{eq:morita-local-evaluation}, we obtain $d\rho\bigl(\rho_{A_P}(\Psi_V(\rho^*\sigma))\bigr)=\rho_A(\sigma)\circ\rho$. Since $\rho|_V$ is a diffeomorphism, $\Psi_V$ preserves the anchor.

By the construction of $\Psi_V$, corresponding transverse vector fields have the same local projection to $OF(T_i)$ along $\widetilde{s}_i$ and $\widetilde{s}_i\circ\widehat\alpha$, respectively. Since the Lie bracket of two projectable vector fields projects to the Lie bracket of their projections, $\Psi_V$ preserves Lie brackets.

Thus $\Psi$ is a Lie algebroid isomorphism. Applying the same construction to $(\widehat\alpha',\rho')$, we obtain $\Psi':(\rho')^*A'\rightarrow A_P$.
\end{proof}


\begin{proposition}[Compatibility with structural Lie algebras]
\label{prop:morita-isotropy-upgrade}
Under the hypotheses of Lemma \ref{lem:morita-basic-algebroid-comparison}, the isomorphisms $\Psi$ and $\Psi'$ restrict to Lie algebra bundle isomorphisms
\begin{equation}
\rho^*(\ker\rho_A)
\cong\ker\rho_{A_P}
\cong(\rho')^*(\ker\rho_{A'}).
\end{equation}
Let $\omega^\kappa_{\mathrm{MC}}$, $\omega^{\kappa'}_{\mathrm{MC}}$, and $\omega^{\kappa_P}_{\mathrm{MC}}$ be the forms along the fibres of $\kappa$, $\kappa'$, and $\kappa_P$, respectively. After pullback by $\kappa_P$, these isomorphisms identify the pullbacks\footnote{Here the pullbacks of the sections are taken on $\ker d\kappa_P$ using $d\widehat\alpha$ and $d\widehat\alpha'$.} of $\omega^\kappa_{\mathrm{MC}}$ and $\omega^{\kappa'}_{\mathrm{MC}}$ along $\widehat\alpha$ and $\widehat\alpha'$, respectively, with $\omega^{\kappa_P}_{\mathrm{MC}}$.

For each $b_P\in B_P$, the inverses of the isotropy restrictions of $\Psi_{b_P}$ and $\Psi'_{b_P}$ induce, under the restriction isomorphisms of Lemma \ref{lem:basic-algebroid-closure-submersion}\textup{(ii)}, the unique structural Lie algebra isomorphisms that identify the canonical Maurer--Cartan forms by pullback along the restrictions of $\widehat\alpha$ and $\widehat\alpha'$.
\end{proposition}

\begin{proof}
Since $\rho$ and $\rho'$ are local diffeomorphisms, the anchors of $\rho^*A$ and $(\rho')^*A'$ are obtained by transporting the anchors of $A$ and $A'$. Hence $\ker\rho_{\rho^*A}=\rho^*(\ker\rho_A)$ and $\ker\rho_{(\rho')^*A'}=(\rho')^*(\ker\rho_{A'})$. Since $\Psi$ and $\Psi'$ are Lie algebroid isomorphisms, they preserve brackets and anchors. Their restrictions are therefore Lie algebra bundle isomorphisms
\begin{equation}\label{eq:morita-isotropy-restrictions}
\begin{aligned}
\Psi|_{\rho^*(\ker\rho_A)}&:\rho^*(\ker\rho_A)\xlongrightarrow{\cong}\ker\rho_{A_P},\\
\Psi'|_{(\rho')^*(\ker\rho_{A'})}&:(\rho')^*(\ker\rho_{A'})\xlongrightarrow{\cong}\ker\rho_{A_P}.
\end{aligned}
\end{equation}

Let $b_P\in B_P$, $u\in\kappa_P^{-1}(b_P)$, and $\xi\in\ker(d\kappa_P)_u$. By the kernel identities in Proposition \ref{prop:regular-morita-invariance}\textup{(1)}, we have $d\widehat\alpha_u\xi\in\ker(d\kappa)_{\widehat\alpha(u)}$ and $d\widehat\alpha'_u\xi\in\ker(d\kappa')_{\widehat\alpha'(u)}$. We now compare the three Maurer--Cartan sections. By \eqref{eq:morita-basic-evaluation} and \eqref{eq:basic-closure-MC-formula}, we have
\begin{equation}
\begin{aligned}
&\overline{d\widehat\alpha}_u\operatorname{ev}^P_u
\Psi_{b_P}\Bigl(
(\omega^\kappa_{\mathrm{MC}})_{\widehat\alpha(u)}
(d\widehat\alpha_u\xi)
\Bigr)\\
&\qquad=
\operatorname{ev}_{\widehat\alpha(u)}\Bigl(
(\omega^\kappa_{\mathrm{MC}})_{\widehat\alpha(u)}
(d\widehat\alpha_u\xi)
\Bigr)\\
&\qquad=[d\widehat\alpha_u\xi].
\end{aligned}
\end{equation}
Applying \eqref{eq:basic-closure-MC-formula} to $\kappa_P$, we obtain
\begin{equation}
\overline{d\widehat\alpha}_u\operatorname{ev}^P_u
\bigl((\omega^{\kappa_P}_{\mathrm{MC}})_u(\xi)\bigr)
=\overline{d\widehat\alpha}_u[\xi]
=[d\widehat\alpha_u\xi].
\end{equation}

The same calculation applies to $\Psi'$ and $\widehat\alpha'$.

Since $\overline{d\widehat\alpha}_u\circ\operatorname{ev}^P_u$ and $\overline{d\widehat\alpha'}_u\circ\operatorname{ev}^P_u$ are injective, we obtain
\begin{equation}\label{eq:morita-MC-compatibility}
\begin{aligned}
(\omega^{\kappa_P}_{\mathrm{MC}})_u(\xi)
&=\Psi_{b_P}\Bigl((\omega^\kappa_{\mathrm{MC}})_{\widehat\alpha(u)}(d\widehat\alpha_u\xi)\Bigr)\\
&=\Psi'_{b_P}\Bigl((\omega^{\kappa'}_{\mathrm{MC}})_{\widehat\alpha'(u)}(d\widehat\alpha'_u\xi)\Bigr).
\end{aligned}
\end{equation}

Let $b\colonequals \rho(b_P)$ and $b'\colonequals \rho'(b_P)$, and let $L_b$, $L_P$, and $L'_{b'}$ be the fibres in Proposition \ref{prop:regular-morita-invariance}\textup{(2)}. Let
\begin{equation}\label{eq:morita-fibre-map-definitions}
r\colonequals \widehat\alpha|_{L_P}:L_P\longrightarrow L_b,\qquad
r'\colonequals \widehat\alpha'|_{L_P}:L_P\longrightarrow L'_{b'},
\end{equation}
and let $\omega^b_{\mathrm{MC}}$, $\omega^P_{\mathrm{MC}}$, and $\omega^{b'}_{\mathrm{MC}}$ be the canonical Maurer--Cartan forms of the restricted foliations on these fibres, respectively. Write $\operatorname{res}_b$, $\operatorname{res}_{b_P}$, and $\operatorname{res}_{b'}$ for the restriction isomorphisms of Lemma \ref{lem:basic-algebroid-closure-submersion}\textup{(ii)}. Define
\begin{equation}\label{eq:morita-structural-isomorphisms}
\begin{aligned}
\Phi_r&\colonequals \operatorname{res}_b \circ\bigl(\Psi_{b_P}|_{(\ker\rho_A)_b}\bigr)^{-1} \circ\operatorname{res}_{b_P}^{-1},\\
\Phi_{r'}&\colonequals \operatorname{res}_{b'}\circ\bigl(\Psi'_{b_P}|_{(\ker\rho_{A'})_{b'}}\bigr)^{-1}\circ\operatorname{res}_{b_P}^{-1}.
\end{aligned}
\end{equation}
All factors are Lie algebra isomorphisms, so these are Lie algebra isomorphisms
\begin{equation}
\begin{aligned}
\Phi_r&:l(L_P,\widetilde{\mathcal O}_P|_{L_P})\longrightarrow l(L_b,\widetilde{\mathcal O}|_{L_b}),\\
\Phi_{r'}&:l(L_P,\widetilde{\mathcal O}_P|_{L_P})\longrightarrow l(L'_{b'},\widetilde{\mathcal O}'|_{L'_{b'}}).
\end{aligned}
\end{equation}
Apply $\operatorname{res}_b\circ\Psi_{b_P}^{-1}$ to the first equality in \eqref{eq:morita-MC-compatibility} at points of $L_P$. Using \eqref{eq:basic-closure-MC-res} on $L_b$ and $L_P$, we obtain
\begin{equation}\label{eq:morita-structural-MC}
\begin{aligned}
r^*\omega^b_{\mathrm{MC}}
&=\operatorname{res}_b\circ\Psi_{b_P}^{-1}
\circ\operatorname{res}_{b_P}^{-1}\circ\omega^P_{\mathrm{MC}}\\
&=\Phi_r\circ\omega^P_{\mathrm{MC}}.
\end{aligned}
\end{equation}
By the same argument, we have 
\begin{equation}\label{eq:morita-structural-MC-prime}
(r')^*\omega^{b'}_{\mathrm{MC}} =\Phi_{r'}\circ\omega^P_{\mathrm{MC}}.
\end{equation}
 Since $\omega^P_{\mathrm{MC}}$ is pointwise surjective, these identities determine $\Phi_r$ and $\Phi_{r'}$ uniquely. Substituting their definitions gives, on $(\ker\rho_A)_b$,
\begin{equation}
\Phi_{r'}\circ\Phi_r^{-1}\circ\operatorname{res}_b=\operatorname{res}_{b'}\circ(\Psi'_{b_P})^{-1}\circ\Psi_{b_P}.
\end{equation}
\end{proof}


We would like to conclude the article by applying the preceding results to Riemannian Morita equivalence bibundles. Let $\cG'\curvearrowright P\curvearrowleft\cG$ be such a bibundle with $P$ Hausdorff. By Lemma \ref{lem:bibundle-pullback-groupoid}, there are Morita fibrations $\cG'\xleftarrow{z'}\mathcal K_P\xrightarrow{z}\cG$ with object maps $\alpha'$ and $\alpha$. Assume that $\cG$ is regular, that $M$ is connected, and that $\cG$ and $\cG'$ have connected source fibres. Then $\alpha$ and $\alpha'$ have connected fibres. The bibundle metric $\eta_P$ makes both maps Riemannian submersions, so $\cG'$ is regular by Corollary \ref{cor:theta-omega-morita}.

If, in addition, $M$ is compact, we obtain the leaf closure submersions and the $O(q)$-equivariant diffeomorphisms of their bases by Corollary \ref{cor:regular-morita-compact-bases}\textup{(1)}. Combining this with Lemma \ref{lem:morita-basic-algebroid-comparison} and Proposition \ref{prop:morita-isotropy-upgrade}, we obtain the following corollary.

\begin{corollary}\label{cor:morita-bibundle-source-connected}
With the notation and assumptions above, $\cG'$ is regular. The leaf closures of the lifted foliations $\widetilde{\mathcal O}$, $\widetilde{\mathcal O}'$, and $\widetilde{\mathcal O}_P$ are the fibres of smooth surjective submersions $\kappa$, $\kappa'$, and $\kappa_P$ onto Hausdorff manifolds $B$, $B'$, and $B_P$, respectively. The induced maps $\rho$ and $\rho'$ are $O(q)$-equivariant diffeomorphisms $B\xleftarrow{\ \rho\ }B_P\xrightarrow{\ \rho'\ }B'$.

The corresponding basic Lie algebroids $A$, $A'$, and $A_P$ are related by the induced Lie algebroid isomorphisms $\Psi:\rho^*A\xrightarrow{\ \cong\ }A_P$, $\Psi':(\rho')^*A'\xrightarrow{\ \cong\ }A_P$. These restrict to isomorphisms of the isotropy Lie algebra bundles and satisfy the compatibility identities for the sections $\omega^{\kappa}_{\mathrm{MC}}$, $\omega^{\kappa'}_{\mathrm{MC}}$, and $\omega^{\kappa_P}_{\mathrm{MC}}$ in Proposition \ref{prop:morita-isotropy-upgrade}.\qed
\end{corollary}

\end{document}